\documentclass[12pt]{amsart}
\usepackage{preamble}

\title{Positive density for consecutive runs of sums of two squares}
\author{Noam Kimmel}
\address{N. Kimmel: Raymond and Beverly Sackler School of Mathematical Sciences, Tel Aviv University, Tel Aviv 69978, Israel.}
\email{\href{mailto:noamkimmel@mail.tau.ac.il}{noamkimmel@mail.tau.ac.il}}
\author{Vivian Kuperberg}
\address{V. Kuperberg: Departement Mathematik, ETH Z\"{u}rich, R\"{a}mistrasse 101, 8092 Z\"{u}rich, Switzerland.}
\email{\href{mailto:viviankuperberg@math.ethz.ch}{vivian.kuperberg@math.ethz.ch}}

\thanks{N.K. was supported by the European Research Council (ERC) under the European Union's  Horizon 2020 research and innovation program  (Grant agreement No. 786758). V.K. was supported by the NSF Mathematical Sciences Research Program through the grant DMS-2202128. The authors would like to thank the anonymous referee for helpful comments.
}

\makeindex

\begin{document}
\begin{abstract}
We study the distribution of consecutive sums of two squares in arithmetic progressions. We show that for any odd squarefree modulus $q$, any two reduced congruence classes $a_1$ and $a_2$ mod $q$, and any $r_1,r_2 \ge 1$, a positive density of sums of two squares begin a chain of $r_1$ consecutive sums of two squares, all of which are $a_1$ mod $q$, followed immediately by a chain of $r_2$ consecutive sums of two squares, all of which are $a_2$ mod $q$. This is an analog of the result of Maynard for the sequence of primes, showing that for any reduced congruence class $a$ mod $q$ and for any $r \ge 1$, a positive density of primes begin a sequence of $r$ consecutive primes, all of which are $a$ mod $q$.
\end{abstract}

\maketitle
\setcounter{tocdepth}{2}
\tableofcontents

\section{Introduction}
The sequence of primes is known, by the prime number theorem in arithmetic progressions, to be equidistributed among reduced congruence classes to any modulus $q$. To be precise, for any modulus $q$ and for any reduced congruence class $a \mod q$, let  $\pi(x;q,a)$ denote the number of primes $p \le x$ with $p \equiv a \mod q$ and let $\pi(x)$ denote the number of primes $p\le x$. Then
\[\pi(x;q,a) = \frac{\pi(x)}{\phi(q)}(1 + o(1)).\]

Much less is known about analogous questions for strings of consecutive primes. Let $p_n$ denote the sequence of primes in increasing order. For any $M \ge 1$, for a fixed modulus $q$ and any $M$-tuple $\mathbf a = [a_1, \dots, a_M]$ of reduced residue classes mod $q$, let $\pi(x;q,\mathbf a)$ denote the number of strings of consecutive primes matching the residue classes of $\mathbf a$. That is, define
\begin{equation*}
\pi(x;q,\mathbf a) := \#\{p_n \le x: p_{n+i-1} \equiv a_i \pmod q \qquad \forall 1 \le i \le M\}.
\end{equation*}
Any randomness-based model of the primes would suggest that $M$-tuples of consecutive primes equidistribute among the possibilities for $\mathbf a$, as is the case when $M = 1$. That is, one would expect that $\pi(x;q,\mathbf a) \sim \frac{\pi(x)}{\phi(q)^M}$ as $x \to \infty$. Lemke Oliver and Soundararajan \cite{MR3624386-lemke-oliver-sound} provide a heuristic argument based on the Hardy--Littlewood $k$-tuples conjectures for estimating $\pi(x;q,\mathbf a)$ which agrees with this expectation (although it also predicts large second-order terms creating biases among the patterns).

However, little is known about $\pi(x;q,\mathbf a)$ when $M \ge 2$. In most cases, it is not even known that $\pi(x;q,\mathbf a)$ tends to infinity as $x \to \infty$, i.e. it is not known that $\mathbf a$ occurs infinitely often as a consecutive pattern in the sequence of primes mod $q$. If $\phi(q) = 2$ and $a_1 \ne a_2 \mod q$ are distinct reduced congruence classes, then $\pi(x;q,[a_1,a_2])$ and $\pi(x;q,[a_2,a_1])$ must each tend to infinity as an immediate consequence of Dirichlet's theorem for primes in arithmetic progressions; Knapowski and Tur\'an \cite{MR0466043-knapowski-turan} observed that if $\phi(q) = 2$, all four patterns of length $2$ occur infinitely often. 

As for arbitrary $q$, Shiu \cite{MR1760689-Shiu} used the Maier matrix method to prove that for any constant tuple $\mathbf a$ of any length, $\pi(x;q,\mathbf a)$ tends to infinity as $x \to \infty$. That is, for any fixed reduced residue class $a$ mod $q$, there are infinitely many arbitrarily long strings of consecutive primes, all of which are congruent to $a$ mod $q$. This result was rederived by Banks, Freiberg, and Turnage-Butterbaugh \cite{MR3316460-banks-freiberg-turnage-butterbaugh} using new developments in sieve theory. Maynard \cite{MR3530450-Maynard-dense-clusters} showed further that a positive density of primes begin strings of $M$ consecutive primes, all of which are congruent to $a$ mod $q$; that is, that $\pi(x;q,\mathbf a) \gg \pi(x)$ whenever $\mathbf a$ is a constant pattern. 

It is not currently known that $\pi(x;q,\mathbf a)$ tends to infinity for any other case, leading to the question of what more can be proven for other arithmetic sequences. In previous work \cite{kimmel-kuperberg-tuples-infinitude}, the authors considered the sequence of integer sums of two squares. Let $\mathbf E$\index{Main problem setup! $\mathbf E$} denote the set of sums of two squares and let $E_n$ denote the increasing sequence of sums of two squares, so that
\[\mathbf E = \{a^2 + b^2 : a,b \in \Z\} = \{E_n:n \in \N\}.\]
Let $N(x)$\index{Main problem setup! $N(x)$} denote the number of sums of two squares less than $x$. A number $n$ is in $\mathbf E$ if and only if every prime congruent to $3$ mod $4$ divides $n$ to an even power; that is, if $n$ factors as $n = \prod_p p^{e_p}$ then $e_p$ is even whenever $p \equiv 3 \mod 4$. For a modulus $q = \prod_p p^{e_p}$ and a congruence class $a$ mod $q$, write $(a,q) = \prod_p p^{f_p}$, where $f_p \le e_p$ for all $p$. There are infinitely many $n \in \mathbf E$ congruent to $a$ mod $q$ if and only if the following two conditions hold:
\begin{itemize}
    \item for any prime $p \equiv 3 \mod 4$, $f_p$ is either even or $f_p = e_p$, and 
    \item if $e_2 - f_2 \ge 2$, then $\frac{a}{2^{f_2}} \not\equiv 3 \mod 4$.
\end{itemize}
We will call a congruence class $a$ mod $q$ \emph{$\mathbf E$-admissible} if it satisfies these conditions, i.e. if there exists a solution to $x^2 + y^2 \equiv a \bmod q$.
For a modulus $q$, an integer $M \ge 1$, and an $M$-tuple $\mathbf a = [a_1, \dots, a_M]$ of $\mathbf E$-admissible residue classes mod $q$, let\index{Main problem setup! $N(x;q,\mathbf a)$}
\begin{equation*}
N(x;q,\mathbf a) := \#\{E_n \le x : E_{n+i-1} \equiv a_i \mod q \forall 1 \le i \le M\}.
\end{equation*}
Just as in the prime case, one expects $N(x;q,\mathbf a)$ to tend to infinity for any tuple of $\mathbf E$-admissible residue classes, and in fact one expects $N(x;q,\mathbf a) \gg N(x)$. In other words, one expects $N(x;q,\mathbf a)$ to represent a positive proportion of sums of two squares. When the modulus $q \equiv 1 \mod 4$ is a prime, David, Devin, Nam, and Schlitt \cite{MR4498471-david-devin-nam-schlitt} develop heuristics for second-order terms in the asymptotics of $N(x;q,\mathbf a)$ analagously to \cite{MR3624386-lemke-oliver-sound}. Their heuristics are based on the analog of the Hardy--Littlewood $k$-tuples conjecture in the setting of sums of two squares, which was developed in \cite{MR3733767-freiberg-kurlberg-rosenzweig}.
For $\mathbf a$ of length 1, these second-order terms are reminiscent of Chebyshev's bias, and were considered by Gorodetsky in \cite{MR4578006}.

The authors \cite{kimmel-kuperberg-tuples-infinitude} proved that for any modulus $q$, for any $3$-tuple of $\mathbf E$-admissible residue classes $[a_1,a_2,a_3]$, 
\begin{equation*}
\lim_{x \to \infty} N(x;q,[a_1,a_2,a_3]) \to \infty.
\end{equation*}
They also showed that for any odd, squarefree modulus $q$, for any residues $a_1$ and $a_2$ with $(a_i,q) = 1$, for any tuple of the form $[a_1,\dots,a_1,a_2,\dots,a_2]$, i.e. the concatenation of two constant tuples with values $a_1$ and $a_2$,
\begin{equation}\label{eq:twice-constant-tuples-go-to-infinity}
\lim_{x\to\infty} N(x;q,[a_1, \dots, a_1, a_2, \dots, a_2]) \to \infty.
\end{equation}
Note that this result does not extend to all $\mathbf E$-admissible residue classes $a_1$ and $a_2$.

In this paper, we strengthen \eqref{eq:twice-constant-tuples-go-to-infinity} by proving the following theorem.
\begin{theorem}\label{thm:main-theorem-on-density}
Let $q \ge 1$\index{Main problem setup! $q$} be a squarefree odd modulus and let $\Tilde{a_1}$ and $\Tilde{a_2}$\index{Main problem setup! $\Tilde{a_1},\Tilde{a_2}$} be reduced residue classes modulo $q$. 
Let $M \ge 1$\index{Main problem setup! $M, M_1$}, and let $\mathbf a = [a_1, \dots, a_M]$ be a tuple of residue classes such that for some $1 \le M_1 \le M$, $a_i = \Tilde{a_1}$\index{Main problem setup! $a_i$} whenever $i \le M_1$ and $a_i = \Tilde{a_2}$ whenever $i > M_1$. Then $N(x;q,\mathbf a) \gg N(x)$.
\end{theorem}
That is, any concatenation of two constant tuples appears with positive density among consecutive increasing sums of two squares modulo $q$. 
\begin{remark}
    Again, this result does not extend to all $\mathbf E$-admissible residue classes; $\Tilde{a_1}$ and $\Tilde{a_2}$ must be relatively prime to $q$. For squarefree odd $q$, in fact, all residue classes modulo $q$ are $\mathbf E$-admissible. For fixed squarefree odd $q$, and for $\Tilde{a_1}, \Tilde{a_2}$ modulo $q$ such that if $p|(\Tilde{a_i},q)$ then $p \equiv 1 \mod 4$, we expect our proof to apply with only minor adjustments in the computations of the technical results. We also expect that Theorem \ref{thm:main-theorem-on-density} extends with essentially no new ideas to the case where $q$ is not squarefree, if substantially more care is taken on the background lemmas on evaluating sums of two squares in Section \ref{subsec:sieve-auxiliary-lemmas-sums-of-two-squares}. Finally, our proof may apply essentially as written to the case where $(\Tilde{a_i},q)$ is divisible by primes that are $3$ mod $4$. However, these should appear with a smaller (yet still positive) density (for example, there are more sums of two squares that are $1$ mod $3$ than that are $0$ mod $3$), and it may be that understanding the case when $q$ is not squarefree is necessary for understanding this case.
\end{remark}

The proof of \Cref{thm:main-theorem-on-density} follows along the same basic idea as Maynard's result \cite{MR3530450-Maynard-dense-clusters} that constant tuples appear with positive density among consecutive increasing primes. This work in turn expands on the work of Maynard \cite{MR3272929-Maynard-small-gaps}, in which he shows that for any $m$, for any large enough $k$, and for any $\mathcal P$-admissible (that is, admissible in a precise sense with respect to the sequence of prime numbers) $k$-tuple of linear forms $\{L_1(n), \dots, L_k(n)\}$, there exist infinitely many $n$ such that at least $m$ of the $L_i(n)$ are simultaneously prime. In \cite{MR3530450-Maynard-dense-clusters}, for a tuple $\{L_1(n) = qn+a_1, \dots, L_k(n)=qn+a_k\}$ where each $L_i(n)$ is chosen such that $L_i(n) \equiv a \mod q$ for all $i$, Maynard shows that for infinitely many $n$, at least $m$ of the $L_i(n)$ are simultaneously prime \emph{and} the numbers in between the outputs of the $L_i(n)$ have small prime factors (and thus are not themselves prime). He then averages over many such tuples of $L_i(n)$ in order to obtain a lower bound of positive density.

In the setting of sums of two squares, stronger sieving results are available than those that are available in the prime case. McGrath \cite{MR4498475-McGrath} showed that for any $m$, for large enough $k$, for any $k$-tuple $\{h_1, \dots, h_k\}$ which is $\mathcal P$-admissible, and for any partition of $\{h_1, \dots, h_k\}$ into $m$ sub-tuples or ``bins,'' for infinitely many $n$, there exists an $h_i$ in each bin such that $n+h_i \in \mathbf E$. Banks, Freiberg, and Maynard \cite{MR3556490-Banks-Freiberg-Maynard} use a similar, but weaker, result in the case of primes to show that a positive proportion of real numbers are limit points of the sequence of normalized prime gaps, work which was refined in \cite{MR3855375-Pintz} and \cite{MR4143728-merikoski}.

In order to prove \Cref{thm:main-theorem-on-density}, we strengthen the sieve result of McGrath \cite{MR4498475-McGrath} in the same way that Maynard \cite{MR3530450-Maynard-dense-clusters} had expanded his previous work \cite{MR3272929-Maynard-small-gaps}. Our paper is organized as follows. In Section \ref{sec:statement-of-sieve-and-pf-of-main}, we will state our sieve theoretic results and use them to prove \Cref{thm:main-theorem-on-density}. In Section \ref{sec:proofs-of-sieve-props}, we will prove the sieve theoretic results. Our notation and setup is explained in Section \ref{subsec:density-general-setup}, with an additional explanation of more technical sieve notation in Section \ref{subsec:density-sieve-notation}. Finally, in Section \ref{sec:singular-series-estimates}, we evaluate certain averages of ``singular series'' constants that appear in the proof of Theorem \ref{thm:main-theorem-on-density}. 
\section{Statement of sieve results and proof of the main theorem} \label{sec:statement-of-sieve-and-pf-of-main}
\subsection{GPY sieve setup} \label{subsec:density-general-setup}

Our argument will follow the Goldston--Pintz--Y{\i}ld{\i}r{\i}m method for detecting primes in $\mathcal P$-admissible $k$-tuples, building off of work of Maynard \cite{MR3530450-Maynard-dense-clusters}, which uses a rather sophisticated version of this method, and of McGrath \cite{MR4498475-McGrath}, which develops a second-moment version of this method for sums of two squares. 

An \emph{$\mathcal P$-admissible} $k$-tuple of linear forms $(\ell_1(n), \dots, \ell_k(n))$ is one such that, for every prime $p$, there exists some $a \mod p$ with $\ell_i(a) \ne 0 \mod p$ for all $1\le i \le k$. Using the GPY method, Maynard \cite{MR3272929-Maynard-small-gaps} showed that for all integers $m \ge 2$, there exists large enough $k$ such that for any $\mathcal P$-admissible $k$-tuple of linear forms $(\ell_1(n), \dots, \ell_k(n))$, there are many integers $n \ge 1$ for which at least $m$ of the values $\ell_1(n), \dots, \ell_k(n)$ are simultaneously prime.

This statement follows from the construction of positive weights $w(n)$ such that for all $x$,
\begin{equation}\label{eq:original-maynard}
\sum_{x \le n < 2x} \left(\sum_{i=1}^k \mathbf 1_{\mathcal P}(\ell_i(n)) - m + 1\right) w(n) > 0,
\end{equation}
where $\mathbf 1_{\mathcal P}$ denotes the indicator function of the set $\mathcal P$ of prime numbers. The inequality \eqref{eq:original-maynard} implies that there exists a strictly positive summand, so that for some $n$ with $x \le n < 2x$,
\begin{equation*}
    \sum_{i=1}^k \mathbf 1_{\mathcal P}(\ell_i(n)) > m - 1,
\end{equation*}
and thus there are at least $m$ primes among the values of $\ell_i(n)$. 

We will require a version of this technique that is adapted in three different ways: first, we will detect sums of two squares instead of primes; second, we will need a ``second moment'' adaptation to detect slightly more delicate patterns among the sequence of sums of two squares; and third, we will exclude certain values of $n$ so that we will be able to average over many different $k$-tuples.

We begin by defining a certain weighted indicator function of sums of two squares. For any function $f$ (say, the indicator functions $\mathbf 1_{\mathcal P}$ or $\mathbf 1_{\mathbf E}$), in practice, applying the ``second moment'' adaptation requires an understanding of two-point correlations of the form
\begin{equation*}
    \sum_{x \le n < 2x} f(\ell_i(n))f(\ell_j(n)).
\end{equation*}
Estimates for two-point correlations of the standard indicator function of sums of two squares are not known, so we will instead make use of Hooley's $\rho$-function, which was first introduced in \cite{hooley1963number} and also used in this context by McGrath \cite{MR4498475-McGrath}. 

The $\rho$\index{Sums of two squares! $\rho(n)$} function is defined by 
\begin{equation}\label{eq:defn-of-rho}
    \rho(n) = r_2(n)t(n),
\end{equation}
where $r_2(n)$\index{Sums of two squares! $r_2(n)$} is the representation function of $n$, given by
\begin{align}
r_2(n) &:= \# \{(x,y) \in \mathbb Z^2 : x^2 + y^2 = n\}\label{eq:defn-of-r-2} \\
&= 4\sum_{\substack{d|n \\ d \text{ odd}}} (-1)^{\tfrac{d-1}{2}}.\nonumber
\end{align}
and\index{Sums of two squares! $t(n)$}\index{Sums of two squares! $v$}
\begin{equation}\label{eq:defn-of-t-and-v}
t(n) = t_{x,\theta_1}(n) := \sum_{\substack{a|n \\ a \le v \\ p|a \Rightarrow p \equiv 1 \mod 4}} \frac{\mu(a)}{g_2(a)}\left(1 - \frac{\log a}{\log v}\right), \quad (v = x^{\theta_1}).
\end{equation}
Here $\theta_1$\index{Sums of two squares! $\theta_1$} is a fixed small constant with $\theta_1 < 1/18$; for example, Hooley takes $\theta_1 = 1/20$. Moreover, $g_2$\index{Sums of two squares! $g_2(n)$} is the multiplicative function defined on primes via
\begin{equation}\label{eq:defn-of-g-2}
    g_2(p) = \begin{cases}
        2-\tfrac 1p &\text{ if } p \equiv 1 \pmod 4 \\
        \tfrac 1p &\text{ if } p \equiv 3 \pmod 4.
    \end{cases}
\end{equation}

Using the indicator function $\rho$, McGrath \cite{MR4498475-McGrath} uses a second-moment bound to prove the existence of sums of two squares in different ``bins'' of the same tuple. To state this precisely, fix $M,k \ge 1$, and let $K$ denote the product $K = Mk$. Let $q \ge 1$ be a fixed odd integer, and fix a tuple $\mathcal H^*$ of size $K$ such that $4|h_i$, $(h_i,q) = 1$, and for $\ell_i(n) = qn+h_i$, the tuple of linear forms $\{\ell_1(n), \dots, \ell_{K}(n)\}$ is $\mathcal P$-admissible (indeed, McGrath's result is phrased as requiring the tuple to be $\mathcal P$-admissible, \emph{not} $\mathbf E$-admissible). Suppose further that we have a fixed partition $\mathcal H = B_1 \sqcup \cdots \sqcup B_M$ where $|B_i| = k$ for all $i$. McGrath showed that there exists a real number $u \ge 1$ and a non-negative weight function $w(n)$ such that for all sufficiently large $x$,
\begin{equation}\label{eq:mcgrath-maynard}
\sum_{x < n \le 2x} \left[ u^2 - \sum_{i=1}^M \left(\sum_{\ell \in B_i}\rho(\ell(n)) - u\right)^2\right] w(n) > 0.
\end{equation}
The positivity of the left-hand side of \eqref{eq:mcgrath-maynard} implies that for all sufficiently large $x$, there exists some $n$ with $x < n \le 2x$ such that 
\begin{equation*}
\sum_{i = 1}^M \left(\sum_{\ell \in B_i} \rho(\ell(n)) - u\right)^2 < u^2.
\end{equation*}
If for any bin $B_i$, there is \emph{no} $\ell \in B_i$ with $\ell(n) \in \mathbf E$, then $\sum_{\ell \in B_i} \rho(\ell(n)) = 0$, and thus 
\begin{equation*}
u^2 \le \left(\sum_{\ell \in B_i} \rho(\ell(n)) - u\right)^2 \le \sum_{i = 1}^M \left(\sum_{\ell \in B_i} \rho(\ell(n)) - u\right)^2 < u^2,
\end{equation*}
a contradiction. Thus in particular the inequality \eqref{eq:mcgrath-maynard} implies that for all sufficiently large $x$, there exists an $n$ with $x < n \le 2x$ and such that for every bin $B_i$, there exists an $\ell \in B_i$ with $\ell(n) = qn+h \in \mathbf E$.

Our aim is to combine this second-moment version of the GPY sieve setup
with the goal of excluding certain values of $n$ for each tuple $\mathcal H^*$ in order to be able to average over many different tuples. 
In particular, we will choose weights $w(n)$ such that for any $n$ making a positive contribution to the left-hand side of \eqref{eq:mcgrath-maynard}, $\ell_i(n)$ does not have any `small' prime factors $p \equiv 3 \mod 4$ for \emph{any} of the $\ell_i$, and for any $b \le \eta\sqrt{\log x}$ which is \emph{not} in $\mathcal H$ (i.e. $b\neq h_i$), the integer $qn + b$ is divisible exactly once by some `small' prime $p \equiv 3 \mod 4$. 
These may seem like artificial constraints to place on the values $n$, but in fact $n$ that do not satisfy these constraints are exceptionally rare; intuitively, although it cannot be proven explicitly, the weights $w(n)$ place emphasis on those $n$ where \emph{all} $\ell_i(n) \in \mathbf E$ (or close to it), and values $qn+b$ that are outside of the tuple are unlikely to be sums of two squares. 
In \cite{MR3530450-Maynard-dense-clusters}, Maynard takes advantage of a similar device to average over different subsets $\mathcal H^*$, which allows him to prove a lower bound of positive density on the tuples he is counting.

Our precise setup is as follows. As in the setup of \Cref{thm:main-theorem-on-density}, we let $q$ be a fixed odd squarefree modulus, and we also fix the parameters $M$ 
and two congruence classes $\Tilde{a_1}$ and $\Tilde{a_2}$ modulo $q$, as well as $M_1$ with $1 \le M_1 \le M$. We will consider tuples of length $K$, where $K = kM$, split into bins of size $k$\index{Averaging argument! $K,k$}. We define integers $a_1, \dots, a_K$ as follows. For $i$ with $1 \le i \le M_1k$, we let $a_i$ be the smallest positive integer with $a_i \equiv \Tilde{a_1} \mod q$ and $a_i \equiv 1 \mod 4$, whereas for $i$ with $M_1k+1 \le i \le K$, we let $a_i$ be the second-smallest positive integer with $a_i \equiv \Tilde{a_2} \mod q$ and $a_i \equiv 1 \mod 4$ (that is, $a_i-4q$ is the smallest such positive integer). The values of $a_i$ for $M_1k+1 \le i \le K$ are shifted by $q$ to ensure that $a_{i_1} < a_{i_2}$ whenever $1 \le i_1 \le M_1k < i_2 \le K$. Note that there are only two distinct values for the $a_i$, but for ease of notation we define $K$ values $a_i$, even though these values are repetitive.

Then, for any tuple of integers $\mathbf b = (b_1, \dots, b_K)$\index{Averaging argument! $\mathbf b$} with $b_i \equiv 3 \mod 4$\index{Averaging argument! $b_i$} and $3\le b_i \le \frac{\eta}{q} \sqrt{\log x}$ for all $i$, we will define the $K$-tuple $\mathcal L = \mathcal L(\mathbf b) = \{\ell_i(n)\}_{i=1}^K$\index{Averaging argument! $\mathcal L, \mathcal L(\mathbf b)$} of linear forms given by\index{Averaging argument! $\ell_i(n)$}
\begin{equation}\label{eq:defn-of-ell-i-of-n}
\ell_i(n) := qn + a_i + qb_i.
\end{equation}
Here $\eta$ is a positive constant to be set later.
Note that the constraints on $a_i$ and $b_i$ modulo $4$ imply that whenever $n \equiv 1 \mod 4$, we also have $\ell_i(n) \equiv 1 \mod 4$.

We will ultimately average over many different choices of $\mathbf b$. Our average will be taken over $\mathbf b$ lying in a slightly restricted set of tuples $\mathcal B$\index{Averaging argument! $\mathcal B$}, where we define
\begin{equation}\label{eq:defn-of-mathcal-B}
\mathcal B := \left\{\mathbf b = (b_1, \dots, b_K) \middle\vert b_i \equiv 3 \mod 4, 
\begin{array}{l
}b_1 = 3\\ 
3\le b_i \le \frac{\eta}{2q}\sqrt{\log x} \: \;\forall \: 2 \le i \le M_1k \\ 
\frac{\eta}{2q} \sqrt{\log x} < b_i \le \frac{\eta}{q} \sqrt{\log x} \: \;\forall \: M_1k < i \le K
\end{array}
\right\}. 
\end{equation}
The key consequence of this definition (along with the definition of the $a_i$'s) is that for any $n$, $\ell_{i_1}(n) < \ell_{i_2}(n)$ whenever $1 \le i_1 \le M_1k$ and $M_1k+1 \le i_2 \le K$.

As described above, we will write $\mathcal L = B_1 \sqcup \cdots \sqcup B_M$, where\index{Averaging argument! $B_i$} 
\begin{equation}\label{eq:defn-of-the-bins}
    B_i := \{\ell_{(i-1)k+1}(n), \dots, \ell_{ik}(n)\}.
\end{equation}
The $B_i$, which we refer to as \emph{bins}, partition the tuple $\mathcal L$ into $M$ bins, each of size $k$.

For certain real numbers $\xi,\eta > 0$\index{Averaging argument! $\xi$} (to be fixed later), a certain real number $u$\index{Averaging argument! $u$}, and a nonnegative weight function $w_n(\mathcal L)$, we consider a sum of the shape
\begin{multline}\label{eq:our-maynard-style-sum}
\sum_{x < n \le 2x} 
\Biggl[ u^2 - 
\sum_{i=1}^M \left(\sum_{\ell \in B_i} \rho(\ell(n)) - u\right)^2
- \sum_{j=1}^{K} \sum_{\substack{p < x^\xi \\ p \equiv 3 \mod 4 \\ p|\ell_j(n)}} u^2 
\\ 
- \sum_{\substack{b \le \eta \sqrt{\log x} \\ \ell^{(b)} \not\in \mathcal L}} \mathbf 1_{S(\xi)}(\ell^{(b)}(n))u^2 
\Biggr]
w_n(\mathcal L),
\end{multline}
where $S(\xi)$\index{Averaging argument! $S(\xi)$} is the set of integers such that for all primes $p < x^\xi$ which satisfy $p \equiv 3 \mod 4$ either $p\nmid n$ or $p^2|n$.
We write $\ell^{(b)}(n) := qn + b$\index{Averaging argument! $\ell^{(b)}(n)$}, so that the final sum in \eqref{eq:our-maynard-style-sum} is a sum over $b \le \eta \sqrt{\log x}$ such that $\ell^{(b)} \not\in \mathcal L$.
A choice of weights $w_n(\mathcal L)$ such that \eqref{eq:our-maynard-style-sum} is positive implies that for some $n$ with $x<n \le 2x$,
\begin{equation*}
u^2 - \sum_{i=1}^M \left(\sum_{\ell \in B_i} \rho(\ell(n)) - u\right)^2 - \sum_{j=1}^{K} \sum_{\substack{p < x^\xi \\ p \equiv 3 \mod 4 \\ p |\ell_j(n)}} u^2 - \sum_{\substack{b \le \eta \sqrt{\log x} \\ \ell^{(b)} \not\in \mathcal L}} \mathbf 1_{S(\xi)}(\ell^{(b)}(n))u^2 > 0,
\end{equation*}
which in turn implies that:
\begin{itemize}
    \item for each $i$, there exists a linear form $\ell \in B_i$ with $\rho(\ell(n)) \ne 0$ and thus $\ell(n) \in \mathbf E$;
    \item for each $j$, with $1 \le j \le K$, $\ell_j(n)$ is not divisible by any prime $p < x^\xi$ with $p \equiv 3 \mod 4$; and
    \item for each $b \le \eta \sqrt{\log x}$ with $\ell^{(b)}$ \emph{not} in $\mathcal L$,  we have $\ell^{(b)}(n) \not\in S(\xi)$, so there exists some prime $p < x^\xi$ with $p \equiv 3 \mod 4$ such that $p\|\ell^{(b)}(n)$. 
\end{itemize}

In order to take advantage of this positivity argument, we will need to evaluate the sums over $n$ appearing in \eqref{eq:our-maynard-style-sum}. These evaluations are accomplished in \Cref{thm:S-i-sums-estimate}, which we state in the next section before completing the proof of \Cref{thm:main-theorem-on-density}.

\subsection{Conventions and notation}\label{sec:index-chart}

Before stating our main sieve theorem and presenting the proof of Theorem \ref{thm:main-theorem-on-density}, we first fix some notation and conventions that we will use throughout the paper. An index for key quantities appears after the references.

All asymptotic notation, such as $O(\cdot), o(\cdot), \ll,$ and $\gg,$ should be interpreted as referring to the limit $x \to \infty$. 
We will use Vinogradov $f \ll g$ to mean $f = O(g)$, that is, $|f| \le Cg$ for some absolute constant $C$. 
Any constants are absolute unless otherwise noted. 
For all sums or products over a variable $p$ (or $p'$), the variable $p$ will be assumed to lie in the prime numbers; all other sums and products will be assumed to be taken over variables lying in the natural numbers $\mathbb N_{\ge 1}$ unless otherwise specified.

Recall that the squarefree odd modulus $q$ is fixed throughout. We denote $q = q_1q_3$\index{Main problem setup! $q_1,q_3$}, where $q_1$ is a product of primes that are $1$ mod $4$ and $q_3$ is a product of primes that are $3$ mod $4$.

Let $\theta_2>0$\index{Sieve parameters! $\theta_2$} be a fixed positive real number such that $0 < \theta_1+\theta_2 < 1/18$, 
and let $R = x^{\theta_2/2}$\index{Sieve parameters! $R$}. Letting $D_0 = \eta\sqrt{\log x}$\index{Sieve parameters! $D_0$} for a constant $\eta > 0$\index{Averaging argument! $\eta$} to be fixed later, we define\index{Sieve parameters! $W$}
\begin{equation}\label{eq:defn-of-W}
    W = \prod_{\substack{p \le D_0 \\ p \equiv 3 \mod 4 \\ p \nmid q}} p.
\end{equation}
Note that $q_3W$ is the product of all primes $p\le D_0$ which are $3$ mod $4$. This definition of $W$ differs from that of McGrath \cite{MR4498475-McGrath} because, while the value of $D_0$ is much larger than that used by McGrath, it is not divisible by any primes $p \equiv 1 \mod 4$.

We denote by $A$\index{Sums of two squares! $A$} the \emph{Landau--Ramanujan constant}, given by
\begin{equation}\label{eq:defn-of-A}
    A = \frac 1{\sqrt 2} \prod_{p \equiv 3 \mod 4} \left(1-\frac 1{p^2}\right)^{-\tfrac 12} = \frac{\pi}{4}\prod_{p \equiv 1 \mod 4} \left(1-\frac 1{p^2}\right)^{1/2}.
\end{equation}
We also make use of a normalization constant $B$\index{Sieve parameters! $B$}, defined as
\begin{equation}\label{eq:defn-of-B}
B = \frac{A}{\Gamma(1/2)\sqrt{L(1,\chi_4)}} \cdot \frac{\phi(q_3W)(\log R)^{1/2}}{q_3W} = 
\frac{2A}{\pi}\frac{\phi(q_3W)(\log R)^{1/2}}{q_3W}.
\end{equation}
Here $\chi_4$\index{Sums of two squares! $\chi_4$} denotes the non-trivial Dirichlet character modulo $4$. Finally, we will denote by $V$\index{Sums of two squares! $V$} the constant given by
\begin{equation}\label{eq:defn-of-V}
    V = \prod_{p \equiv 1 \mod 4} \left(1 + \frac 1{(2p-1)^2}\right) \approx 1.016.
\end{equation}

For $K$-tuples in $\mathbb N^K$, we will use the notation that a boldface letter such as $\mathbf d$ represents a tuple $\mathbf d = (d_1, \dots, d_K)$, whereas a non-boldface $d$ represents the product of the entries $\prod_{i=1}^K d_i$. Given tuples $\mathbf d$ and $\mathbf e$, we will let $[\mathbf d, \mathbf e]$ denote the product of the least common multiples $\prod_{i=1}^K [d_i,e_i]$, let $(\mathbf d, \mathbf e)$ denote the product of the greatest common divisors $\prod_{i=1}^K (d_i,e_i)$, and let $\mathbf d|\mathbf e$ denote the $K$ conditions that $d_i|e_i$ for $1 \le i \le K$.

\subsection{Statement of the main sieve theorem}

We are now ready to state our main sieving theorem, which we will use in the next section to deduce Theorem \ref{thm:main-theorem-on-density}.

\begin{theorem}\label{thm:S-i-sums-estimate}
Fix $\mathbf b \in \mathcal B$ and let $\mathcal L(\mathbf b)$ be the fixed $K$-tuple of linear forms $\{\ell_i(n)\}_{i = 1}^{K}$ given by \eqref{eq:defn-of-ell-i-of-n}. Let $\nu_0$ be a fixed residue class modulo $W$ such that for all $\ell \in \mathcal L$, $(\ell(\nu_0),W) = 1$. 
Then there exists a choice of nonnegative weights $w_n(\mathcal L) \ge 0$, as well as a constant $L_K(F)$, such that 
\begin{equation} \label{eq:w_n-upper-bound-for-average}
w_n(\mathcal L) \ll \left(\frac{\log R}{\log D_0}\right)^{K} \prod_{i=1}^K \prod_{\substack{p|\ell_i(n) \\ p \equiv 3 \mod 4}} 4
\end{equation}
and
the following estimates hold:
\begin{enumerate}[a)]
\item Let $S_1(\nu_0)$ be the sum defined by
\begin{equation*}
S_1(\nu_0) := \sum_{\substack{x<n\le 2x \\ n \equiv 1 \mod 4 \\ n \equiv \nu_0 \mod W}} w_n(\mathcal L).
\end{equation*}
Then
\begin{equation} \label{eq:thm:S-i-sums:S-1}
S_1(\nu_0) = (1+o(1)) \frac{B^Kx}{4W}L_K(F).
\end{equation}
\item Let $S_2^{(m)}(\nu_0)$ be the sum defined by
\begin{equation*}
S_2^{(m)}(\nu_0) := \sum_{\substack{x<n\le 2x \\ n \equiv 1 \mod 4 \\ n \equiv \nu_0 \mod W}} \rho(\ell_m(n)) w_n(\mathcal L).
\end{equation*}
Then
\begin{equation}\label{eq:thm:S-i-sums:S-2}
S_2^{(m)}(\nu_0) = (1+o(1)) \frac{4\pi \sqrt{\tfrac{\log R}{\log v}}B^Kx}{(\pi+2)\sqrt{K} W}L_K(F).
\end{equation}
\item Let $S_3^{(m_1,m_2)}(\nu_0)$ be the sum defined by
\begin{equation*}
S_3^{(m_1,m_2)}(\nu_0) := \sum_{\substack{x< n \le 2x \\ n \equiv 1 \mod 4 \\ n \equiv \nu_0 \mod W}} \rho(\ell_{m_1}(n))\rho(\ell_{m_2}(n))w_n(\mathcal L).
\end{equation*}
Then 
\begin{equation}\label{eq:thm:S-i-sums:S-3}
S_3^{(m_1,m_2)}(\nu_0) \le (1 + o(1)) \frac{64\pi^2\tfrac{\log R}{\log v} B^Kx}{(\pi + 2)^2 K W}V L_K(F),
\end{equation}
where $V$ is the constant defined in \eqref{eq:defn-of-V}.
\item Let $S_4^{(m)}(\nu_0)$ be the sum defined by
\begin{equation*}
S_4^{(m)}(\nu_0) := \sum_{\substack{x < n \le 2x \\ n \equiv 1 \mod 4 \\ n\equiv \nu_0 \bmod W}} \rho(\ell_m(n))^2w_n(\mathcal L).
\end{equation*}
Then
\begin{equation}\label{eq:thm:S-i-sums:S-4}
S_4^{(m)}(\nu_0) = (1+o(1))\frac{8\pi \sqrt{\frac{\log R}{\log v}} \left(\frac{\log x}{\log v} + 1 \right)B^{K}x}
{(\pi + 2)\sqrt{K}W} 
\prod_{\substack{p \equiv 1 \bmod 4 \\  p\nmid q_1}}\left( 1 + \frac{1}{(2p-1)^2}\right)
L_K(F).
\end{equation}
\item Assume that $\xi > 0$ satisfies $\xi < \frac 1K$, and let $S_5^{(m)}(\nu_0)$ be the sum defined by
\begin{equation*}
S_5^{(m)}(\nu_0) := \sum_{\substack{x < n \le 2x\\n \equiv 1 \mod 4\\ n \equiv \nu_0 \mod W}}\sum_{\substack{p < x^\xi \\ p \equiv 3 \mod 4 \\ p|\ell_m(n)}} w_n(\mathcal L).
\end{equation*}
Then $S_5^{(m)}(\nu_0)$ satisfies
\begin{equation}\label{eq:thm:S-i-sums:S-5}
S_5^{(m)}(\nu_0) \ll \frac{K^2\xi^2}{\theta_2^2} \frac{B^Kx}{W} L_K(F).
\end{equation}
\item Let $\nu_1$ be a congruence class modulo $q_3^2W^2$ such that $(\ell(\nu_1),q_3^2W^2)$ is a square for all $\ell \in \mathcal L$. Fix $3 < b \le \eta \sqrt{\log x}$ and a constant $\xi$ with $0 < \xi < 1/4$. Let $S_6^{(b)}(\nu_1)$ be defined by 
\begin{equation*}
S_6^{(b)}(\nu_1) := \sum_{\substack{x < n \le 2x \\ n \equiv 1 \mod 4 \\ n \equiv \nu_1 \mod q_3^2W^2}} \mathbf 1_{S(\xi)}(\ell^{(b)}(n)) w_n(\mathcal L).
\end{equation*}
Then
\begin{equation}\label{eq:thm:S-i-sums:S-6}
S_6^{(b)}(\nu_1) \ll_K \frac{x}{4q_3^2W^2}\xi^{-1/2} \left(\frac{\theta_2}{2}\right)^{-1/2} \left(\frac{\log R}{\log D_0}\right)^{\frac{K-1}{2}} L_K(F).
\end{equation}
\end{enumerate}
\end{theorem}
This theorem is key in all of our computations, and will be proven in Section \ref{sec:proofs-of-sieve-props}. In the remainder of this section, we derive our main result as a consequence of Theorem \ref{thm:S-i-sums-estimate}.

\subsection{Proof of \texorpdfstring{\Cref{thm:main-theorem-on-density}}{Theorem 1}}

The goal of this subsection is to prove \Cref{thm:main-theorem-on-density} as a consequence of \Cref{thm:S-i-sums-estimate} and the evaluations of the linear functionals therein.

We will consider an average of $\mathbf E$-admissible tuples $\mathcal L = \mathcal L(\mathbf b) = \{\ell_i(n)\}_{i=1}^K$, given by \eqref{eq:defn-of-ell-i-of-n}, over the set $\mathcal B$ (defined in \eqref{eq:defn-of-mathcal-B}) of $K$-tuples $\mathbf b$. We consider the sum
\begin{multline}\label{eq:averaged-sum-S-with-everything}
S = 
\sum_{\substack{\mathbf b \in \mathcal B \\ \mathcal L = \mathcal L(\mathbf b) \text{ $\mathbf E$-admissible}}} 
\sum_{\substack{\nu_1 \mod q_3^2W^2 \\(\ell(\nu_1),W) = 1 \forall \ell \in \mathcal L}}\;
\sum_{\substack{x < n \le 2x \\ n \equiv \nu_1 \mod q_3^2W^2}}
\Biggl[ u^2 - \sum_{i=1}^M \left(\sum_{\ell \in B_i} \rho(\ell(n)) - u\right)^2 
\\
- \sum_{i=1}^{K} \sum_{\substack{p < x^\xi \\ p \equiv 3 \mod 4 \\ p|\ell_i(n)}} u^2 
- \sum_{\substack{b \le \eta \sqrt{\log x} \\ \ell^{(b)} \not\in \mathcal L \\ (\ell^{(b)}(\nu_1), q_3^2W^2) = \square}} \mathbf 1_{S(\xi)}(\ell^{(b)}(n))u^2  \Biggr]
w_n(\mathcal L).
\end{multline}
For technical reasons involving the final sum, we will initially sum over congruence classes modulo $q_3^2W^2$ instead of modulo $W$. However, note that the condition that $(\ell(\nu_1),W) = 1$ is determined only by the congruence class of $\nu_1 \mod W$, so this is in some sense really a sum over congruence classes modulo $W$. 

Here $w_n(\mathcal L)$ are the weights given by \Cref{thm:S-i-sums-estimate} for the $\mathbf E$-admissible set $\mathcal L = \mathcal L(\mathbf b)$. 
For fixed $\mathcal L$, the term in the square parentheses in \eqref{eq:averaged-sum-S-with-everything} is positive only if the following conditions all hold:
\begin{enumerate}[(i)]
    \item for each $i$ with $1 \le i \le M$, there exists some $\ell \in B_i$ with $\rho(\ell(n)) \ne 0$, or equivalently with $\ell(n) \in \mathbf E$;
    \item for each $\ell \in \mathcal L$, $\ell(n)$ has no prime factors $p$ with $p < x^\xi$ and $p \equiv 3 \mod 4$; and
    \item for all other $\ell^{(b)} \not\in \mathcal L$ with $b \le \eta \sqrt{\log x}$, and $(\ell^{(b)}(n),q_3^2W^2)$ a square, $\ell^{(b)}(n)$ has a prime factor $p$ with $p < x^\xi$, $p \equiv 3 \mod 4$, and $p\|\ell^{(b)}(n)$.
\end{enumerate}
This has two crucial implications. One is that no $n$ can make a positive contribution from two different tuples $\mathcal L$, since if $n$ makes a positive contribution for any $\mathcal L$, then the values $\ell(n)$ are uniquely determined as the integers in $[qn, qn+\eta \sqrt{\log x}]$ which are:
\begin{enumerate}[(i)]
    \item congruent to $1 \mod 4$,
    \item congruent to $\Tilde{a_1} \mod q$ if they lie in $[qn+a_1,qn+a_1+(\eta/2)\sqrt{\log x}]$, or congruent to $\Tilde{a_2}\mod q$ if they lie in $[qn+a_K+(\eta/2)\sqrt{\log x},qn+a_K+\sqrt{\log x}]$, and
    \item not divisible to an odd power by any primes $p < x^\xi$ with $p \equiv 3 \mod 4$.
\end{enumerate}
The second observation is that if $n$ makes a positive contribution for a tuple $\mathcal L$, then since for all $\ell^{(b)} \not\in \mathcal L$ with $b \le \eta \sqrt{\log x}$, $\ell^{(b)}(n) = qn+b \not\in \mathbf E$, we have that the sums of two squares appearing in $\mathcal L$ (of which there is at least one in each bin) must be \emph{consecutive} sums of two squares.

Also, if $n$ makes a positive contribution, then none of the $\ell_i(n)$ can have any prime factors $p \equiv 3 \mod 4$ which are less than $x^{\xi}$, so each $\ell_i(n)$ can have at most $O(1/\xi)$ prime factors $p \equiv 3 \mod 4$. In particular, this implies by \eqref{eq:w_n-upper-bound-for-average} that
\begin{equation}\label{eq-upper_bound_wn}
    w_n(\mathcal L) \ll\left(\frac{\log R}{\log D_0}\right)^{K} \prod_{i=1}^K \prod_{\substack{p\mid \ell_i(n) \\ p \equiv 3 \mod 4}} 4 \ll \left(\frac{\log R}{\log D_0}\right)^{K} \mathrm{exp}(O(K/\xi)),
\end{equation}
for any pair $n$ and $\mathcal L$ making a positive contribution to \eqref{eq:averaged-sum-S-with-everything}.

We now evaluate the sum in \eqref{eq:averaged-sum-S-with-everything}. To begin with, we can swap the order of summation for the various different terms to get
\begin{align}\label{eq-S_using_Si}
\begin{split}
S= 
&\sum_{\substack{\mathbf b \in \mathcal B \\ \mathcal L = \mathcal L(\mathbf b) \text{ adm.}}} 
\Biggl[ 
\sum_{\substack{\nu_0 \mod W \\ (\ell(\nu_0),W) = 1 \forall \ell \in \mathcal L}} 
\biggl( 
u^2(1-M)  S_1(\nu_0) + 
2u \sum_{m = 1}^K  S_2^{(m)}(\nu_0) 
\\ &
- \sum_{i = 1}^M \sum_{\substack{\ell_{m_1},\ell_{m_2} \in B_i \\ m_1 \ne m_2}}  S_3^{(m_1,m_2)}(\nu_0) 
-  \sum_{m=1}^K S_4^{(m)}(\nu_0) 
-  u^2 \sum_{i=1}^K S_5^{(i)}(\nu_0)
\biggr)
\\ &
-  u^2  \sum_{\substack{\nu_1 \mod q_3^2W^2 \\ (\ell(\nu_1),W) = 1 \forall \ell \in \mathcal L}}  \sum_{\substack{b \le \eta \sqrt{\log x} \\ \ell^{(b)} \not\in \mathcal L \\ (\ell^{(b)}(\nu_1), q_3^2W^2) = \square}} S_6^{(b)}(\nu_1)\Biggr],
\end{split}
\end{align}
where the sums $S_1(\nu_0), S_2^{(m)}(\nu_0), S_3^{(m_1,m_2)}(\nu_0), S_4^{(m)}(\nu_0), S_5(\nu_0),$ and $S_6^{(b)}(\nu_1)$ are in the notation of \Cref{thm:S-i-sums-estimate}. 

We now wish to use our estimates from \Cref{thm:S-i-sums-estimate}.
For the sum $S_6^{(b)}(\nu_1)$ we will require a more careful analysis that takes the averaging over $\mathbf{b}, \nu_1, b$ into account.
Specifically, we require the following lemma, which is proven in \Cref{subsec-s6_average}.
\begin{lemma}\label{lem-average_s6}
With the notation above,
\global\long\def\Ssix{ \xi^{-1/2}
    \theta_2^{\frac{K}{2} - 1}
    L_K(F)
    \left(\frac{q_3}{\phi(q_3)}\right)^{K}
    \frac{\eta^K}{q_3^2q^{K-1}}
    \frac{\left(\log x\right)^{K - \frac{1}{2}}}
    {\left(\log D_0\right)^K} x}
\begin{multline*}
\sum_{\substack{\mathbf b \in \mathcal B \\ \mathcal L = \mathcal L(\mathbf b) \text{ adm.}}}    
\sum_{\substack{\nu_1 \mod q_3^2W^2 \\ (\ell(\nu_1),W) = 1 \forall \ell \in \mathcal L}}  
\sum_{\substack{b \le \eta \sqrt{\log x} \\ \ell^{(b)} \not\in \mathcal L \\ (\ell^{(b)}(\nu_1), q_3^2W^2) = \square}} S_6^{(b)}(\nu_1) \\
\ll_K \Ssix,
\end{multline*}
where the implied constant depends only on $K$.
\end{lemma}

Applying the estimates \eqref{eq:thm:S-i-sums:S-1}, \eqref{eq:thm:S-i-sums:S-2}, \eqref{eq:thm:S-i-sums:S-3}, \eqref{eq:thm:S-i-sums:S-4}, and \eqref{eq:thm:S-i-sums:S-5} from \Cref{thm:S-i-sums-estimate} to \eqref{eq-S_using_Si}, we get
\begin{align}
\begin{split}
S&\ge 
(1 + o(1))  \frac{B^K x}{W} L_K(F)
\sum_{\substack{\mathbf b \in \mathcal B \\ \mathcal L = \mathcal L(\mathbf b) \text{ adm.}}} 
\sum_{\substack{\nu_0 \mod W \\ (\ell(\nu_0),W) = 1 \\ \forall \ell \in \mathcal L}} 
\left[\rule{0cm}{1.2cm}\right.
u^2(1-M)  \frac{1}{4} + 
2u \sum_{m = 1}^K \frac{4\pi \sqrt{\tfrac{\log R}{\log v}}}{(\pi + 2)\sqrt{K} }
\\ &
- \sum_{i = 1}^M \sum_{\substack{\ell_{m_1},\ell_{m_2} \in B_i \\ m_1 \ne m_2}}  \frac{64\pi^2\left(\tfrac{\log R}{\log v}\right) }{(\pi+2)^2K}V 
-  \sum_{m=1}^K \frac{8\pi \sqrt{\frac{\log R}{\log v}} \left(\frac{\log x}{\log v} + 1 \right)}
{(\pi+2) \sqrt{K}  } 
\prod_{\substack{p \equiv 1 \bmod 4 \\  p\nmid q_1}}\left( 1 + \frac{1}{(2p-1)^2}\right) 
\\ &
-  u^2 \sum_{i=1}^K O\left(
\frac{ K^2\xi^2}{4} L_K(F)
\right)
\left.\rule{0cm}{1.2cm}\right]
- 
u^2 
\sum_{\substack{\mathbf b \in \mathcal B \\ \mathcal L = \mathcal L(\mathbf b) \text{ adm.}}} 
\sum_{\substack{\nu_1 \mod q_3^2W^2 \\ (\ell(\nu_1),W) = 1 \\ \forall \ell \in \mathcal L}}
\sum_{\substack{b \le \eta \sqrt{\log x} \\ \ell^{(b)} \not\in \mathcal L \\ (\ell^{(b)}(\nu_1), q_3^2W^2) = \square}} S_6^{(b)}(\nu_1).
\end{split}
\end{align}

We now use \Cref{lem-average_s6} to evaluate the last triple sum, and simplify using the facts that $\log R = \tfrac{\theta_2}{2} \log x$ and $\log v = \theta_1 \log x$, which gives
\begin{align*}
\begin{split}
S\ge &
(1 + o(1))  
\left(\sum_{\substack{\mathbf b \in \mathcal B \\ \mathcal L = \mathcal L(\mathbf b) \text{ adm.}}} 
\sum_{\substack{\nu_0 \mod W \\ (\ell(\nu_0),W) = 1 \forall \ell \in \mathcal L}} 1\right)
\frac{B^K x}{W}L_K(F)
\left[\rule{0cm}{1cm}\right.
 \frac{u^2(1-M) }{4} 
+\frac{8u\pi\sqrt{K}\sqrt{\theta_2/2\theta_1}}{\pi+2}
 \\ &
- 
M \frac{k(k+1)}{2}\frac{64\pi^2(\theta_2/2\theta_1)}{(\pi+2)^2K}V
-  \frac{8\pi\sqrt{K}\sqrt{\theta_2/2\theta_1} \left(\frac{1}{\theta_1} + 1 \right)}
{(\pi + 2)} 
\prod_{\substack{p \equiv 1 \bmod 4 \\  p\nmid q_1}}\left( 1 + \frac{1}{(2p-1)^2}\right)
\\ &
-  O\left(u^2 K^3\xi^2 
\right)
\left.\rule{0cm}{1cm}\right]
-  
O\left(u^2 \Ssix\right).
\end{split}
\end{align*}
We will make the change of variables
\begin{equation*}
    u = \frac{\pi}{\pi+2}\sqrt{\frac{\theta_2}{2\theta_1}} \Tilde{u},
\end{equation*}
so that the sum above simplifies to
\begin{align*}
\begin{split}
S\ge &
(1 + o(1))  
\left(\sum_{\substack{\mathbf b \in \mathcal B \\ \mathcal L = \mathcal L(\mathbf b) \text{ adm.}}} 
\sum_{\substack{\nu_0 \mod W \\ (\ell(\nu_0),W) = 1 \forall \ell \in \mathcal L}} 1
\right)
\frac{B^K x}{W}L_K(F)
\left[\rule{0cm}{1.2cm}\right.
\left(\frac{\pi}{\pi+2}\right)^2 \frac{\theta_2}{2\theta_1}
 \\ &
\left(\rule{0cm}{1cm}\right.
 \frac{\Tilde{u}^2(1-M)}{4} 
+ 8\sqrt{K}\Tilde{u}
- 32 V\left(\frac{K}{M} + 1\right)
\left.\rule{0cm}{1cm}\right)
-  O_{\theta_1,\theta_2}\left( \sqrt{K}\right)
-  O\left(
u^2 K^3\xi^2 
\right)
\left.\rule{0cm}{1.2cm}\right]
-  \\
&
O\left(u^2\Ssix\right).
\end{split}
\end{align*}
We then set $\Tilde{u} = \frac{16\sqrt{K}}{M-1}$ to maximize the expression above, so that (recalling that $K = Mk$)
\begin{align*}
\begin{split}
S\ge &(1 + o(1))
\left(
\sum_{\substack{\mathbf b \in \mathcal B \\ \mathcal L = \mathcal L(\mathbf b) \text{ adm.}}} 
\sum_{\substack{\nu_0 \mod W \\ (\ell(\nu_0),W) = 1 \forall \ell \in \mathcal L}} 1
\right)
  \frac{B^K x}{W}L_K(F)
 \\ &
\times\left[\rule{0cm}{1.2cm}\right.
\left(\frac{\pi}{\pi+2}\right)^2 \left(\frac{\theta_2}{2\theta_1}\right)\cdot 
32\left(k\frac{(2-V)M+V}{M-1}-V\right) - 
O_{\theta_1,\theta_2}\left( \sqrt{K}\right)
-  O\left(u^2K^3\xi^2
\right)
\left.\rule{0cm}{1.2cm}\right]
\\ &-
O\left(u^2\Ssix\right).
\end{split}
\end{align*}
Recall that $V \approx 1.016 < 2$, so
for a given $M$, we can pick $k$ large enough in terms of $M$, $\theta_1$, and $\theta_2$ so that the quantity
$$
\Delta = \left(\frac{\pi}{\pi+2}\right)^2 \left(\frac{\theta_2}{2\theta_1}\right)\cdot 
32\left(K\frac{(2-V)M+V}{M(M-1)}-V\right) - 
O_{\theta_1,\theta_2}\left( \sqrt{K}\right)
$$
will be positive.
We can then pick the constant $\xi$ to be a small enough multiple of $K^{-4}$ so that the term $O\left(u^2 K^3\xi^2\right)$ will be negligible (for example smaller than $\frac{\Delta}{100}$). Note that this is consistent with the constraint from the evaluation of $S_5^{(m)}$ that $\xi < \frac 1K$.

By \Cref{lem-averaging_admissible_tuples}, the sums over $\mathbf b$ and $\nu_0$ are bounded below by
\[
\sum_{\substack{\mathbf b \in \mathcal B \\ \mathcal L = \mathcal L(\mathbf b) \text{ adm.}}} 
\sum_{\substack{\nu_0 \mod W \\ (\ell(\nu_0),W) = 1 \forall \ell \in \mathcal L}} 1
\gg_{K}
 \left(\frac{\eta}{q}\right)^{K-1} \left(\log x\right)^{\frac{K-1}{2}} W \left(\frac{\phi(W)}{W}\right)^K .
\]
Thus by definition of $B$,
\[
\frac{B^K x}{W}
\sum_{\substack{\mathbf b \in \mathcal B \\ \mathcal L = \mathcal L(\mathbf b) \text{ adm.}}} 
\sum_{\substack{\nu_0 \mod W \\ (\ell(\nu_0),W) = 1 \forall \ell \in \mathcal L}} 1
\gg_K
 \left(\frac{\eta}{q}\right)^{K-1}\left(\frac{q_3}{\phi(q_3)}\right)^K \frac{\left(\log x\right)^{K-\frac{1}{2}}}{(\log D_0)^K} x.
\]
Returning to $S$, we have that
\begin{multline*}
S \gg_K \left(\frac{\eta}{q}\right)^{K-1}\left(\frac{q_3}{\phi(q_3)}\right)^K L_K(F) \left( \Delta - \frac{1}{100}\Delta\right)  \frac{\left(\log x\right)^{K-\frac{1}{2}}}{(\log D_0)^K} x
+ \\
O\left(u^2\Ssix\right).
\end{multline*}
We can now set the parameter $\eta$ to be sufficiently small (in terms of $K,M,\theta_1,\theta_2,\xi$) such that the big-O term will be negligible, which implies that
\begin{equation}\label{eq:lower-bound-on-S-final}
S \gg_{K,M,\theta_1,\theta_2,\xi,\eta} \left(\frac{1}{q}\right)^{K-1}\left(\frac{q_3}{\phi(q_3)}\right)^K \frac{\left(\log x\right)^{K-\frac{1}{2}}}{(\log D_0)^K} x.
\end{equation}

Equation \eqref{eq-upper_bound_wn} implies that 
$$
S \ll  \#\{E_n \le x : E_{n+i-1} \equiv a_i \mod q \:\forall 1 \le i \le M\} \times \mathrm{exp}(O(K/\xi)) \left(\frac{\log R}{\log D_0}\right)^K,
$$
which along with equation \eqref{eq:lower-bound-on-S-final} and the fact that $\log R = \tfrac{\theta_2}{2} \log x$ implies that
\begin{equation*} 
\#\{E_n \le x : E_{n+i-1} \equiv a_i \mod q \: \forall 1 \le i \le M\}
\gg_{K,M,\theta_1,\theta_2,\xi,\eta}
\left(\frac{1}{q}\right)^{K-1}\left(\frac{q_3}{\phi(q_3)}\right)^K
\frac{x}{\sqrt{\log x}}.
\end{equation*}
This completes the proof.

\section{Proofs of sieve results}\label{sec:proofs-of-sieve-props}

The goal of this section is to prove \Cref{thm:S-i-sums-estimate}. Throughout, fix $\eta > 0$ and let $\mathcal L = \{\ell_i(n)\}_{1 \le i \le K}$ be a fixed tuple of linear forms $\ell_i(n) = qn + a_i + qb_i$, where $qb_i \le \eta \sqrt{\log x}$ for all $i$. Let $\nu_0$ be a congruence class modulo $W$ such that $(\ell(\nu_0),W) = 1$ for all $\ell \in \mathcal L$.

This section will be organized as follows. In Section \ref{subsec:density-sieve-notation}, we introduce notation that will be used throughout, and define the sieve weights $w_n(\mathcal L)$. Sections \ref{subsec:sieve-auxiliary-lemmas-sieve-weights} and \ref{subsec:sieve-auxiliary-lemmas-sums-of-two-squares} contain lemmas and computations that will be used throughout the proof of Theorem \ref{thm:S-i-sums-estimate}; the estimate \eqref{eq:w_n-upper-bound-for-average} is proven in Lemma \ref{lem:bounds-on-lambda-and-w-n}. Finally, equations \eqref{eq:thm:S-i-sums:S-1}, \eqref{eq:thm:S-i-sums:S-2}, \eqref{eq:thm:S-i-sums:S-3}, \eqref{eq:thm:S-i-sums:S-4}, \eqref{eq:thm:S-i-sums:S-5}, and \eqref{eq:thm:S-i-sums:S-6} are proven (respectively) in Sections \ref{subsec:sieve-s-1}, \ref{subsec:sieve-s-2}, \ref{subsec:sieve-s-3}, \ref{subsec:sieve-s-4}, \ref{subsec:sieve-s-5}, and \ref{subsec:sieve-s-6}, which completes the proof of Theorem \ref{thm:S-i-sums-estimate}.

\subsection{Sieve notation and setup} \label{subsec:density-sieve-notation}
We begin by fixing some notation in preparation for defining the weights $w_n(\mathcal L)$. Recall that $W$ is the product of primes
$p \equiv 3 \mod 4$ satisfying $p \le D_0$ and $(p,q) = 1$. In particular, this means that if a prime $p \equiv 3 \mod 4$ satisfies $p|\ell_i(n)$ and $p|\ell_j(n)$ for $x < n \le 2x$ and for two distinct linear forms $\ell_i,\ell_j \in \mathcal L$, then $p|q_3W$. Let $\mathcal D_{K} \subset \mathbb Z^{K}$\index{Sieve parameters! $\mathcal D_{K}$} denote the set of $K$-tuples $\mathbf d = (d_i)$ such that for all $i$, $(d_i,q_3W) = 1$, such that $(d_i,d_j) = 1$ for all $i \ne j$, and such that each $d_i$ is divisible only by primes congruent to $3$ mod $4$.

Let $F:[0,1]^{K} \to \mathbb R$\index{Sieve parameters! $F$} be a smooth function defined as follows. Let $R_{K} = \{(x_1, \dots, x_{K}) \in [0,1]^{K} : \sum_{i=1}^{K} x_i \le 1\}$\index{Sieve parameters! $R_K$}. Define $F(t_1, \dots, t_K)$ as
\begin{equation}\label{eq:defn-of-F}
    F(t_1, \dots, t_{K}) := \prod_{i=1}^{K} g(Kt_i),
\end{equation}
where\index{Sieve parameters! $g(t)$}
\begin{equation}\label{eq:defn-of-F:defn-of-g}
    g(t) = \begin{cases} \frac 1{1+t}, &\text{ if } t \le 1, \\ 0, &\text{ otherwise.}\end{cases}
\end{equation}
Note that $F$ is supported on the set $R_{K}$.

We are now ready to define the sieve weights $w_n(\mathcal L)$\index{Sieve parameters! $w_n(\mathcal L)$}, which are nearly identical in structure to the multi-dimensional Selberg sieve weights used in, among other papers, \cite{MR3272929-Maynard-small-gaps} and \cite{MR4498475-McGrath}. We define
\begin{equation}\label{eq:defn-of-wn-L}
w_n(\mathcal L) = \Big(\sum_{\substack{\mathbf d\in\mathcal{D}_K \\ d_i\mid \ell_i(n)}}\lambda_{\mathbf d}\Big)^2,
\end{equation}
where\index{Sieve parameters! $\lambda_{\mathbf d}$}
\begin{equation}\label{eq:defn-of-wn-L:defn-of-lambda}
\lambda_{\mathbf d} = \left(\prod_{i=1}^{K} \mu(d_i)d_i\right) \sum_{\substack{\mathbf r \in \mathcal D_{K}\\ d_i|r_i \forall i}} \frac{\mu(r)^2}{\phi(r)} F\left(\frac{\log r_1}{\log R}, \dots, \frac{\log r_K}{\log R}\right).
\end{equation}
We will write 
\begin{equation}\label{eq:defn-of-y-r}
    y_{\r} := F\left(\frac{\log r_1}{\log R}, \dots, \frac{\log r_K}{\log R}\right)
\end{equation}\index{Sieve parameters! $y_{\r}$}
where $F$ is a function defined in \eqref{eq:defn-of-F}.

The results of our sieve evaluations will depend on the following functionals on $F$:\index{Sieve parameters! $L_K(F), L_{K;m}(F), L_{K;m_1,m_2}(F)$}
\begin{align}\label{eq:defn-of-LKF-functionals}
\begin{split}
&L_K(F) := \int_0^1 \cdots \int_0^1 \left[F(x_1, \dots, x_K)\right]^2 \prod_{i=1}^K \frac{\mathrm dx_i}{\sqrt{x_i}}, \\
&L_{K;m}(F) := \int_0^1 \cdots \int_0^1 \left[\int_0^1 F(x_1, \dots, x_K)\frac{\mathrm dx_m}{\sqrt{x_m}}\right]^2 \prod_{\substack{i = 1 \\ i \ne m}}^K \frac{\mathrm dx_i}{\sqrt{x_i}}, \\ 
&L_{K;m_1,m_2}(F) := \int_0^1 \cdots \int_0^1 \left[\int_0^1 \left(\int_0^1 F(x_1, \dots, x_K)\frac{\mathrm dx_{m_1}}{\sqrt{x_{m_1}}}\right)\frac{\mathrm dx_{m_2}}{\sqrt{x_{m_2}}}\right]^2 \prod_{\substack{i = 1 \\ i \ne m_1,m_2}}^K \frac{\mathrm dx_i}{\sqrt{x_i}}.
\end{split}
\end{align}
Using the function $F$ that is explicitly given by \eqref{eq:defn-of-F} and \eqref{eq:defn-of-F:defn-of-g}, we can evaluate each of $L_K(F)$, $L_{K;m}(F)$, and $L_{K;m_1,m_2}(F)$; each of these will be a constant depending only on $K$. More convenient, however, is using the following lemma, which relates each of these values to $L_K(F)$.

\begin{lemma}[\cite{MR4498475-McGrath}, Lemma 6.4: Evaluation of sieve functionals.]\label{lem:evaluation-of-sieve-functions-in-terms-of-F}
Let $F(t_1, \dots, t_K)$ be given by equation \eqref{eq:defn-of-F} and let $L_K(F)$, $L_{K;m_1}(F)$, and $L_{K;m_1,m_2}(F)$ denote the functionals defined in \eqref{eq:defn-of-LKF-functionals}.
Then for any $m_1, m_2,$
\begin{equation*}
\frac{L_{K;m_1}(F)}{L_K(F)} = \frac{\pi^2}{\pi + 2}  \sqrt{\frac{1}{K}} \quad \text{ and } \quad
\frac{L_{K;m_1, m_2}(F)}{L_K(F)} = \left(\frac{\pi^2}{\pi + 2}\right)^2 \frac{1}{K}.
\end{equation*}
\end{lemma}

\subsection{Auxiliary lemmas for sieve weights}
\label{subsec:sieve-auxiliary-lemmas-sieve-weights}
This subsection and the next collect various lemmas that will be used throughout our estimates. To begin with, we present several lemmas concerning the sieve weights defined in Section \ref{subsec:density-sieve-notation}. 

\begin{lemma}\label{lem:taking-out-one-factor-lemma-8.2}
\begin{enumerate}[(i)]
    \item Let $\r,\s \in \mathcal D_K$ with $s_i = r_i$ for all $i \ne j$ and $s_j = Ar_j$ for some $A \in \mathbb N$. Then for $y_{\r}$ and $y_{\s}$ defined by \eqref{eq:defn-of-y-r}, we have
    \begin{equation*}
        y_{\s} = y_{\r} + O\left(K \frac{\log A}{\log R}y_{\r}\right).
    \end{equation*}
    \item Let $\r,\s \in \mathcal D_K$ with $r = s$ and let $A$ be the product of primes dividing $r$ but not $(\r,\s)$. 
    Then for $y_{\r}$ and $y_{\s}$ defined by \eqref{eq:defn-of-y-r}, we have
    \begin{equation*}
        y_{\s} = y_{\r} + O\left(K \frac{\log A}{\log R}(y_{\r} + y_{\s})\right).
    \end{equation*}
\end{enumerate}
\end{lemma}
\begin{proof}
Recall that $y_{\r} = F\left(\frac{\log r_1}{\log R}, \dots, \frac{\log r_K}{\log R}\right)$, where $F(t_1, \dots, t_K) := \prod_{i=1}^K g(Kt_i)$ and $g(t) = \frac 1{1+t}$ for $t \le 1$ and $g(t) = 0$ otherwise. Given $u,v \ge 0$ with $|u-v| < \ep$, we have
\begin{equation*}
\frac 1{1+Ku} = \frac{1+O(K\ep)}{1+Kv}.
\end{equation*}
Let $u_i = \log r_i/\log R$, $v_i = \log s_i/\log R$ and $\ep_i = v_i-u_i$. In part (i), $\ep_i = 0$ for $i \ne j$ and $\ep_j = \log A/\log R$. Thus
\begin{equation*}
\frac 1{1+Kv_j} = \frac{1+O\left(\tfrac{K\log A}{\log R}\right)}{1+Ku_j}.
\end{equation*}
Multiplying by $\prod_{i\ne j} 1/(1+Ku_i)$ gives the result for (i).

Now consider part (ii), and let $\mathbf t$ be the vector with $t_i = [r_i,s_i]$. Applying part (i) to each component in turn implies that
\begin{equation*}
    y_\s=y_{\mathbf t} + O\left(Ky_{\s} \sum_{i=1}^K \frac{\log\left([r_i,s_i]\right)/s_i}{\log R}\right) = y_{\mathbf t} + O\left(Ky_{\s}\frac{\log A}{\log R}\right).
\end{equation*}
The same holds for $\r$ and $\mathbf t$, which implies (ii).
\end{proof}

The following lemma is a standard evaluation of sums of multiplicative functions that appear frequently in sieve computations.

\begin{lemma}\label{lem:evaluating_single_variable}
Let $A_1,A_2,L >0$ and let $\gamma$ be a multiplicative function satisfying 
$$
0\leq \frac{\gamma(p)}{p} \leq 1 - \frac{1}{A_1}
$$
and
$$
-L \leq 
\sum_{w\leq p < z} \frac{\gamma(p)\log p}{p}
- \frac{1}{2} \log\left(\frac{z}{w}\right)
\leq A_2
$$
for $2\leq w \leq z$.
Let $g$ be the multiplicative function defined by $g(p) = \frac{\gamma(p)}{p - \gamma(p)}$, and let $G: [0,1] \to \R$ be a piece-wise differentiable function.
Then
$$
\sum_{d< z}\mu^2(d)g(d)G\left(\frac{\log d}{\log z}\right)
=
c_\gamma \frac{(\log z)^{\frac{1}{2}}}{\Gamma\left(1/2\right)}
\int_0^1 G(t)\frac{dt}{\sqrt{t}}
+O\left(c_\gamma L G_{\mathrm{max}}(\log z)^{-\frac{1}{2}}\right)
$$
where 
$$
c_\gamma = \prod_{p}\left(1 - \frac{\gamma(p)}{p}\right)^{-1}\left(1 - \frac{1}{p}\right)^{\frac{1}{2}}
$$
and 
$$
G_{\mathrm{max}} = \sup_{t\in [0,1]}\left(|G(t)| + |G'(t)|\right).
$$
\end{lemma}
\begin{proof}
This is \cite[Lemma 4]{MR2500871} with $\kappa = \frac{1}{2}$.
\end{proof}
Using Lemma \ref{lem:evaluating_single_variable}, we show the following lemma, which is similar to \cite[Lemma 8.4]{MR3530450-Maynard-dense-clusters}.
\begin{lemma}\label{lem:evaluating-sieve-sums-lemma-8.4-maynard}
Let $Q$ be a squarefree modulus of the form $Q = q_3 W \alpha$ with $\alpha = R^{O(K)}$.
Let $f$ be a multiplicative function with $f(p) = p + O(K)$, and let $G:\R \to \R$ be a smooth decreasing function supported on $[0,1]$.

Then for $K$ sufficiently large, we have
\begin{multline*}
\sum_{\substack{\mathbf e \in \mathcal{D}_K \\ (e,Q) = 1}}
\frac{\mu^2(e)}{f(e)} 
\prod_{i=1}^K G\left(\frac{\log e_i}{\log R}\right) 
\\
= (1+o(1))
\left(\frac{\phi(\alpha)}{\alpha}\right)^{\frac{K}{2}}
\left(\frac{e^{-\gamma /2}}{\Gamma(1/2)}\right)^K
\left(\frac{\log R}{\log D_0}\right)^{K/2}
\left(\int_{t_1, \dots, t_K \ge 0} \prod_{i=1}^K G(t_i) \frac{\mathrm{d}t_i}{\sqrt{t_i}}\right).
\end{multline*}
\end{lemma}
\begin{proof}
We would like to apply \Cref{lem:evaluating_single_variable} for each variable.
However, the variables $e_i$ are not independent, since $\mathbf e \in \mathcal{D}_K$ implies $(e_i,e_j) = 1$ for $i\neq j$.
Our first task is to decouple these variables with a negligible penalty.

Denote the sum evaluated in the lemma by
$$
S = \sum_{\substack{\mathbf e \in \mathcal{D}_K \\ (e,Q) = 1}}
\frac{\mu^2(e)}{f(e)} 
\prod_{i=1}^K G\left(\frac{\log e_i}{\log R}\right).
$$
If $p|(e_r,e_j)$ for some $1\leq r< j \leq K$, then we must have $p > D_0$ since $(e,Q) = 1$.
Thus, using also the fact that $G$ is decreasing, we have that 
\begin{align*}
\sum_{\substack{\mathbf e \in \mathbb N_K \\ (e,Q) = 1 \\ p\mid(e_i,e_j)}}
&\prod_{j=1}^K \frac{\mu^2(e_j)}{f(e_j)} 
\prod_{i=1}^K G\left(\frac{\log e_i}{\log R}\right)
\\ &\ll 
\frac{1}{f(p)^2}
\sum_{\substack{\mathbf e \in \mathbb N_K \\ (e,Q) = 1}}
\prod_{j=1}^K\frac{\mu^2(e_j)}{f(e_j)} 
G\left(\frac{\log e_r + \log p}{\log R}\right)
G\left(\frac{\log e_j + \log p}{\log R}\right)
\prod_{\substack{i=1 \\ i \neq r,j}}^K G\left(\frac{\log e_i}{\log R}\right)
\\ & \ll
(p + O(K))^{-2}S.
\end{align*}

Writing
$$
S' = 
\sum_{\substack{\mathbf e \in\mathbb{N}^ K \\ (e,Q) = 1}}
\prod_{j=1}^K \frac{\mu^2(e_j)}{f(e_j)} 
\prod_{i=1}^K G\left(\frac{\log e_i}{\log R}\right) = \Big(\sum_{\substack{n \in \mathbb N \\ (n,Q) = 1}} \frac{\mu^2(n)}{f(n)} G\left(\frac{\log n}{\log R}\right)\Big)^K ,
$$
it follows that 
$$
S - S'
\ll S \binom{K}{2}\sum_{p> D_0}\frac{1}{(p+O(K))^2}
\ll_K \frac{S}{D_0}.
$$

We now consider $S'$ in place of $S$ and apply
\Cref{lem:evaluating_single_variable}, where we take $g(p) = \frac 1{f(p)}$ and $\gamma(p) = \frac{p}{f(p)+1}$ for $p \nmid Q$ and $g(p) = \gamma(p) = 0$ for $p|Q$. 
For this we need a bound on the constant $L$ from \Cref{lem:evaluating_single_variable}.
If we did not have the restriction $g(p) = \gamma(p) = 0$ for $p|Q$, then using the prime number theorem in arithmetic progressions we could take $L$ to be a constant, since
$$
\sum_{\substack{w\leq p < z \\ p\equiv 3 \bmod 4}}\frac{\log p}{p + O(K)}
- \frac{1}{2}\log{\left(\frac{z}{w}\right)} \ll_K 1.
$$
It follows that in our case we can choose $L$ satisfying
$$
L \ll_K 1 + 
\sum_{\substack{p \leq D_0 \\ p\equiv 3 \bmod 4}}\frac{\log p }{p} + 
\sum_{p \mid \alpha }\frac{\log p}{p}.
$$
The first sum is $(1+o(1))\frac{1}{2}\log{D_0}$.
For the second sum, we have the bound
$$
\sum_{p|\alpha} \frac{\log p}{p} \ll_{K} \log\log R
$$
since $\alpha = R^{O(K)}$ and the sum is dominated by taking the smallest possible primes.
This implies that
$$
L \ll_{K} \frac{1}{2}\log D_0 + \log\log R \ll_{K} \log\log R.
$$

We will write
$$
\Omega_G =G_{\mathrm{max}}\cdot  
\left(\int_{t\geq 0} G(t)\frac{dt}{\sqrt{t}}\right)^{-1},
$$
where $G_{\mathrm{max}}$ is defined as in \Cref{lem:evaluating_single_variable}.
Applying \Cref{lem:evaluating_single_variable} successively $K$ times, we get
\begin{align*}
S'
&= 
\frac{c_\gamma^K}{\Gamma(1/2)^K}(\log R)^{K/2}
\left(\int_{t_1, \dots, t_K \ge 0} \prod_{i=1}^K G(t_i) \frac{\mathrm{d}t_i}{\sqrt{t_i}}\right)
\left(1 +  \sum_{\ell = 1}^K
\binom{K}{\ell}
O_{K}\left( \frac{\Omega_G  \log\log R}{\log R}\right)^\ell
\right)
\\ &=
\frac{c_\gamma^K}{\Gamma(1/2)^K}(\log R)^{K/2}
\left(\int_{t_1, \dots, t_K \ge 0} \prod_{i=1}^K G(t_i) \frac{\mathrm{d}t_i}{\sqrt{t_i}}\right)
\left( 1 + O_{K,G}\left(\frac{  \log\log R}{\log R}\right)\right)
\end{align*}
where $c_\gamma$ is given in \Cref{lem:evaluating_single_variable} and satisfies
$$
c_\gamma
= 
(1 + o(1))
\prod_{p\mid \alpha}\left(1 - \frac{1}{p}\right)^{\frac{1}{2}}
\prod_{p\leq D_0} \left(1 - \frac{1}{p}\right)^{\frac{1}{2}}
=
(1 + o(1))
\left(\frac{\phi(\alpha)}{\alpha}\right)^{\frac{1}{2}}
\frac{e^{-\gamma/2}}{(\log D_0)^{\frac{1}{2}}}
$$

Combining our estimates completes the proof.

\end{proof}

\Cref{lem:evaluating-sieve-sums-lemma-8.4-maynard} will be useful for many of our computations to follow. For now, we use it to verify that the weights $w_n(\mathcal L)$ given in \eqref{eq:defn-of-wn-L}
are bounded above as in \eqref{eq:w_n-upper-bound-for-average}, which we show in the following lemma.
\begin{lemma}\label{lem:bounds-on-lambda-and-w-n}
We have the bounds
\begin{enumerate}[(i)]
\item \begin{equation}\label{eq:lambda-d-bound}
    |\lambda_{\mathbf d}| \ll_K \left(\frac{\log R}{\log D_0}\right)^{K/2},
\end{equation}
\item \[w_n(\mathcal L) \ll_K\left(\frac{\log R}{\log D_0}\right)^{K} \prod_{i=1}^K \prod_{\substack{p|\ell_i(n) \\ p \equiv 3 \mod 4}}4, \]
\item 
\begin{equation}\label{eq:w-n-upper-bound-for-sieve-proofs}
w_n(\mathcal L) \ll_K R^{2+o(1)}.
\end{equation}
\end{enumerate}
\end{lemma}
\Cref{lem:bounds-on-lambda-and-w-n}, (ii) is precisely \eqref{eq:w_n-upper-bound-for-average}, whereas (i) and (iii) will be used in our sieve estimates.
\begin{proof}
By our choice of $y_{\mathbf r}$, for any $\mathbf d \in \mathcal D_K$, we have 
\begin{equation}\label{eq:lambda-d-abs-value-upper-bound-part-1}
|\lambda_{\mathbf d}| = \prod_{i=1}^K d_i \sum_{\substack{\mathbf r \in \mathcal D_K \\ \mathbf d|\mathbf r}} \frac{y_{\mathbf r}}{\phi(r)} = \frac{d}{\phi(d)}\sum_{\mathbf d|\mathbf r \in \mathcal D_K} \frac{F\left(\frac{\log r_1}{\log R},\dots, \frac{\log r_K}{\log R}\right)}{\phi(r/d)}.
\end{equation}
Since $F$ is decreasing in each argument, we can bound $|\lambda_{\mathbf d}|$ above by replacing each $\frac{\log r_i}{\log R}$ in the argument of $F$ with $\sigma_i = \frac{(\log r_i/d_i)}{\log R}$. 

We now apply \Cref{lem:evaluating-sieve-sums-lemma-8.4-maynard} with $Q = dq_3W$, which gives
\begin{align*}
&\sum_{\mathbf d|\mathbf r \in \mathcal D_K} 
\frac{F(\sigma_1, \dots, \sigma_K)}{\phi(r/d)} 
= 
\sum_{\mathbf d|\mathbf r\in \mathcal D_K} 
\frac 1{\phi(r/d)}
\prod_{i=1}^K g(K\sigma_i) 
\\& \qquad \sim
\left(\frac{\phi(d)}{d}\right)^{\frac{K}{2}}
\left(\frac{e^{-\gamma /2}}{\Gamma(1/2)}\right)^K
\left(\frac{\log R}{\log D_0}\right)^{K/2}
\left(\int_{t_1, \dots, t_K \ge 0} \prod_{i=1}^K g(K t_i) \frac{\mathrm{d}t_i}{\sqrt{t_i}}\right).
\end{align*}
Evaluating the integral gives
\begin{equation*}
    \int_{t_1, \dots, t_K \ge 0} \prod_{i=1}^K g(Kt_i) \frac{\mathrm dt_i}{\sqrt{t_i}} = \left(\frac{\pi}{2\sqrt{K}}\right)^K.
\end{equation*}
Substituting this expression back into \eqref{eq:lambda-d-abs-value-upper-bound-part-1}, we get that
\begin{align*}
|\lambda_{\mathbf d}|
\ll_K 
\left(\frac{\phi(d)}{d}\right)^{\frac{K}{2}-1}
\left(\frac{\log R}{\log D_0}\right)^{K/2}
\ll
\left(\frac{\log R}{\log D_0}\right)^{K/2}
\end{align*}
for $K\geq4$.
This completes the proof of \eqref{eq:lambda-d-bound}.

For the second claim, recall that 
\begin{align*}
w_n(\mathcal L) &=
\Big(\sum_{\substack{\mathbf d \in \mathcal D_K \\ d_i\mid\ell_i(n)}} \lambda_{\mathbf d}\Big)^2 
\ll_K 
\left(\frac{\log R}{\log D_0}\right)^{K} \#\{\mathbf d \in \mathcal D_K : d_i\mid \ell_i(n)\}^2 
\\ &\ll_K
\left(\frac{\log R}{\log D_0}\right)^{K} \prod_{i=1}^K \prod_{\substack{p\mid \ell_i(n) \\ p \equiv 3 \mod 4}}4,
\end{align*}
as desired.

For the third claim, note that since $\lambda_{\mathbf d}$ is supported on $d=d_1 \cdots d_K < R$, we have
\begin{align*}
w_n(\mathcal L) &\ll_K
\left(\frac{\log R}{\log D_0}\right)^{K} \left(\sum_{d_1\cdots d_K < R} 1\right)^2 \ll_K R^{2+o(1)} \left(\sum_{d_1 \cdots d_K < R} \frac 1{d_1 \cdots d_K}\right)^2 \ll_K R^{2+o(1)}.
\end{align*}
\end{proof}

Finally, throughout our sieve arguments, we will make crucial use of \cite[Lemma 6.6]{MR4498475-McGrath}, which we restate below for convenience.

\begin{lemma}\label{lem:mcgrath-lem-6-6}
Let $J \subseteq \{1, \dots, K\}$ (possibly empty) and $p_1, p_2 \in \mathbb P \cup \{1\}$ be fixed. Write $I = \{1, \dots, K\} \setminus J$. Define the sieve sum $S_{J,p_1,p_2,m} = S_{J,p_1,p_2,m,f,g}$ by
\begin{equation*}
S_{J,p_1,p_2,m} = \sum_{\substack{\mathbf d, \mathbf e \in \mathcal D_K \\ q_3W,[d_1,e_1], \dots, [d_{K},e_{K}] \text{ coprime} \\ p_1|d_m, p_2|e_m }} \lambda_{\mathbf d} \lambda_{\mathbf e} \prod_{i \in I} f([d_i,e_i]) \prod_{j \in J} g([d_j,e_j]),
\end{equation*}
with weights $\lambda_{\mathbf d}$ defined as in \eqref{eq:defn-of-wn-L}.
If $J = \emptyset$ we define $f(p) = 1/p$ (and in this case there is no dependence on $g$ in the sum). Otherwise, $f$ and $g$ are nonzero multiplicative functions defined on primes by
\begin{equation*}
f(p) = \frac 1p + O\left(\frac 1{p^2}\right), \quad g(p) = \frac 1{p^2} + O\left(\frac 1{p^3}\right),
\end{equation*}
and moreover we assume that $f(p) \ne 1/p$. We write $S_J$ for $S_{J,1,1,m}$. 
Then for $|J| \in \{0,1,2\}$ we have the following:
\begin{enumerate}[(i)]
\item If $m \in J$, then
\begin{equation*}
    S_{J,p_1,p_2,m} \ll \frac{F_{max}^2B^{K+|J|} (\log \log R)^2}{(p_1p_2/(p_1,p_2))^2}.
\end{equation*}
\item If $m \not\in J$ then
\begin{equation*}
    S_{J,p_1,p_2,m} \ll \frac{F_{max}^2B^{K+|J|} (\log \log R)^2}{p_1p_2/(p_1,p_2)}.
\end{equation*}
\item We have
\begin{equation*}
    S_J = (1+o(1))B^{K+|J|}L_J(F),
\end{equation*}
where we write $L_J(F)$ as shorthand for $L_{K;j\in J}(F)$, defined analogously to \eqref{eq:defn-of-LKF-functionals}.
\end{enumerate}
\end{lemma}

\begin{remark}
This is \cite[Lemma 6.6]{MR4498475-McGrath} with $q_3W$ in place of $W$.
Note also that the restriction $\mathbf d, \mathbf e \in \mathcal D_K$ in the definition of $S_{J,p_1,p_2,m}$ doesn't affect the statement since $\lambda_{\mathbf d} = 0$ for $\mathbf d\not\in \mathcal D_K$.
\end{remark}

\subsection{Auxiliary lemmas for sums of two squares}
\label{subsec:sieve-auxiliary-lemmas-sums-of-two-squares}

In this subsection we record several useful results on estimates of the functions $\rho$ and $r_2$. To begin with, we have the following lemma, which is \cite[Lemma 5.3]{MR4498475-McGrath}, and will be used in the proof of \eqref{eq:thm:S-i-sums:S-2}.

\begin{lemma}\label{lem:mcgrath-lem-5-3}
Let $(\a,r) = (d,r) = 1$, where $d$ and $r$ are squarefree, odd, and $\ll x^{O(1)}$. Then
\begin{equation*}
\sum_{\substack{n \le x \\ n \equiv \a \pmod r \\ n \equiv 1 \pmod 4 \\ d|n}} r_2(n) = \frac{g_1(r)g_2(d)}{2rd} \pi x + R_1(x;d,r),
\end{equation*}
where $g_2$ is defined as in \eqref{eq:defn-of-g-2},
$g_1$\index{Sums of two squares! $g_1(n)$} is the multiplicative function defined on primes by 
$$
g_1(p) = 1-\chi(p)/p,
$$ 
and
\begin{equation*}
    R_1(x;d,r) \ll_\ep ((rd)^{\tfrac 12} + x^{\tfrac 13})d^{\tfrac 12}x^\ep.
\end{equation*}
\end{lemma}

The following lemma is nearly identical to \cite[Lemma A.3]{MR4498475-McGrath}, and will be used to prove \eqref{eq:thm:S-i-sums:S-3}. 
\begin{lemma}\label{lem:mcgrath-lem-A-3}
Let $r$ be a modulus and suppose that $(\a,r) = (\a+h,r) = (c_1,r) = (c_2,r) = 1$ and $4|h$, where $c_1,c_2,r$ are squarefree and odd, of size $\ll x^{O(1)}$. Then we have 
\begin{equation*}
\sum_{\substack{n \le x \\ n \equiv \a \pmod r \\ n \equiv 1 \pmod 4 \\ c_1|n \\ c_2|n+h}} r_2(n)r_2(n+h) = \frac{g_1(r)^2\Gamma(h,c_1,c_2,r)}{r} \pi^2 x + R_2(x;c_1,c_2,r),
\end{equation*}
where
\begin{equation*}
    \Gamma(h,c_1,c_2,r) = \frac{g_2(c_1)g_2(c_2)}{c_1c_2} \sum_{(t,2r) = 1} \frac{c_t(h)(c_1,t)(c_2,t)\chi[(c_1^2,t)]\chi[(c_2^2,t)]}{t^2 \Psi(c_1,t)\Psi(c_2,t)},
\end{equation*}
where $\Psi(u,t) := g_2((u,t/(u,t)))$\index{Sums of two squares!$\Psi(u,t)$}, $c_t(h)$\index{Sums of two squares!$c_t(h)$} is the Ramanujan sum $c_t(h) = \sum_{d\mid (t,h)}\mu\left(\frac{t}{d}\right)d$, and
\begin{equation*}
    R_2(x;c_1,c_2,r) \ll_\ep r^{\tfrac 12}c_1c_2x^{\tfrac 34 + \ep} + c_1^{\tfrac 12}c_2^{\tfrac 12}x^{\tfrac 56 + \ep}.
\end{equation*}
\end{lemma}
The only difference between Lemma \ref{lem:mcgrath-lem-A-3} and \cite[Lemma A.3]{MR4498475-McGrath} is that the latter requires the additional constraint that $(c_1,c_2) = 1$. However, this constraint is not used in the proof, which refers heavily to the proof of \cite[Lemma 4]{MR0947418-Plaksin}. Note that if $(c_1,c_2)\nmid h$, then the sum over $r_2(n)r_2(n+h)$ is empty; in this case $\Gamma(h,c_1,c_2,r) = 0$ and the equation still holds. One can see that $\Gamma(h,c_1,c_2,r) = 0$ when $(c_1,c_2)\nmid h$ by noting that $\Gamma$ is multiplicative over primes $p|[h,c_1,c_2,r]$ and that $c_t(h) = \mu(t)$ whenever $(t,h) = 1$. Then if any prime $p$ divides $(c_1,c_2)$ but $p\nmid h$, the $p$-component of $\Gamma(h,c_1,c_2,r)$ is $0$.

The following lemma is \cite[Lemma 5.5]{MR4498475-McGrath}, and will be used in the proof of \eqref{eq:thm:S-i-sums:S-4}.
\begin{lemma}\label{lem:mcgrath-lem-5-5}
Let $(a,r) = (d,r) = 1$, where $d$ and $r$ are squarefree, odd, and $\ll x^{O(1)}$.
Then
\begin{equation*}
\sum_{\substack{n \leq x \\ n\equiv a \bmod r \\n\equiv 1 \bmod 4\\ d\mid n}}
r_2^2(n)
=
\frac{g_3(r)g_4(d)}{rd}\left(
\log x + A_2 + 2\sum_{p\mid r}g_5(p) - 2\sum_{p\mid d}g_6(p)\right)x 
+ O_\varepsilon\left(rx^{\frac{3}{4} + \varepsilon}\right),
\end{equation*}
where\index{Sums of two squares! $g_3(n),g_4(n),g_5(n),g_6(n)$}
\begin{align*}
g_3(p) &:=\begin{cases} \tfrac{(p-1)^2}{p(p+1)} &\text{ if } p \equiv 1 \mod 4 \\ g_1(p) &\text{ if } p \equiv 3 \mod 4 \end{cases} &\quad
g_4(p) &:= \begin{cases} \tfrac{4p^2 - 3p + 1}{p(p+1)} &\text{ if } p \equiv 1 \mod 4 \\ g_2(p) &\text{ if } p \equiv 3 \mod 4 \end{cases} \\
g_5(p) &:= \begin{cases} \tfrac{(2p+1)\log p}{p^2 -1} &\text{ if } p \equiv 1 \mod 4 \\ \tfrac{\log p}{p^2-1} &\text{ if } p \equiv 3 \mod 4 \end{cases}  &\quad
g_6(p) &:= \begin{cases} \tfrac{(p-1)^2(2p+1)\log p}{(p+1)(4p^2 - 3p + 1)} &\text{ if } p \equiv 1 \mod 4 \\ \log p &\text{ if } p \equiv 3 \mod 4 \end{cases} \\
\end{align*}
and $A_2$ is a constant given explicitly in \cite[Lemma 5.5]{MR4498475-McGrath}.
\end{lemma}

The following lemma is \cite[Lemma 5.6]{MR4498475-McGrath}, and will be used to prove \eqref{eq:thm:S-i-sums:S-2}, \eqref{eq:thm:S-i-sums:S-3}, and \eqref{eq:thm:S-i-sums:S-4}.
\begin{lemma}\label{lem:mcgrath-lem-5-6} 
Let $Q$ be a squarefree integer such that $Q\mid \prod_{p \le (\log \log x)^3} p$ and any prime $p$ dividing $Q$ is $1 \bmod 4$. Define $X_{x,Q}$\index{Sums of two squares!$X_{x,Q}$}, $Z_{x,Q}^{(1)}$\index{Sums of two squares!$Z_{x,Q}^{(1)}$}, and $Z_{x,Q}^{(2)}$\index{Sums of two squares!$Z_{x,Q}^{(2)}$} as follows:
\begin{align*}
X_{x,Q} &= 
\sum_{\substack{a \le v \\ (a,Q) = 1 \\ p|a \Rightarrow p \equiv 1 \pmod 4}} \frac{\mu(a)}{a}\log \frac va, \\
 Z_{x,Q}^{(1)} &=
\sum_{\substack{a,b\leq v \\ (a,Q) = (b,Q) = 1 \\ p\mid ab \Rightarrow p\equiv 1 \bmod 4}}
\frac{\mu(a)\mu(b)g_4([a,b])}{g_2(a)g_2(b)[a,b]}\log\frac{v}{a}\log\frac{v}{b} \\
 Z_{x,Q}^{(2)} &=
\sum_{\substack{a,b\leq v \\ (a,Q) = (b,Q) = 1 \\ p\mid ab \Rightarrow p\equiv 1 \bmod 4}}
\frac{\mu(a)\mu(b)g_4([a,b])}{g_2(a)g_2(b)[a,b]}\log\frac{v}{a}\log\frac{v}{b}
\sum_{p\mid ab}g_6(p).
\end{align*}

Then
\begin{align*}
X_{x,Q} &= 
(1+o(1)) \frac{8A\log^{\tfrac 12}v}{\pi g_1(Q)}, \\
Z_{x,Q}^{(1)} &= (1+o(1)) 8A \frac{g_7(Q)\log^{\tfrac 12}v}{\phi(Q)g_1(Q)}\prod_{\substack{p \equiv 1 \mod 4 \\ p\nmid Q}} \left(1 + \frac 1{(2p-1)^2}\right),  \text{ and}\\ 
Z_{x,Q}^{(2)} &= -(1+o(1)) 4A \frac{g_7(Q)\log^{\tfrac 32}v}{\phi(Q)g_1(Q)}\prod_{\substack{p \equiv 1 \mod 4 \\ p\nmid Q}} \left(1 + \frac 1{(2p-1)^2}\right), \\ 
\end{align*}
where $g_7$\index{Sums of two squares! $g_7(n)$} is the multiplicative function defined on primes by $g_7(p) = p+1$.
\end{lemma}
\begin{proof}
The proof of this lemma is for the most part identical to the proof in Appendix B of \cite{MR4498475-McGrath}, so here we will restrict ourselves to highlighting the differences in the argument.

In general, the application of the Selberg--Delange method is identical to that described in \cite{MR4498475-McGrath}, with the same arguments applying for bounding, for example, relevant analytic functions; the only change that need be made to McGrath's arguments is replacing $W$ by an arbitrary $Q$ (which must divide the $W$ that McGrath uses) everywhere. Thus by following McGrath's proof we get that, for 
\begin{equation*}
K_1(s)^2 = \left(1-\frac 1{2^s}\right)^{-1}\prod_{p\equiv 3 \mod 4}\left(1-\frac 1{p^{2s}}\right)^{-1}
\end{equation*}
and
\begin{equation*}
    G_1(Q,s) = \prod_{\substack{p|Q \\ p \equiv 1 \mod 4 }} \left(1-\frac 1{p^s}\right)^{-1},
\end{equation*}
we have
\begin{equation*}
X_{x,Q} = \frac{K_1(1)G_1(Q)}{\Gamma(3/2)\sqrt{L(1,\chi_4)}}(\log v)^{\tfrac 12} + O_\ep\left(\frac{(\log x)^{\ep}}{(\log v)^{\tfrac 12}}\right),
\end{equation*}
which simplifies to the desired expression.

In the same manner, expressions for $Z^{(1)}_{x,Q}$ and $Z^{(2)}_{x,Q}$ can be derived.
\end{proof}

\subsection{Estimating \texorpdfstring{$S_1$}{S1}}
\label{subsec:sieve-s-1}

In this section we will prove \Cref{thm:S-i-sums-estimate}, equation \eqref{eq:thm:S-i-sums:S-1}, which we restate in the following proposition. 

\begin{proposition}
Consider the sum $S_1(\nu_0)$ defined by
\[S_1(\nu_0) := \sum_{\substack{x < n \le 2x \\ n \equiv 1 \mod 4 \\ n\equiv \nu_0 \bmod W}} w_n(\mathcal L).\]
Then
\[S_1(\nu_0) = (1 + o(1)) \frac{B^K x}{4W} L_K(F).\]
\end{proposition}

\begin{proof}
We have
\begin{align*}
S_1(\nu_0) &=
\sum_{\substack{x\leq n \leq 2x \\ n\equiv 1\bmod 4 \\ n \equiv \nu_0 \mod W}}w_n(\mathcal L)
=
\sum_{\substack{x\leq n \leq 2x \\ n\equiv 1\bmod 4 \\ n \equiv \nu_0 \mod W}}
\left(\sum_{\mathbf d \in \mathcal D_{K}} \lambda_{\mathbf d}\right)^2. \\
\end{align*}
By expanding the square and swapping the order of summation, we get that
\begin{align*}
S_1(\nu_0) &=
\sum_{\substack{\mathbf d, \mathbf e \in \mathcal D_K}}
\lambda_{\mathbf d}\lambda_{\mathbf e}
\sum_{\substack{x \leq n \leq 2x \\ qn+a_i+qb_i \equiv 0 \bmod{[d_i,e_i]} \\ n\equiv 1 \bmod 4 \\ n \equiv \nu_0 \mod W}} 1
\\ 
&= \sum_{\substack{ \mathbf d, \mathbf e \in \mathcal D_K}}\lambda_{\mathbf d}\lambda_{\mathbf e}
\left(\frac{x}{4W}\prod_{i=1}^{K}\frac{1}{[d_i,e_i]} + O(1)
\right)
\\ 
&= \frac{x}{4W}
\sum_{\mathbf d, \mathbf e \in \mathcal D_{K}} \frac{\lambda_{\mathbf d}\lambda_{\mathbf e}}{[\mathbf d, \mathbf e]}
+O\left(
\sum_{\substack{\mathbf d, \mathbf e \in \mathcal D_K}}\lambda_{\mathbf d}\lambda_{\mathbf e}
\right),
\end{align*}
where the second line follows from the first by the Chinese remainder theorem and the observation that $4$, $W$, and the $[d_i,e_i]$ are all pairwise relatively prime.
In the notation of \Cref{lem:mcgrath-lem-6-6}, the sum in the main term is precisely the sum $S_{\emptyset} = S_{\emptyset,1,1,m}$, which by \Cref{lem:mcgrath-lem-6-6} is equal to $B^KL_K(F)$.

It remains only to bound the error term;
We have
$$
\left|
\sum_{\mathbf d, \mathbf e\in\mathcal D_K}\lambda_{\mathbf d}\lambda_{\mathbf e}
\right|
\ll \lambda_{\mathrm{max}}^2 |\mathcal D_K|^2.
$$
We use \eqref{eq:lambda-d-bound} to bound $\lambda_{\text{max}}$, and the bound
\begin{equation}\label{eq-bound_DK}  
|\mathcal D_K| \leq \sum_{n\leq R}\tau_K(n) \ll_\epsilon R^{1 + \epsilon}.
\end{equation}
Since $R = x^{\theta_2/2}$ with $\theta_2<\frac{1}{18}$ we conclude
$$
\left|
\sum_{\mathbf d, \mathbf e\in\mathcal D_K}\lambda_{\mathbf d}\lambda_{\mathbf e}
\right|
\ll_{K,\epsilon} R^{2+\epsilon} 
= x^{\theta_2 + \epsilon} \ll x^{1/18 + \ep},
$$
which is negligible.
This completes the proof.
\end{proof}
\subsection{Estimating \texorpdfstring{$S_2^{(m)}$}{S2}}
\label{subsec:sieve-s-2}
In this section we will prove \Cref{thm:S-i-sums-estimate}, equation \eqref{eq:thm:S-i-sums:S-2}, which we restate in the following proposition.

\begin{proposition}
For fixed $1 \le m \le K$, consider the sum $S_2^{(m)}(\nu_0)$ defined by
\begin{equation*}
    S_2^{(m)}(\nu_0) := \sum_{\substack{x < n \le 2x \\ n \equiv 1 \mod 4 \\ n\equiv \nu_0\bmod W}} \rho(\ell_m(n))w_n(\mathcal L).
\end{equation*}
Then
\[S_2^{(m)}(\nu_0) = (1+o(1)) \frac{4\pi \sqrt{\frac{\log R}{\log v}} B^Kx}{(\pi + 2)\sqrt{K} W} L_K(F).\]
\end{proposition}

\begin{proof}
By definition of $\rho$, we have that
\begin{equation*}
\rho(\ell_m(n))
=
\frac{r_2(\ell_m(n))}{\log v}\sum_{\substack{a\mid \ell_m(n) \\ p\mid a \Rightarrow p\equiv 1 \bmod 4 \\ a\leq v}}
\frac{\mu(a)}{g_2(a)} \log \left( \frac{v}{a}\right).
\end{equation*}
Note that since $a|\ell_m(n)$, which is relatively prime to $q$, we must have $(a,q) = 1$ for any $a$ in the definition of $\rho(\ell_m(n))$. Since $a$ is only divisible by primes that are $1$ mod $4$, the only nontrivial constraint on $a$ is that $(a,q_1) = 1$. 
By expanding the definitions of $w_n(\mathcal L)$ and $\rho$ in the expression for $S_2^{(m)}(\nu_0)$ and changing the order of summation, we get
\begin{equation}\label{eq_s2m_r}
S_2^{(m)}(\nu_0) =
\frac 1{\log v} 
\sum_{\mathbf d, \mathbf e \in \mathcal D_K}
    \lambda_{\mathbf d}\lambda_{\mathbf e} 
\sum_{\substack{a \le v \\ (a,q_1)=1 \\ p|a \Rightarrow p \equiv 1 \mod 4}}
    \frac{\mu(a)}{g_2(a)} \log \frac{v}{a} 
\sum_{\substack{x < n \le 2x \\ n \equiv 1 \mod 4 \\ n\equiv \nu_0\bmod W \\ [d_i,e_i]\mid \ell_i(n) \forall i \\ a \mid \ell_m(n) }} r_2(\ell_m(n)).
\end{equation}

The parameter $a$ is supported on integers whose prime factors are all $1 \mod 4$, whereas by the definition of $\mathcal D_K$, each $[d_i,e_i]$ is only divisible by primes that are $3 \mod 4$. Thus the $[d_i,e_i]$'s and $a$ are pairwise coprime. Also, all prime factors of each $[d_i,e_i]$ are larger than $D_0$, whereas each $qb_i < D_0$, so each $[d_i,e_i]$ is coprime to each $qb_i$.

Thus by \Cref{lem:mcgrath-lem-5-3}, (applied with $r = qW\prod_{i\ne m} [d_i,e_i]$ and with $d = a[d_m,e_m]$) we get
\begin{multline}\label{eq_r_ap}
\sum_{\substack{x < n \le 2x \\ n \equiv 1 \mod 4 \\ n \equiv \nu_0 \mod W\\ [d_i,e_i]\mid \ell_i(n) \forall i \\ a \mid \ell_m(n) }} r_2(\ell_m(n))
=
\frac{ g_1\left( qW \prod_{i\neq m}[d_i,e_i]\right) g_2(a[d_m,e_m]) } 
{2 qW a \prod_i[d_i,e_i]}
\pi qx \\
+ O_\epsilon\left(
\left((WR^3v)^{1/2} + x^{1/3}\right)v^{1/2}x^{\epsilon}
\right).
\end{multline}
Recall that $v = x^{\theta_1}$, $R^2 = x^{\theta_2}$, $\theta_1 + \theta_2 < 1/18$, and $W \ll_\ep x^{\ep}$ for all $\ep > 0$. 
Thus
\begin{equation*}
S_2^{(m)}(\nu_0)
=
\frac {\pi x}{\log v} 
\sum_{\mathbf d, \mathbf e \in \mathcal D_K}
    \lambda_{\mathbf d}\lambda_{\mathbf e} 
\sum_{\substack{a \le v \\ (a,q_1) = 1 \\ p|a \Rightarrow p \equiv 1 \mod 4}}
    \frac{\mu(a)}{g_2(a)} \log \frac{v}{a} 
\frac{ g_1\left(q W \prod_{i\neq m}[d_i,e_i]\right) g_2(a[d_m,e_m]) }{2 W a \prod_i[d_i,e_i]}
+ E,
\end{equation*}
where
\begin{align*}
|E| 
&\ll_\ep
\frac 1{\log v} 
\left|\sum_{\mathbf d, \mathbf e \in \mathcal D_K}
 \lambda_{\mathbf d}\lambda_{\mathbf e} \right|
\sum_{\substack{a \le v \\ p|a \Rightarrow p \equiv 1 \mod 4}}
    \frac{1}{g_2(a)} \log \frac{v}{a}
    x^{\frac{1}{3} + (\theta_1 + \theta_2)/2 + \epsilon} \\
& \ll
\lambda_{\mathrm{max}}^2 |\mathcal D_K|^2 v x^{\frac{1}{3} + (\theta_1 + \theta_2)/2 + \epsilon}
\ll_{\ep'}
 x^{\frac{1}{3} + (\theta_1 + \theta_2)\frac{3}{2} + \epsilon'}
\ll x^{\frac{5}{12} + \epsilon'}
\end{align*}
where we used use \eqref{eq:lambda-d-bound} to bound the $\lambda_{\mathrm{max}}$ and \eqref{eq-bound_DK} for $|\mathcal D_K|$. 

It remains to consider the main term, which is given by
\begin{align*}
&\frac { \pi x}{\log v} 
\sum_{\mathbf d, \mathbf e \in \mathcal D_K}
    \lambda_{\mathbf d}\lambda_{\mathbf e} 
\sum_{\substack{a \le v \\ (a,q_1) = 1 \\ p|a \Rightarrow p \equiv 1 \mod 4}}
    \frac{\mu(a)}{g_2(a)} \log\left( \frac{v}{a} \right)
\frac{ g_1\left( qW \prod_{i\neq m}[d_i,e_i]\right) g_2(a[d_m,e_m]) } {2 W a \prod_i[d_i,e_i]} \\
= &\frac{\pi x g_1(qW)}{2W\log v} \sum_{\mathbf d, \mathbf e \in \mathcal D_K} \lambda_{\mathbf d} \lambda_{\mathbf e} \frac{g_1\left(\prod_{i\ne m}[d_i,e_i]\right) g_2([d_m,e_m])}{\prod_i[d_i,e_i]} \sum_{\substack{a \le v \\ (a,q_1) = 1 \\ p|a \Rightarrow p \equiv 1 \mod 4}} \frac{\mu(a)}{a} \log \frac va.
\end{align*}

The inside sum is exactly $X_{x,q_1}$, as defined in \Cref{lem:mcgrath-lem-5-6}. 
Also, all primes dividing $[d_m,e_m]$ are congruent to $3$ mod $4$, so $g_2([d_m,e_m]) = \frac{1}{[d_m,e_m]}$. We can thus rewrite the main term as
\begin{equation}\label{eq_sm_x}
\frac{X_{x,q_1} \pi x g_1(qW) }{2 W \log v}
\sum_{\mathbf d, \mathbf e \in \mathcal D_K}
    \lambda_{\mathbf d}\lambda_{\mathbf e}
\prod_{i \ne m}\left(\frac{g_1([d_i,e_i])}{[d_i,e_i]}\right) 
\frac 1{[d_m,e_m]^2}.
\end{equation}

By \Cref{lem:mcgrath-lem-5-6}, $X_{x,q_1} = (1+o(1))\frac{8A\sqrt{\log v}}{g_1(q_1)\pi}$, and by \Cref{lem:mcgrath-lem-6-6} applied with $f(p) = \frac{g_1(p)}{p}$ and $g(p) = 1/p^2$,
\begin{equation*}
\sum_{\mathbf d, \mathbf e \in \mathcal D_K}
    \lambda_{\mathbf d}\lambda_{\mathbf e}
\prod_{i \ne m}\left(\frac{g_1([d_i,e_i])}{[d_i,e_i]}\right)\frac 1{[d_m,e_m]^2} 
 = S_{\{m\}}=B^{K+1}L_{K;m}(F)(1+o(1)).
\end{equation*}
Then \eqref{eq_sm_x} can be rewritten as
$$
\frac{\pi x g_1(qW)}{2 W \log v}\cdot \frac{8A\sqrt{\log v}}{g_1(q_1)\pi}B^{K+1}L_{K;m}(F)(1+o(1)).
$$
By definition of $g_1$, 
\begin{equation}\label{eq:g1-q3W-to-A-identity}
\frac{g_1(q_3W) \phi(q_3W)}{q_3W} = \frac{1}{2A^2} + O\left(\frac{1}{D_0}\right).
\end{equation}
Recalling also that $B = \frac{2A}{\pi}\frac{\phi(q_3W)}{q_3W}(\log R)^{\tfrac 12}$, we get
\begin{equation*}
S_2^{(m)}(\nu_0) = 
(1+o(1))\frac{4\sqrt{\frac{\log R}{\log v}} B^K x}{\pi W}
L_{K;m}(F),
\end{equation*}
which, along with the fact from Lemma \ref{lem:evaluation-of-sieve-functions-in-terms-of-F} that $L_{K;m}(F) = \frac{\pi^2}{\pi+2} \frac{L_K(F)}{\sqrt K}$, completes the proof.
\end{proof}
\subsection{Estimating \texorpdfstring{$S_3^{(m_1,m_2)}$}{S3}}
\label{subsec:sieve-s-3}
In this section we will prove \Cref{thm:S-i-sums-estimate}, equation \eqref{eq:thm:S-i-sums:S-3}, which we restate in the following proposition.

\begin{proposition}
For fixed $1 \le m_1, m_2 \le K$ lying in the same bin of the tuple $\mathcal L$,
consider the sum $S_3^{(m_1,m_2)}(\nu_0)$ defined by
\[S_3^{(m_1,m_2)}(\nu_0) := \sum_{\substack{x < n \le 2x \\ n \equiv 1 \mod 4 \\ n \equiv \nu_0 \mod W}} \rho(\ell_{m_1}(n))\rho(\ell_{m_2}(n))w_n(\mathcal L).\]
Then
\[S_3^{(m_1,m_2)}(\nu_0) \le (1+o(1)) \frac{64\pi^2 \left(\tfrac{\log R}{\log v}\right) B^K x}{(\pi + 2)^2K W} V L_K(F).\]
\end{proposition}
\begin{proof}
    We begin by expanding the definition of $w_n(\mathcal L)$ to rewrite $S_3^{(m_1,m_2)}(\nu_0)$ as
\begin{equation*}
    S_3^{(m_1,m_2)}(\nu_0) = \sum_{\mathbf d, \mathbf e \in \mathcal D_K} \lambda_{\mathbf d} \lambda_{\mathbf e} \sum_{\substack{x < n \le 2x \\ n \equiv 1 \mod 4 \\ n \equiv \nu_0 \mod W  \\ [d_i,e_i]|\ell_i(n) \forall i}} \rho(\ell_{m_1}(n))\rho(\ell_{m_2}(n)).
\end{equation*}
Upon expanding the definition of $\rho$ and changing the order of summation, the inside sum over $n$ is equal to
\begin{equation*}
    \frac 1{\log^2 v} \sum_{\substack{a,b \le v \\ (a,q_1) = (b,q_1) = 1 \\ p|a,b \Rightarrow p \equiv 1 \mod 4}} \frac{\mu(a)\mu(b)}{g_2(a)g_2(b)} \log \frac va \log \frac vb \sum_{\substack{x < n \le 2x \\ n \equiv 1 \mod 4 \\ n \equiv \nu_0 \mod W  \\ [d_i,e_i]|\ell_i(n) \forall i \\ a|(qn+a_{m_1} + qb_{m_1}) \\ b|(qn+a_{m_2} + qb_{m_2})}} r_2(qn+a_{m_1} + qb_{m_1})r_2(qn+a_{m_2} + qb_{m_2}).
\end{equation*}
Since $\ell_{m_1}(n)$ and $\ell_{m_2}(n)$ are always relatively prime to $q$, we must also always have $(a,q) = (b,q) = 1$; since $a$ and $b$ are only divisible by primes congruent to $1$ mod $4$, this is equivalent to the constraint that $(a,q_1) = (b,q_1) = 1$. 
For the inside sum to be nonzero, $[d_1, e_1], \dots, [d_K,e_K],$ and $W$ must all be pairwise coprime, and each of these must be coprime to both $a$ and $b$. Moreover, if any prime $p$ divides $(a,b)$, then $p|q(b_{m_2}-b_{m_1})$. We thus have
\begin{multline*}
S_3^{(m_1,m_2)}(\nu_0) = 
\sum_{\mathbf d, \mathbf e \in \mathcal D_K}
\frac {\l_{\mathbf d}\l_{\mathbf e}}{\log^2 v} 
\sum_{\substack{a,b \le v \\ (a,q_1) = (b,q_1) = 1 \\ p|a,b \Rightarrow p \equiv 1 \mod 4}} 
\frac{\mu(a)\mu(b)}{g_2(a)g_2(b)} 
\log \frac va \log \frac vb 
\\ \times
\sum_{\substack{\ell_{m_1}(x)< n \le \ell_{m_1}(2x) \\ n \equiv 1 \mod 4 \\ n \equiv \alpha \mod qW\prod_{i\neq m_1,m_2}[d_i,e_i] \\ a[d_{m_1},e_{m_1}]|n \\ b[d_{m_2},e_{m_2}]|n+h}} r_2(n)r_2(n+h)
\end{multline*}
where $\alpha$ is relatively prime to $qW\prod_{i\neq m_1,m_2}[d_i,e_i]$ and $h = q(b_{m_2} - b_{m_1})$.

We now apply Lemma \ref{lem:mcgrath-lem-A-3} to estimate the inner sum, taking $r = qW\prod_{i\ne m_1,m_2}[d_i,e_i]$, $c_1 = a[d_{m_1},e_{m_1}]$, and $c_2=b[d_{m_2},e_{m_2}]$. Note that Lemma \ref{lem:mcgrath-lem-A-3} does not require that $c_1$ and $c_2$ be relatively prime, but that the main term is $0$ unless $(c_1,c_2)|h$. Thus
\begin{align*}
S_3^{(m_1,m_2)}(\nu_0) &= 
\sum_{\mathbf d, \mathbf e \in \mathcal D_K} 
\frac {\lambda_{\mathbf d} \lambda_{\mathbf e}}{\log^2 v} 
\sum_{\substack{a,b \le v \\ (a,q_1) = (b,q_1) = 1 \\ p|a,b \Rightarrow p \equiv 1 \mod 4}} 
\frac{\mu(a)\mu(b)}{g_2(a)g_2(b)} 
\log \frac va \log \frac vb 
\\
&\times \Biggl[ \pi^2 q x
\frac{g_1^2\left(qW\prod_{i\neq m_1,m_2}[d_i,e_i]\right)} {qW\prod_{i\neq m_1,m_2}[d_i,e_i]}
\frac{g_2(a[d_{m_1},e_{m_1}])g_2(b[d_{m_2},e_{m_2}])}{a[d_{m_1},e_{m_1}] b [d_{m_2},e_{m_2}] }
 \\
&\times \sum_{\left(t, 2qW\prod_{i\neq m_1,m_2}[d_i,e_i]\right) = 1}\frac{c_t(h)}{t^2}\frac{(a,t)(b,t)}{\Psi(a,t)\Psi(b,t)}
\prod_{i=1,2}
\frac{\left([d_{m_i},e_{m_i}],t\right)
\chi\left(([d_{m_i},e_{m_i}]^2,t)\right)
}{\Psi\left([d_{m_i},e_{m_i}],t\right)}
\\
&+ O\left(x^{\frac{5}{6} + \epsilon}a^{\frac{1}{2}}b^{\frac{1}{2}}
[d_{m_1},e_{m_1}]^{\frac{1}{2}}[d_{m_2},e_{m_2}]^{\frac{1}{2}}
\right)
\\ 
&+
O\Big(
x^{\frac{3}{4} + \epsilon}
\Big(qW\prod_{i\neq m_1,m_2}[d_i,e_i]\Big)^{\frac{1}{2}}
a b [d_{m_1},e_{m_1}][d_{m_2},e_{m_2}]
\Big)
\Biggr].
\end{align*}

Taking absolute values and noting that $d_{m_i},e_{m_i}\ll R^{\frac{1}{K}}$, the first error term is bounded by
$$
\lambda_{\mathrm{max}}^2 |\mathcal D_K|^2 \left( x^{\frac{5}{6} + \epsilon} v^{\frac{3}{2}} R^{\frac{1}{K}}
+ 
x^{\frac{3}{4} + \epsilon} v^{2} R^{2}\right)
\ll
\lambda_{\mathrm{max}}^2 |\mathcal D_K|^2 x^{\frac{5}{6}}v^{2} R^{2}.
$$
Using \eqref{eq:lambda-d-bound} to bound $\lambda_{\mathrm{max}}$, \eqref{eq-bound_DK} for $|\mathcal D_K|$, and the fact that $\theta_1 + \theta_2 < \frac{1}{18}$, we get that the error term is bounded by
$$
\ll_{\epsilon'}
x^{\frac{5}{6} + 2(\theta_1 + \theta_2) + \epsilon'}
\ll
x^{\frac{17}{18} + \epsilon'}
$$
which is negligible.

We return to the main term which, after some simplification, and recalling that $g_2(p) = \frac{1}{p}$ for $p\equiv 3 \bmod{4}$, becomes
\begin{align}\label{eq-s3m1m2_deabt}
\begin{split}
S_3^{(m_1,m_2)}(\nu_0) &\sim 
\frac{\pi^2 x g_1^2\left(qW\right)}{W \log^2 v}
\sum_{\mathbf d, \mathbf e \in \mathcal D_K} \lambda_{\mathbf d} \lambda_{\mathbf e}
\frac{
\prod_{i\neq m_1,m_2}g_1^2\left([d_i,e_i]\right)
}{
[d_{m_1},e_{m_1}]^2 [d_{m_2},e_{m_2}]^2
\prod_{i\neq m_1,m_2}[d_i,e_i]
}
 \\ &\times
\sum_{\substack{a,b \le v \\ (a,q_1) = (b,q_1) = 1 \\ p|a,b \Rightarrow p \equiv 1 \mod 4}} 
\frac{\mu(a)\mu(b)}{a b } 
\log \frac va \log \frac vb
\sum_{\left(t, 2qW\prod_{i\neq m_1,m_2}[d_i,e_i]\right) = 1}\frac{c_t(h)}{t^2}
\\ &\times
\prod_{i=1,2}
\frac{\left([d_{m_i},e_{m_i}],t\right) 
\chi\left(([d_{m_i},e_{m_i}]^2,t)\right)
}{\Psi\left([d_{m_i},e_{m_i}],t\right) 
} \frac{(a,t)(b,t)}{\Psi(a,t)\Psi(b,t)}.
\end{split}
\end{align}

The sum over $t$ is multiplicative, and can thus be written as a product $\Sigma_1 \times \Sigma_3$ where $\Sigma_1$ ranges over integers divisible only by primes congruent to $1$ mod $4$ and $\Sigma_3$ ranges over integers divisible only by primes congruent to $3$ mod $4$. 
That is:
$$
\Sigma_3 = 
\sum_{\substack{t\\ p\mid t \Rightarrow p\equiv 3 \bmod 4 \\ (t,q_3W \prod_{i\neq m_1,m_2}[d_i,e_i])=1}}
\frac{c_t(h)}{t^2}\prod_{i=1,2}
\frac{\left([d_{m_i},e_{m_i}],t\right) 
\chi\left(([d_{m_i},e_{m_i}]^2,t)\right)
}{\Psi\left([d_{m_i},e_{m_i}],t\right) 
},
$$
and
$$
\Sigma_1 =
\sum_{\substack{t \\ (t,q_1) = 1\\ p\mid t \Rightarrow p\equiv 1 \mod 4}}
\frac{c_t(h)}{t^2} \frac{(a,t)(b,t)}{\Psi(a,t)\Psi(b,t)}.
$$
We have used the fact that in $\Sigma_3$ we have $(a,t) = (b,t) = \Psi(a,t) = \Psi(b,t) = 1$, and in $\Sigma_1$ we have $\left([d_{m_i},e_{m_i}],t\right) = \chi\left(([d_{m_i},e_{m_i}]^2,t)\right) = \Psi\left([d_{m_i},e_{m_i}],t\right) = 1$.

We begin by considering $\Sigma_3$. Using the definition $c_t(h) = \sum_{t_1|(t,h)} \mu(t/t_1) t_1$, and swapping sums and relabeling via $t_2 = t/t_1$, we get
\begin{equation*}
\Sigma_3
=
\sum_{\substack{t_1,t_2\\p|t_1t_2 \Rightarrow p \equiv 3 \mod 4 \\ (t_1t_2,q_3W\prod_{i\neq m_1,m_2}[d_i,e_i]) = 1\\t_1 \mid h}}
\frac{t_1 \mu(t_2)}{(t_1t_2)^2}
\prod_{i=1,2}
\frac{\left([d_{m_i},e_{m_i}],t_1t_2\right) 
\chi\left(([d_{m_i},e_{m_i}]^2,t_1t_2)\right)
}{\Psi\left([d_{m_i},e_{m_i}],t_1t_2\right) 
}.
\end{equation*}
Any prime $p|h$ with $p \equiv 3 \mod 4$ divides $q_3W$, and $t_1|h$ must be co-prime with $q_3W$, so $t_1 = 1$.
Furthermore, $t_2$ is squarefree due to the term $\mu(t_2)$, which implies that $\Psi([d_{m_i},e_{m_i}],t_2) = 1$. Thus
\begin{align*}
\Sigma_3
&=
\sum_{\substack{t_2\\p|t_2 \Rightarrow p \equiv 3 \mod 4\\ (t_2,q_3W\prod_{i\neq m_1,m_2}[d_i,e_i]) = 1}}
\frac{ \mu(t_2)}{t_2^2}
\prod_{i=1,2}
\left([d_{m_i},e_{m_i}],t_2\right) 
\chi\left(([d_{m_i},e_{m_i}]^2,t_2)\right) \\
&=
\prod_{\substack{p\nmid q_3W\prod [d_i,e_i] \\ p\equiv 3 \mod 4}}
\left(1 - \frac{1}{p^2}\right)
\prod_{\substack{p\nmid q_3W\prod_{i\neq m_1,m_2} [d_i,e_i] \\ p\mid \prod_{i=1,2}[d_{m_i},e_{m_i}] \\ p\equiv 3 \mod 4}}
\left(1 + \frac{1}{p}\right).
\end{align*}
Since $q_3W$ is the product of all primes congruent to 3 mod 4 smaller than $D_0$, we have
\begin{align*}
\Sigma_3 &= 
\prod_{\substack{p>D_0 \\ p\equiv 1 \bmod 4}}\left(1 - \frac{1}{p^2}\right)
\prod_{i\neq m_1,m_2}\left(\prod_{p\mid [d_{i},e_{i}]} \left(1 - \frac{1}{p^2}\right)^{-1}\right)
g_1([d_{m_1},e_{m_1}])g_1([d_{m_2},e_{m_2}]) \\
&= \left(1 + o(1)\right) 
g_1([d_{m_1},e_{m_1}])g_1([d_{m_2},e_{m_2}])
\prod_{i\neq m_1,m_2}f\left([d_i,e_i]\right),
\end{align*}
where $f(n) = \prod_{p\mid n}\left(1 - \frac{1}{p^2}\right)^{-1}$.

Plugging this back into \eqref{eq-s3m1m2_deabt} we get
\begin{multline*}
S_3^{(m_1,m_2)}(\nu_0) \sim 
\frac{\pi^2 x g_1^2\left(qW\right)}{W \log^2 v}
\sum_{\mathbf d, \mathbf e \in \mathcal D_K} \lambda_{\mathbf d} \lambda_{\mathbf e}
\frac{
\prod_{i\neq m_1,m_2}g_1^2\left([d_i,e_i]\right)f\left([d_i,e_i]\right)
\prod_{i=1,2}g_1\left([d_{m_i},e_{m_i}]\right)
}{
[d_{m_1},e_{m_1}]^2 [d_{m_2},e_{m_2}]^2
\prod_{i\neq m_1,m_2}[d_i,e_i]
}
 \\ \times
\sum_{\substack{a,b \le v \\ (a,q_1) = (b,q_1) = 1 \\ p|a,b \Rightarrow p \equiv 1 \mod 4}} 
\frac{\mu(a)\mu(b)}{a b } 
\log \frac va \log \frac vb
\sum_{\substack{t \\ (t,q_1) = 1\\ p\mid t \Rightarrow p\equiv 1 \mod 4}}
\frac{c_t(h)}{t^2} \frac{(a,t)(b,t)}{\Psi(a,t)\Psi(b,t)}.
\end{multline*}

In the notation of \Cref{lem:mcgrath-lem-6-6}, the sum over $\mathbf d$ and $\mathbf e$ (which is independent of $a,b,t$) is of the form $S_J$ for $J = \{m_1,m_2\}$, so that by Lemma \ref{lem:mcgrath-lem-6-6} we get
\begin{multline*}
S_3^{(m_1,m_2)}(\nu_0) \sim 
\frac{\pi^2 x g_1^2\left(qW\right)}{W \log^2 v}
B^{K+2}L_{K;m_1,m_2}(F)
 \\ \times
\sum_{\substack{a,b \le v \\ (a,q_1) = (b,q_1) = 1 \\ p|a,b \Rightarrow p \equiv 1 \mod 4}} 
\frac{\mu(a)\mu(b)}{a b } 
\log \frac va \log \frac vb
\sum_{\substack{t \\ (t,q_1) = 1\\ p\mid t \Rightarrow p\equiv 1 \mod 4}}
\frac{c_t(h)}{t^2} \frac{(a,t)(b,t)}{\Psi(a,t)\Psi(b,t)}.
\end{multline*}

Switching the order of summation gives
\begin{multline}\label{eq-s3m1m2_tab}
S_3^{(m_1,m_2)}(\nu_0) \sim 
\frac{\pi^2 x g_1^2\left(qW\right)}{W \log^2 v}
B^{K+2}L_{K;m_1,m_2}(F)
 \\ \times
\sum_{\substack{t \\ (t,q_1) = 1\\ p\mid t \Rightarrow p\equiv 1 \mod 4}}
\frac{c_t(h)}{t^2}
\sum_{\substack{a,b \le v \\ (a,q_1) = (b,q_1) = 1 \\ p|a,b \Rightarrow p \equiv 1 \mod 4}} 
\frac{\mu(a)\mu(b)}{a b } 
 \frac{(a,t)(b,t)}{\Psi(a,t)\Psi(b,t)}
 \log \frac va \log \frac vb.
\end{multline}

Denoting the inner sum as $\Sigma_{a,b}(t)$, we can write
$$
\Sigma_{a,b}(t) = 
\Biggr(\sum_{\substack{a \le v \\ (a,q_1)= 1 \\ p|a \Rightarrow p \equiv 1 \mod 4}} 
\frac{\mu(a)(a,t)}{a\Psi(a,t)} 
\log \frac va \Biggr)^2.
$$
The calculations from \cite[Lemma 6]{MR0294281-Hooley-intervals-sums-of-squares} (along much the same lines as Lemma \ref{lem:mcgrath-lem-5-6}) imply that
$$
\Sigma_{a,b}(t) = 
\left(\frac{8A}{\pi g_1(q_1)} C(t) \log^{\frac{1}{2}} v 
+ O\left(
\frac{t^{\frac{1}{4}}}{\log^{\frac{1}{2}} v}
\right)\right)^2,
$$
where $C(t)$ is the constant in \cite[Lemma 6]{MR0294281-Hooley-intervals-sums-of-squares}, given by
\begin{equation*}
    C(t) = \begin{cases}\prod_{\substack{p\mid t \\ p\equiv 1 \bmod 4}}\left(2 - \frac{1}{p}\right)^{-1} &\text{ if }p|t \text{ and }p \equiv 1 \mod 4\text{ implies } p^2|t \\ 0 &\text{ otherwise.} \end{cases}
\end{equation*}
Note that in \cite[Lemma 6]{MR0294281-Hooley-intervals-sums-of-squares}, the Landau--Ramanujan constant is normalized as $\sqrt{2}A$, and that the statements of \cite[Lemma 5]{MR0294281-Hooley-intervals-sums-of-squares} and \cite[Lemma 6]{MR0294281-Hooley-intervals-sums-of-squares} are missing another factor of $\sqrt{2}$.

Plugging this estimate back into \eqref{eq-s3m1m2_tab} we get
\begin{equation*}
S_3^{(m_1,m_2)}(\nu_0) \sim 
\frac{64 A^2  x g_1^2\left(q_3W\right) B^{K+2}}{W \log v}
L_{K;m_1,m_2}(F)
\Biggr(\sum_{\substack{t \\ (t,q_1) = 1\\ p\mid t \Rightarrow p\equiv 1 \mod 4}}
\frac{c_t(h)C^2(t)}{t^2}
+
O\left(E\right)\Biggr),
\end{equation*}
where 
\begin{multline*}  
E \ll \frac{1}{\log v}
\sum_{\substack{t \\ (t,q_1) = 1\\ p\mid t \Rightarrow p\equiv 1 \mod 4}}
\frac{|c_t(h)|C(t)}{t^{\frac{7}{4}}}
+
\frac{1}{\log^2 v}\sum_{\substack{t \\ (t,q_1) = 1\\ p\mid t \Rightarrow p\equiv 1 \mod 4}}
\frac{|c_t(h)|}{t^{\frac{3}{2}}}
\\ \ll 
\frac{1}{\log v}
\sum_{\substack{t \\ (t,q_1) = 1\\ p\mid t \Rightarrow p\equiv 1 \mod 4}}
\frac{|c_t(h)|}{t^{\frac{3}{2}}}.
\end{multline*}

From \cite[Eq 15]{MR0294281-Hooley-intervals-sums-of-squares}, we have that $E \ll \frac{1}{\log v} \sigma_{-\frac{1}{2}}(h)$, which implies $E \ll \frac{\log\log x}{\log x}$ since $h \leq \eta (\log x)^{\frac{1}{2}}$.
As for the main term, from \cite[Eq 18]{MR0294281-Hooley-intervals-sums-of-squares} we have
\begin{equation*}
\sum_{\substack{t \\ (t,q_1) = 1\\ p\mid t \Rightarrow p\equiv 1 \mod 4}}
\frac{c_t(h)C^2(t)}{t^2}
=
\prod_{\substack{p^\beta \| h\\ p\equiv 1 \mod 4 \\ p\nmid q_1}}
\left(1 + \frac{1}{(2p-1)^2}\left(1 - \frac{1}{p^{\beta -1}} - \frac{1}{p^\beta}\right)\right).
\end{equation*}

Thus
\begin{equation*}
S_3^{(m_1,m_2)} \sim 
\frac{64 A^2  x g_1^2\left(q_3W\right) B^{K+2}}{W \log v}
\prod_{\substack{p^\beta \| b_{m_2} - b_{m_1}\\ p\equiv 1 \mod 4 \\ p\nmid q_1}}
\left(1 + \frac{1}{(2p-1)^2}\left(1 - \frac{1}{p^{\beta -1}} - \frac{1}{p^\beta}\right)\right)
L_{K;m_1,m_2}(F).
\end{equation*}

The product over $p^\beta \|(b_{m_2}-b_{m_1})$ is bounded above by $V$ (defined in \eqref{eq:defn-of-V}), so that
\begin{equation*}
    S_3^{(m_1,m_2)}(\nu_0) \le (1+o(1))\frac{64 A^2 x g_1^2\left(q_3W\right) B^{K+2}}{W \log v} V L_{K;m_1,m_2}(F).
\end{equation*}
Finally, using the identity that $B = \frac{2A}{\pi}\frac{\phi(q_3W)(\log R)^{1/2}}{q_3W}$ as well as applying \eqref{eq:g1-q3W-to-A-identity} and Lemma \ref{lem:evaluation-of-sieve-functions-in-terms-of-F} completes the proof.

\end{proof}
\subsection{Estimating \texorpdfstring{$S_4^{(m)}$}{S4}}
\label{subsec:sieve-s-4}

In this section we will prove \Cref{thm:S-i-sums-estimate}, equation \eqref{eq:thm:S-i-sums:S-4}, which we restate in the following proposition.
\begin{proposition}
For fixed $1 \le m \le K$, consider the sum $S_4^{(m)}(\nu_0)$ defined by
\[S_4^{(m)}(\nu_0) = \sum_{\substack{x < n \le 2x \\ n \equiv 1 \mod 4 \\ n\equiv \nu_0 \bmod W}} \rho^2(\ell_m(n))w_n(\mathcal L).\]
Then
\[S_4^{(m)}(\nu_0) = (1+o(1)) 
\frac{8\pi \sqrt{\frac{\log R}{\log v}} \left(\frac{\log x}{\log v} + 1 \right)B^{K}x}
{(\pi + 2)\sqrt{K} W} 
\prod_{\substack{p \equiv 1 \bmod 4 \\  p\nmid q_1}}\left( 1 + \frac{1}{(2p-1)^2}\right)
L_K(F).\]
\end{proposition}
\begin{proof}
We first expand the definitions of $w_n(\mathcal L)$ and $\rho^2(\ell_m(n))$ and swap the order of summation to write
\begin{equation}\label{eq-sum1for_s4}
S_4^{(m)}(\nu_0) = \frac 1{\log^2 v} 
\sum_{\mathbf d, \mathbf e \in \mathcal D_K}  
    \lambda_{\mathbf d}\lambda_{\mathbf e}
    \sum_{\substack{a,b \le v \\ (a,q_1) = (b,q_1) = 1 \\ p|ab \Rightarrow p \equiv 1 \mod 4}} 
        \frac{\mu(a)\mu(b)}{g_2(a)g_2(b)} \log\frac va \log \frac vb
        \sum_{\substack{x < n \le 2x \\ n \equiv 1 \mod 4 \\ n\equiv \nu_0 \bmod W \\ [d_i, e_i]\mid \ell_i(n) \forall i \\ [a,b]\mid \ell_m(n) \\ }} 
            r_2^2(\ell_m(n)).
\end{equation}
The quantities $W, [d_1,e_1], \dots, [d_k,e_k]$, $q$, and $[a,b]$ must be pairwise coprime because of the support of $\mathcal D_K$. We use \Cref{lem:mcgrath-lem-5-5} in order to evaluate the inner sum; we will apply \Cref{lem:mcgrath-lem-5-5} with $r = qW\prod_{i \ne m} [d_i,e_i]$ and $d = [a,b][d_m,e_m]$. The sum \eqref{eq-sum1for_s4} can then be written as
\begin{multline}\label{eq-sum2for_s4}
\frac 1{\log^2 v}
\sum_{\mathbf d, \mathbf e \in \mathcal D_K} 
\lambda_{\mathbf d}\lambda_{\mathbf e}
\sum_{\substack{a,b \le v \\ p|ab \Rightarrow p \equiv 1 \mod 4}}
    \frac{\mu(a)\mu(b)}{g_2(a)g_2(b)} \log\frac va \log \frac vb
\times \\
    \left(
    \frac{g_3(r)g_4(d)}{rd} 
    \left( \log qx + A_2 + 2\sum_{p\mid r}g_5(p) -2\sum_{p\mid d}g_6(p)\right) qx
    + O_\epsilon\left( (qx)^{\frac{3}{4} + \theta_2 + \varepsilon}\right)
    \right).
\end{multline}

Taking absolute values, the error term from \eqref{eq-sum2for_s4} is bounded by
$$
\ll_{\ep} \lambda_{\mathrm{max}}^2|\mathcal D_K|^2v^2 (qx)^{\frac{3}{4} + \theta_2 + \varepsilon}
\ll_{\ep'}
 x^{\frac{3}{4} + 2(\theta_2 + \theta_1) + \varepsilon'},
$$
where we used \eqref{eq:lambda-d-bound} to bound the $\lambda_{\mathrm{max}}$ and \eqref{eq-bound_DK} for $|\mathcal D_K|$. 
This error term is negligible, since $\theta_1+\theta_2 < \frac{1}{18}$.

We now evaluate the main term via a process that is identical to the one in \cite[Proposition 6.2, part (iv)]{MR4498475-McGrath}.
Using the notation from \Cref{lem:mcgrath-lem-5-6}, the main term of \eqref{eq-sum2for_s4} is
$$
(1 + o(1))
\frac{g_3(qW) B^{K+1} qx}{W \log^2 v}
\left(
    Z_{x,q_1}^{(1)}\log x - 2Z_{x,q_1}^{(2)}
\right)
L_{K;m}(F).
$$
By \Cref{lem:mcgrath-lem-5-6}, this is equal to 
\begin{align*}
(1 + o(1))&
\frac{g_3(qW) B^{K+1} qx}{W \log^2 v} 
\frac{8 A g_7(q_1) }{\phi(q_1)g_1(q_1)} \sqrt{\log v}(\log x + \log v) \prod_{\substack{p \equiv 1 \bmod 4 \\  p\nmid q_1}}\left( 1 + \frac{1}{(2p-1)^2}\right)
L_{K;m}(F)\\
=
(1 + o(1))&
\frac{8 A g_3(qW)g_7(q_1)B^{K+1}\log^{\frac{3}{2}}v \left(\frac{\log x}{\log v} + 1 \right) qx} 
{qW \phi(q_1)g_1(q_1) \log^2 v }   \prod_{\substack{p \equiv 1 \bmod 4 \\  p\nmid q_1}}\left( 1 + \frac{1}{(2p-1)^2}\right) L_{K;m}(F).
\end{align*}

Recalling the definition of $B$ and equation\eqref{eq:g1-q3W-to-A-identity}, we get that
\begin{align*}    
&S_4^{(m)}(\nu_0)
\\&
= (1+o(1))
\frac{g_3(q_1)g_7(q_1)}{\phi(q_1)g_1(q_1)}
\frac{8 B^{k}\sqrt{\frac{\log R}{\log v}} \left(\frac{\log x}{\log v} + 1 \right)x}
{W \pi }
\prod_{\substack{p \equiv 1 \bmod 4 \\  p\nmid q_1}}\left( 1 + \frac{1}{(2p-1)^2}\right)
L_{K;m}(F).
\end{align*}
Observing that the factors dividing $q_1$ cancel and applying the identity from Lemma \ref{lem:evaluation-of-sieve-functions-in-terms-of-F} that $L_{K;m}(F) = \frac{\pi^2}{\pi +2}\frac{L_K(F)}{\sqrt K}$ completes the proof..

\end{proof}
\subsection{Estimating \texorpdfstring{$S_5^{(m)}$}{S5}}
\label{subsec:sieve-s-5}

In this section we will prove \Cref{thm:S-i-sums-estimate}, equation \eqref{eq:thm:S-i-sums:S-5}, which we restate in the following proposition.
\begin{proposition}\label{prop:S-5-estimate}
Let $\xi > 0$ be a constant with $\xi < \tfrac 1K$. For fixed $1 \le m \le K$, define $S_5^{(m)}(\nu_0)$ to be the sum
\begin{equation*}
    S_5^{(m)}(\nu_0) := \sum_{\substack{x < n \le 2x \\ n \equiv 1 \mod 4 \\ n \equiv \nu_0 \mod W}} \sum_{\substack{p < x^\xi \\ p \equiv 3 \mod 4 \\ p|\ell_m(n)}} w_n(\mathcal L).
\end{equation*}
Then 
\begin{equation*}
    S_5^{(m)}(\nu_0) \ll \frac{K^2\xi^2}{\theta_2^2} \frac{B^Kx}{W}L_K(F).
\end{equation*}
\end{proposition}

The proof of the proposition relies on the following lemma, which we state and prove before turning to the main proof of Proposition \ref{prop:S-5-estimate}.
\begin{lemma}\label{lem:prop-S-5-evaluating-T}
Define
\begin{equation*}
T = \frac xW\sum_{\substack{\u,\v \in \mathcal D_K}}
\frac{y_{\u}y_{\v}}{\phi(u)\phi(v)}
\prod_{\substack{p \mid uv}}\left|\sigma_{p}(\u,\v)\right|,
\end{equation*}
where
\begin{equation*}
\sigma_{p}(\u,\v) 
=
\sum_{\substack{\mathbf d \mid \u,\; \mathbf e \mid \v \\ d_i,e_i \mid p \; \forall i }}
\frac{\mu(d)\mu(e) d e}{[\mathbf d, \mathbf e]}
=
\begin{cases}
    p - 1 & \text{if } p\mid(\u,\v) \\
    0    & \text{if } p\mid uv,\,p\nmid (\u,\v)\\
    1 & \text{if } p \nmid uv
\end{cases}.
\end{equation*}
Then
\begin{equation*}
    T \ll 
    \frac{x}{W} \left(\frac{e^{-\gamma /2}}{\Gamma(1/2)}\right)^K
\left(\frac{\log R}{\log D_0}\right)^{K/2} L_K(F).
\end{equation*}
\end{lemma}
\begin{proof}
First note that if $u \ne v$, then for some prime $p$, $\sigma_{p}(\u,\v) = 0$, so these terms do not contribute. Thus for a fixed $\u \in \mathcal D_K$,
\begin{equation*}
    \sum_{\substack{\v \in \mathcal D_K}} \frac{\prod_{p|uv}|\sigma_{p}(\u,\v)|}{\phi(v)} 
    = 
    \prod_{p|u} \Big(\sum_{\substack{\mathbf w \in \mathcal D_K \\ w_i|p \forall i}} \frac{|\sigma_{p}(\u,\mathbf w)|}{\phi(w)}\Big) = 1. 
\end{equation*}
This, along with the bound that $y_{\u}y_{\v} \ll y_{\u}^2 + y_{\v}^2$, implies that
\begin{align*}
    T \ll \frac{x}{W}\sum_{\u,\v \in \mathcal D_K} \frac{y_{\u}^2 + y_{\v}^2}{\phi(u)\phi(v)}\prod_{\substack{p|uv }} |\sigma_{p}(\u,\v)| 
    &\ll \frac{x}{W}\sum_{\u\in\mathcal D_K} \frac{y_{\u}^2} {\phi(u)}
    \left( \sum_{\substack{\v \in \mathcal D_K}} \frac{\prod_{p|uv}|\sigma_{p}(\u,\v)|}{\phi(v)} \right)
     \\
    &\ll \frac{x}{W}\sum_{\u\in\mathcal D_K} \frac{y_{\u}^2} {\phi(u)}.
\end{align*}

By \Cref{lem:evaluating-sieve-sums-lemma-8.4-maynard}, 
\begin{align*}
\sum_{\r \in \mathcal D_K} \frac{y_{\r}^2}{\phi(r)} &= 
\sum_{\r \in \mathcal D_K} 
\frac{\mu(r)^2}{\phi(r)}\prod_{i=1}^K g\left(K\frac{\log r_1}{\log R}\right)^2 
\ll 
\left(\frac{e^{-\gamma /2}}{\Gamma(1/2)}\right)^K
\left(\frac{\log R}{\log D_0}\right)^{K/2}
L_K(F).
\end{align*}
as desired.
\end{proof}

We are now ready to prove Proposition \ref{prop:S-5-estimate}.

\begin{proof}[Proof of Proposition \ref{prop:S-5-estimate}] 
Expanding the square and swapping the order of summation gives
\begin{equation*}
    S_5^{(m)}(\nu_0) = \sum_{\substack{p < x^\xi \\ p \equiv 3 \mod 4}} \sum_{\mathbf d, \mathbf e \in \mathcal{D}_K} \lambda_{\mathbf d}\lambda_{\mathbf e} \sum_{\substack{x < n \le 2x \\ n \equiv 1 \mod 4 \\ n \equiv \nu_0 \mod W \\ [d_i,e_i]|\ell_i(n) \\ p|\ell_m(n)}} 1.
\end{equation*}
By choice of $\nu_0 \mod W$, if $p|\ell_m(n)$ then $p > D_0$. Because of the support of $\mathcal{D}_K$, if $\lambda_{\mathbf d} \ne 0$ and $\lambda_{\mathbf e} \ne 0$, then any prime $p \equiv 3 \mod 4$ can divide at most one of the $\ell_i(n)$. 
Thus if $p|\ell_m(n)$, then $p \nmid \ell_i(n)$ for all $i \ne m$, which implies that $(d_ie_i,p) = 1$ for all $i \ne m$. 
By the Chinese remainder theorem, the inner sum is of the form $\frac xQ + O(1)$, where $Q = 4W[d_m,e_m,p]\prod_{i \ne m} [d_i,e_i].$ 
Note that $Q < 4WR^2x^\xi$, and for any fixed $Q$ there are $O(\tau_{3k+4}(Q))$ choices of $\mathbf d,\mathbf e, p$ giving rise to the modulus $Q$. 
Thus the error term from the Chinese remainder theorem application and \eqref{eq:lambda-d-bound} makes a contribution that is
\begin{align*}
&\ll \sum_{Q < 4WR^2x^\xi} \tau_{3k+4}(Q) \l_{\mathrm{max}}^2 \ll_\ep  R^2x^{\xi}x^\ep = x^{\theta_2 + \xi + \ep},
\end{align*}
which is negligible for $\xi$ small.

The remaining term is given by
\begin{align*}
&\frac x{4W} 
\sum_{\substack{D_0 < p < x^\xi \\ p \equiv 3 \mod 4}} \frac 1p 
\sum_{\substack{\mathbf d, \mathbf e \in \mathcal{D}_K \\ (d_ie_i,p) = 1 \forall i \ne m}}
\frac{\lambda_{\mathbf d}\lambda_{\mathbf e}p}{[d_m,e_m,p]\prod_{i\ne m}[d_i,e_i]}.
\end{align*}
Expanding the definitions of $\lambda_{\mathbf d}, \lambda_{\mathbf e}$ and rearranging, this is
\begin{align}\label{eq-s5_main_term}
\begin{split} 
&\frac{x}{4W}\sum_{\substack{D_0 < p < x^\xi \\ p \equiv 3 \mod 4}} \frac 1p \sum_{\substack{\mathbf d, \mathbf e \in \mathcal{D}_K \\ (d_ie_i,p) = 1 \forall i \ne m}} \frac{\mu(d)\mu(e)dep}{[d_m,e_m,p]\prod_{i\ne m}[d_i,e_i]} \sum_{\substack{\mathbf r, \mathbf s \in \mathcal{D}_K \\ \mathbf d|\mathbf r,\mathbf e|\mathbf s}} \frac{y_{\mathbf r} y_{\mathbf s}}{\phi(r)\phi(s)} \\
= 
&\frac{x}{4W}\sum_{\substack{D_0 < p < x^\xi \\ p \equiv 3 \mod 4}} \frac 1p
\sum_{\substack{\mathbf r, \mathbf s \in \mathcal{D}_K}} \frac{y_{\mathbf r} y_{\mathbf s}}{\phi(r)\phi(s)} 
\sum_{\substack{\mathbf d, \mathbf e \in \mathcal{D}_K \\ (d_ie_i,p) = 1 \forall i \ne m \\ \mathbf d|\mathbf r, \mathbf e|\mathbf s}} \frac{\mu(d)\mu(e)dep}{[d_m,e_m,p]\prod_{i\ne m}[d_i,e_i]}.
\end{split}
\end{align}
The inside sum is multiplicative over $p'|rs$; write $\sigma_{p'}(\mathbf r, \mathbf s,p)$ for the $p'$ component. If $p'\ne p$, then
\begin{equation*}
    \sigma_{p'}(\r,\s,p) 
    =\sum_{\substack{\mathbf d, \mathbf e \in \mathcal{D}_K \\ d_i,e_i \mid p' \; \forall i \\ \mathbf d|\mathbf r, \mathbf e|\mathbf s}} 
    \frac{\mu(d)\mu(e)dep}{[d_m,e_m,p]\prod_{i\ne m}[d_i,e_i]}
    = \begin{cases} p'-1 &\text{ if } p'|(\r,\s) \\ -1 &\text{ if } p'|r, p'|s, p'\nmid (\r,\s) \\ 0 &\text{ otherwise,} \end{cases}
\end{equation*}
where we recall that $(\r,\s) = \prod_{i}(r_i,s_i)$.
If $p'= p$, then
\begin{equation*}
    \sigma_p(\r,\s,p) =
    \sum_{\substack{\mathbf d, \mathbf e \in \mathcal{D}_K \\ d_m,e_m \mid p \\ d_i=e_i=1 \; \forall i \neq m \\ \mathbf d|\mathbf r, \mathbf e|\mathbf s}} 
    \frac{\mu(d)\mu(e)dep}{[d_m,e_m,p]\prod_{i\ne m}[d_i,e_i]}
    =\begin{cases} (p-1)^2 &\text{ if } p|(r_m,s_m) \\ -(p-1) &\text{ if } p|r_ms_m, p\nmid(r_m,s_m) \\
    1 &\text{ otherwise.}\end{cases}.
\end{equation*}

Let $f_{\mathbf u}(\r) = (r_1,\dots, r_m/(r_m,p), \dots, r_k)$ be the vector formed by removing a possible factor of $p$ from $r_m$. Then our expression \eqref{eq-s5_main_term} can be written as
\begin{align*}
&\frac{x}{4W}\sum_{\substack{D_0 < p < x^\xi \\ p \equiv 3 \mod 4}} \frac 1p\sum_{\substack{\mathbf{u}, \mathbf s \in \mathcal{D}_K \\ (u_m,p) = 1}} \frac{y_{\mathbf s}}{\phi(s)} \sum_{\substack{\r \\ f_{\mathbf u}(\r) = \u}} \frac{ y_{\r}}{\phi(r)} \prod_{p'|rs} \sigma_{p'}(\r,\s,p). \\
\end{align*}
We split the sum above into several parts.
Let $\Sigma_1$ be the summands where $p|u_j$ for some $j \ne m$, and let $\Sigma_2$ be the summands where $p\nmid u_i$ for all $i$. 
Define 
\begin{equation}\label{eq:prop-5-proof-def-of-T}
T = \frac xW \sum_{\substack{\u,\v \in \mathcal{D}_K}}
\frac{y_{\u}y_{\v}}{\phi(u)\phi(v)}
\prod_{\substack{p' \mid uv}}\left|\sigma_{p'}(\u,\v)\right|,
\end{equation}
where
\begin{equation*}
\sigma_{p'}(\u,\v) 
=
\sum_{\substack{\mathbf d \mid \u,\; \mathbf e \mid \v \\ d_i,e_i \mid p' \; \forall i }}
\frac{\mu(d)\mu(e) d e}{[\mathbf d, \mathbf e]}.
\end{equation*}

We will bound both $\Sigma_1$ and $\Sigma_2$ in terms of $T$, showing first that $\Sigma_1 \ll \frac 1{D_0}T$. We have
\begin{equation*}
\Sigma_1 = 
\frac x{4W}\sum_{\substack{D_0 < p < x^\xi \\ p \equiv 3 \mod 4}} \frac 1p
\sum_{i\neq m} \sum_{\substack{\mathbf u, \mathbf s \in \mathcal{D}_K \\ (u_m,p) = 1 \\ p|u_i}}
\frac{y_{\mathbf s}}{\phi(s)} 
\sum_{\substack{\r \\ f_{\mathbf u}(\r) = \u}} 
\frac{y_{\r}}{\phi(r)} 
\sigma_p(\r,\s, p)\prod_{\substack{p'\mid rs \\ p' \ne p}}\sigma_{p'}(\r,\s,p).
\end{equation*}
Given $\mathbf u$, there is only one vector $\r$ such that $f_{\mathbf u}(\r) = \u$; namely, $\r = \u$. Thus
\begin{equation*}
\Sigma_1 = 
\frac x{4W}\sum_{\substack{D_0 < p < x^\xi \\ p \equiv 3 \mod 4}} \frac 1p
\sum_{i\neq m} \sum_{\substack{\mathbf u, \mathbf s \in \mathcal{D}_K \\ (u_m,p) = 1 \\ p|u_i}}
\frac{y_{\mathbf s} y_{\u}}{\phi(s)\phi(u)}
\sigma_p(\u,\s, p)\prod_{\substack{p'\mid us \\ p' \ne p}}\sigma_{p'}(\u,\s).
\end{equation*}

Denote by $\u'$ the vector obtained from $\u$ by removing all factors of $p$. Then $\phi(u) = (p-1)\phi(u')$ and $\sigma_p(\u,\s,p) = \sigma_p(\u',\s,p) = 
\mu((s_m,p))\phi((s_m,p))$, which is independent of $\u$ because we already require $(u_m,p) = 1$. Thus
\begin{equation*}
\Sigma_1 
= \frac x{4W}
\sum_{\s \in \mathcal{D}_K}
\frac{y_{\s}}{\phi(s)}
\sum_{\substack{D_0 < p < x^\xi \\ p \equiv 3 \mod 4}} \frac 1p
\sum_{\substack{\u' \in \mathcal{D}_K \\ p\nmid u_i'}}
\frac{1}{\phi(u')}
\sigma_p(\u',\s,p)
\prod_{\substack{p'\mid u's\\ p'\ne p}}\sigma_{p'}(\u',\s)
\sum_{\substack{\u \in \mathcal{D}_K\\ \u\rightarrow\u'}}
\frac{y_{\u}}{(p-1)}.
\end{equation*}

By \Cref{lem:taking-out-one-factor-lemma-8.2}, we have $y_{\u} = y_{\u'}\left( 1+ O\left(K\xi\right)\right)$. By assumption $K\xi \ll 1$, so (recalling that the weights $y_{\r}$ are nonnegative), $y_{\u} \ll y_{\u'}$ and 
\begin{equation*}
\Sigma_1 
\ll \frac xW
(K-1)
\sum_{\s \in \mathcal{D}_K}
\frac{y_{\s}}{\phi(s)}
\sum_{\substack{D_0 < p < x^\xi \\ p \equiv 3 \mod 4}} \frac 1 {p(p-1)}
\sum_{\substack{\u' \in \mathcal{D}_K \\ p\nmid u_i'}}
\frac{ y_{\u'}}{\phi(u')}
|\sigma_p(\u',\s,p)|
\prod_{\substack{p'\mid u's\\ p'\ne p}}|\sigma_{p'}(\u',\s)|.
\end{equation*}

To bound $\Sigma_1$, we now further split it into subsums.
First, let $T_1$ consist of all those terms with $\s$ such that $p \nmid s_i$ for all $i$. 
In this case $\sigma_p(\u',\s, p) = 1$, so
\begin{equation*}
T_1 \ll \frac xW(K-1)
\sum_{\substack{D_0 < p < x^\xi \\ p \equiv 3 \mod 4}} \frac 1 {p(p-1)}
\sum_{\substack{\s \in \mathcal{D}_K \\ p\nmid s_i \; \forall i}}
\frac{y_{\s}}{\phi(s)}
\sum_{\substack{\u'\in \mathcal{D}_K \\ p\nmid u_i' \; \forall i}}
\frac{y_{\u'}}{\phi(u')}
\prod_{\substack{p'\mid u's}}\left|\sigma_{p'}(\u',\s)\right|.
\end{equation*}
Dropping the requirement that $p\nmid s_i$, $p\nmid u'_i$ only increases $T_1$. The sum over $p$ is then independent of the rest of the expression, and converges to a constant that is $\ll \frac 1{D_0}$, which in turn implies that $T_1 \ll \frac K{D_0}T$, where $T$ is defined in \eqref{eq:prop-5-proof-def-of-T}.

Now consider $T_2$, the terms $\s$ in $\Sigma_1$ such that $p\mid s_i$ for some $i\neq m$.
In this case, $\sigma_p(\u',\s, p) = 1$, so
\begin{equation*}
T_2 \ll \frac xW
(K-1)^2
\sum_{D_0 < p < x^\xi} \frac 1 {p(p-1)}
\sum_{\substack{\s \in \mathcal{D}_K \\ p\mid s_1}}
\frac{y_{\s}}{\phi(s)}
\sum_{\substack{\u' \in \mathcal{D}_K \\ p\nmid u_i' \; \forall i}}
\frac{y_{\u'}}{\phi(u')}
\prod_{\substack{p'\mid u's \\ p' \ne p}}\left|\sigma_{p'}(\u',\s)\right|. 
\end{equation*}
Let $\s'$ be the vector obtained by removing the factor of $p$ from $\s$. Once again $y_{\s} \ll y_{\s'}$, so
\begin{equation*}
T_2 \ll 
\frac xW (K-1)^2
\sum_{\substack{D_0 < p < x^\xi \\ p \equiv 3 \mod 4}} \frac 1 {p(p-1)^2}
\sum_{\substack{\s'\in \mathcal{D}_K \\ p\nmid s_i \; \forall i}}
\frac{y_{\s'}}{\phi(s')}
\sum_{\substack{\u'\in \mathcal{D}_K \\ p\nmid u_i' \; \forall i}}
\frac{y_{\u'}}{\phi(u')}
\prod_{\substack{p'\mid u's}}\left|\sigma_{p'}(\u',\s)\right|. 
\end{equation*}
Once more we can remove the constraints that $p\nmid s_i$ and $p\nmid u'_i$ and evaluate the sum over $p$ to get that $T_2 \ll \frac{K^2}{D_0^2} T$.

Finally consider $T_3$, the subsum of $\Sigma_1$ with those $\s$ such that $p\mid s_m$.
In this case $\sigma_p(\u',\s, p) = -(p-1)$. 
A similar computation gives
\begin{equation*}
T_3 \ll \frac xW
\sum_{\substack{D_0 < p < x^\xi \\ p \equiv 3 \mod 4}} \frac 1 {p(p-1)^2}
\sum_{\substack{\s'\in \mathcal{D}_K \\ p\nmid s_i \; \forall i}}
\frac{y_{\s'}}{\phi(s')}
\sum_{\substack{\u'\in \mathcal{D}_K \\ p\nmid u_i' \; \forall i}}
\frac{y_{\u'}}{\phi(u')}
(p-1)
\prod_{\substack{p'\mid u's}}\left|\sigma_{p'}(\u',\s)\right| 
\ll \frac{1}{D_0}T.
\end{equation*}
Thus $|\Sigma_1| \ll T_1 + T_2 + T_3 \ll \frac{K^2}{D_0} T$. 

Now consider $\Sigma_2$, given by
\begin{equation*}
    \Sigma_2 = \frac x{4W} \sum_{\substack{D_0 < p < x^\xi \\ p \equiv 3 \mod 4}} \frac 1p \sum_{\substack{\u,\s \in \mathcal D_K \\ (u_i,p) = 1 \forall i}} \frac{y_{\s}}{\phi(s)} \sum_{\substack{\r \in \mathcal D_K \\ f_{\u}(\r) = \u}} \frac{y_{\r}}{\phi(r)} \prod_{p'|rs} \sigma_{p'}(\r,\s,p).
\end{equation*}
Observe that for fixed $\u, \s \in \mathcal{D}_K$ with $p\nmid u_i$ for all $i$, 
\begin{equation}\label{eq:prop-5-y-u-terms-in-S-2}
\sum_{\substack{\r \in \mathcal{D}_K \\ f_{\u}(\r)=\u}} \frac{\sigma_p(\r,\s,p)}{\phi(r)} = \frac{\mu((s_m,p))\phi((s_m,p))}{\phi(u)}\left(1 -\frac{p-1}{p-1}\right) = 0.
\end{equation}
We substitute $y_{\r} = y_{\u} + (y_{\r}-y_{\u})$ into $\Sigma_2$. By \eqref{eq:prop-5-y-u-terms-in-S-2}, the $y_{\u}$ do not contribute, leaving only the contribution from $(y_{\r}-y_{\u})$. The only terms remaining have $\r \ne \u$, so that $p|r_m$. Thus
\begin{align*}
\Sigma_2= &\frac x{4W}\sum_{\substack{D_0 < p < x^\xi \\ p \equiv 3 \mod 4}} \frac 1p\sum_{\substack{\mathbf r, \mathbf s \in \mathcal{D}_K \\ p|r_m}} \frac{y_{\mathbf s}(y_{\r}-y_{\u})}{\phi(r)\phi(s)}\prod_{p'|rs} \sigma_{p'}(\r,\s,p).
\end{align*}
By running the same argument for $\s$ and a tuple $\v$ obtained from $\s$ by removing a factor of $p$ from $s_m$ (including bounding the terms where $p|v_i$ for some $i \ne m$ by $\frac {K^2}{D_0}T$ using identical arguments to the bound on $\Sigma_1$), we can also replace $y_{\s}$ by $y_{\s}-y_{\v}$. By \Cref{lem:taking-out-one-factor-lemma-8.2}
we have 
\begin{equation*}
(y_{\r}-y_{\u})(y_{\s}-y_{\v}) \ll y_{\u}y_{\v}K^2 \frac{(\log p)^2}{(\log R)^2},
\end{equation*}
so $\Sigma_2$ is given by
\begin{align*}
\Sigma_2 =&\frac x{4W}\sum_{\substack{D_0 < p < x^\xi \\ p \equiv 3 \mod 4}} \frac 1p \sum_{\substack{\r,\s \in \mathcal{D}_K \\ p|(r_m,s_m)}} \frac{(y_{\r}-y_{\u})(y_{\s}-y_{\v})}{\phi(r)\phi(s)}  \prod_{p'|rs} \sigma_{p'}(\r,\s,p) + O\left(\frac{K^2T}{D_0}\right) \\
\ll &\frac{xK^2}{W} \sum_{\substack{D_0 < p < x^\xi \\ p \equiv 3 \mod 4}} \frac 1p\left(\frac{\log p}{\log R}\right)^2 \sum_{\substack{\u, \v \in \mathcal{D}_K \\(uv,p) = 1}} y_{\u}y_{\v} \prod_{\substack{p'|uv \\ p'\ne p}} |\sigma_{p'}(\u,\v)| 
\sum_{\substack{\r,\s \in \mathcal{D}_K \\ r_m = pu_m \\ s_m = pv_m \\ r_i = u_i \forall i\neq m \\ s_i = v_i \forall i\neq m}}
\frac{|\sigma_p(\r,\s,p)|}{\phi(r)\phi(s)} + \frac{K^2T}{D_0}.
\end{align*}
The sum over $\r$ and $\s$ is equal to $\phi(u)^{-1}\phi(v)^{-1}$, so
\begin{align*}
\Sigma_2 &\ll \frac{xK^2}{W} \sum_{\substack{D_0 < p < x^\xi\\ p \equiv 3 \mod 4}} \frac 1p\left(\frac{\log p}{\log R}\right)^2 \sum_{\substack{\u, \v \in \mathcal{D}_K \\ (uv,p) = 1}} \frac{y_{\u}y_{\v}}{\phi(u)\phi(v)} \prod_{\substack{p'|uv \\ p'\ne p}} |\sigma_{p'}(\u,\v)| + \frac{K^2T}{D_0}
\\ & 
\ll \frac{xK^2\xi^2}{W\theta_2^2} \sum_{\substack{\u, \v \in \mathcal{D}_K}} \frac{y_{\u}y_{\v}}{\phi(u)\phi(v)} \prod_{\substack{p'|uv}} |\sigma_{p'}(\u,\v)| + \frac{KT}{D_0} \ll \left(K^2\frac{\xi^2}{\theta_2^2} + \frac{K^2}{D_0}\right) T.
\end{align*}
Altogether, we get that $S_5^{(m)}(\nu_0) \ll \left(\frac{K^2\xi^2}{\theta_2^2} + \frac {K^2}{D_0}\right)T$. The contribution from the $\frac {K^2}{D_0}$ term vanishes as $x$ grows large. The quantity $T$ is evaluated in Lemma \ref{lem:prop-S-5-evaluating-T}, giving
\begin{align*}
S_5^{(m)}(\nu_0)
&\ll
\frac{xK^2\xi^2}{W\theta_2^2}
\left(\frac{e^{-\gamma /2}}{\Gamma(1/2)}\right)^K
\left(\frac{\log R}{\log D_0}\right)^{K/2}
L_K(F).
\end{align*}
From the definition of $B$ and Mertens' theorem we get
$$
\left(\frac{e^{-\gamma /2}}{\Gamma(1/2)}\right)^K
\left(\frac{\log R}{\log D_0}\right)^{K/2}
\sim B^k,
$$
which completes the argument.
\end{proof}
\subsection{Estimating \texorpdfstring{$S_6^{(b)}$}{S6}}
\label{subsec:sieve-s-6}

In this section we will prove \Cref{thm:S-i-sums-estimate}, equation \eqref{eq:thm:S-i-sums:S-6}, which we restate in the following proposition. 

\begin{proposition}
Let $\nu_1$ be a congruence class modulo $q_3^2W^2$ such that $(\ell(\nu_1),q_3^2W^2)$ is a square for all $\ell \in \mathcal L$. Fix $3 < b \le \eta\sqrt{\log x}$ and consider the linear form $\ell^{(b)}(n) := qn+b$. Fix a constant $\xi$ with $0 < \xi < 1/4$, and define 
\[S_6^{(b)}(\nu_1) := \sum_{\substack{x < n \le 2x \\ n \equiv 1 \mod 4 \\ n \equiv \nu_1 \mod q_3^2W^2}} \mathbf 1_{S(\xi)}(\ell^{(b)}(n)) w_n(\mathcal L),\]
where $S(\xi)$ is the set described in \eqref{eq:our-maynard-style-sum}.
Then
\[S_6^{(b)}(\nu_1) \ll_K \frac{x}{4q_3^2W^2}\xi^{-1/2} \left(\frac{\theta_2}{2}\right)^{-1/2} \left(\frac{\log R}{\log D_0}\right)^{\frac{K-1}{2}} L_K(F). \]
\end{proposition}
\begin{proof}
We will apply Selberg's sieve to bound the function $\mathbf 1_{S(\xi)}(\ell(n))$, while also evaluating the sum over sieve weights $w_n(\mathcal L)$. We begin by defining the additional sieve weights.

Recall that $S(\xi)$ denotes the set of integers such that for all primes $p < x^{\xi}$ with $p \equiv 3 \mod 4$, either $p \nmid n$ or $p^2|n$. Thus for each prime $p \equiv 3 \mod 4$, with $D_0 < p < x^{\xi}$, we sieve by the set $\mathcal A_p$ of integers $x \le n \le 2x$ such that $p \mid \ell(n)$ but $p^2\nmid \ell(n)$. The sieving set $\mathcal A_p$ has density function
\begin{align*}
    g(p) &= \begin{cases}
        \frac{1}{p} - \frac{1}{p^2}& p > D_0 \text{ and } p\equiv3 \bmod 4 \\
        0 & \text{otherwise.}
    \end{cases}
\end{align*}
We extend both $g(p)$ and $\mathcal A_p$ multiplicatively to squarefree $d$, so that
$$
|r_d| := |\mathcal{A}_d - g(d) x| \ll \tau(d).
$$

We will use the upper bound Selberg sieve 
\begin{equation*}
    \mathbf 1_{\mathcal S(\xi)}(\ell(n)) \le \sum_{f|\ell(n)} \mu^+(f)|\mathcal A_f|,
\end{equation*}
where 
\begin{equation*}
    \mu^+(f) = \frac{1}{\Tilde{\lambda_{1}}^2}\sum_{[d_0,e_0] = f}\Tilde{\lambda_{d_0}} \Tilde{\lambda_{e_0}},
\end{equation*}
and $\Tilde{\lambda_d}$ is a sequence of weights defined as follows. Define ``diagonalizing vectors'' $\Tilde{y_{r_0}}$ via
\begin{equation*}
\Tilde{y_{r_0}} := \begin{cases} 1 &\text{ if } (r_0,q_3W) = 1, r_0 < x^{\xi}, \text{ and } p|r_0 \Rightarrow p \equiv 3 \mod 4 \\
0 &\text{ otherwise,}\end{cases}
\end{equation*}
and define $\Tilde{\lambda_{d_0}}$ to be
\begin{equation}\label{eq-tilde-lambda}
\Tilde{\lambda_{d_0}} := \mu(d_0) \frac{d_0^2}{\phi(d_0)} \sum_{d_0|r_0} \frac{\Tilde{y_{r_0}}}{\phi(r_0)}.
\end{equation}
By M\"obius inversion, we also have the relation that
\begin{equation*}
    \Tilde{y_{r_0}} = \mu(r_0)\phi(r_0)\sum_{r_0|d_0} \frac{\Tilde{\lambda_{d_0}}\phi(d_0)}{d_0^2}.
\end{equation*}
Note that $\Tilde{\l_{d_0}}$ is supported on squarefree $d_0$ with $(d_0,q_3W) = 1$, $d_0<x^{\xi}$, and $d_0$ only divisible by primes congruent to $3$ mod $4$. Also, with this choice,
\begin{equation}\label{eq-lambda1-bound}
\Tilde{\lambda_{1}} = \sum_{\substack{r_0 < x^{\xi} \\ (r_0,q_3W) = 1 \\ p|r_0 \Rightarrow p \equiv 3 \mod 4}} \frac{1}{\phi(r_0)} \gg \sqrt{\frac{\xi \log x}{\log D_0}}.
\end{equation}

For what follows, we will fix the notation that
\begin{equation*}
\phi_{\omega^*}(n) = n\prod_{p\mid n}\left(1 - \frac{K+1 }{p}\right),
\end{equation*}
and define further ``cross''-diagonalizing vectors
\begin{equation}\label{eq:y-r-r-0-in-terms-of-lambdas}
y_{\r,r_0} := \mu(r_0r)\phi_{\omega^*}(r_0r) \sum_{\substack{\r|\mathbf d \\ r_0|d_0\\ (d_0,d) = 1}} \frac{\l_{\mathbf d}\Tilde{\l}_{d_0}\phi(d_0)}{dd_0^2}, 
\end{equation}
which satisfy the inverse relation that
\begin{equation}\label{eq:lambda-d-tilde-lambda-d0-diagonalization}
\l_{\mathbf d}\Tilde{\l_{d_0}} = \frac{\mu(d_0d)dd_0^2}{\phi(d_0)} \sum_{\substack{\mathbf d|\mathbf r \\ d_0|r_0 \\ (r_0,r) = 1}} \frac{y_{\mathbf r,r_0}}{\phi_{w^*}(r_r)}.
\end{equation}

We are now ready to apply Selberg's sieve, which gives that
\begin{align*}
S_6^{(b)}(\nu_1) 
&\leq 
\frac{1}{\Tilde{\lambda_{1}}^2}
\sum_{d_0,e_0, \mathbf d, \mathbf e} \Tilde{\lambda_{d_0}} \Tilde{\lambda_{e_0}} \l_{\mathbf d} \l_{\mathbf e} \sum_{\substack{x < n \le 2x \\ [d_i,e_i]|\ell_i(n) \\ p\mid d_0e_0 \Rightarrow p\|\ell(n) \\ n \equiv 1 \mod 4 \\ n \equiv \nu_1 \mod q_3^2W^2}} |\mathcal A_{[d_0,e_0]}| \\
&= 
\frac{1}{\Tilde{\lambda_{1}}^2}
\sum_{\substack{d_0,e_0, \mathbf d, \mathbf e \\ (d_0e_0,de) = 1}} \Tilde{\lambda_{d_0}} \Tilde{\lambda_{e_0}} \l_{\mathbf d} \l_{\mathbf e} 
\left(\frac{xg([d_0,e_0])}{4q_3^2W^2  \prod_{i=1}^K [d_i,e_i]} + O(\tau([d_0,e_0]))\right). \\
\end{align*}
The contribution from the $O(\tau([d_0,e_0]))$ term satisfies
\begin{align*}
\frac 1{\Tilde{\lambda_1}^2} \sum_{\substack{d_0,e_0,\mathbf d, \mathbf e \\ (d_0e_0,de) = 1}} |\Tilde{\lambda_{d_0}}||\Tilde{\lambda_{e_0}}||\lambda_{\mathbf d}||\l_{\mathbf e}| \tau([d_0,e_0]) &\ll R^{2+o(1)} \frac 1{\Tilde{\lambda_1}^2} \sum_{\substack{d_0,e_0}} |\Tilde{\lambda_{d_0}}||\Tilde{\lambda_{e_0}}| \tau([d_0,e_0]). \\
\end{align*}
By construction of $\Tilde{\lambda_{1}}$ and $\Tilde{\lambda_{d_0}}$, we always have that $|\Tilde{\lambda_{d_0}}|/|\Tilde{\lambda_1}| \le \frac{d_0^2}{\phi(d_0)^2}$. 
Thus this term contributes $\ll_\ep R^{2+\ep}x^{2\xi+\ep} \ll x^{\theta_2 + 2\xi+2\ep}$, which 
is negligible since $\xi < 1/4$.

The remaining ``main'' term is given by
\begin{equation}\label{eq_selbsum1}
\frac{x}{4q_3^2W^2\Tilde{\lambda_{1}}^2} \sum_{\substack{d_0,e_0 \\(d_0,e_0,q_3W) = 1}} \Tilde{\l_{d_0}}\Tilde{\l_{e_0}} g([d_0,e_0]) \sum_{\substack{\mathbf d, \mathbf e \in \mathcal D_K\\ (de,d_0e_0) = 1}} \frac{\l_{\mathbf d}\l_{\mathbf e}}{\prod_{i=1}^K [d_i,e_i]}.
\end{equation}
Substituting the formula \eqref{eq:lambda-d-tilde-lambda-d0-diagonalization} for $\lambda_{\mathbf d}\Tilde{\lambda_{d_0}}$ into \eqref{eq_selbsum1}, we get that \eqref{eq_selbsum1} is equal to
\begin{align*}
&= 
\frac{x}{4q_3^2W^2\Tilde{\lambda_{1}}^2}
\sum_{\substack{\mathbf d, \mathbf e \in \mathcal{D}_K \\ d_0,e_0 \\ (d_ie_i, d_je_j) = 1 \; \forall 0\leq i < j \leq k}}
\frac{\mu(d_0 d)\mu(e_0 e)d_0^2 d e_0^2 e g([d_0, e_0])}{[\mathbf{d}, \mathbf{e}]\phi(d_0)\phi(e_0)}
\sum_{\substack{\mathbf d \mid \mathbf r , \mathbf e \mid \mathbf s  \\ d_0\mid r_0, e_0\mid s_0 }}
\frac{y_{\r,r_0} y_{\s,s_0}}{\phi_{\omega^*}(r_0r) \phi_{\omega^*}(s_0s)} 
\\
&= 
\frac{x}{4q_3^2W^2\Tilde{\lambda_{1}}^2}
\sum_{\substack{\mathbf r, \mathbf s \in \mathcal{D}_K \\ r_0,s_0 \\ (r_0s_0, rs) = 1}}
\frac{y_{\r,r_0} y_{\s,s_0}}{\phi_{\omega^*}(r_0r) \phi_{\omega^*}(s_0s)}
\sum_{\substack{\mathbf d \mid \mathbf r,\; \mathbf e \mid \mathbf s \\ d_0 \mid r_0,\; e_0\mid s_0 \\ (d_ie_i, d_je_j) = 1 \; \forall 0\leq i < j \leq k}}
\frac{\mu(d_0 d)\mu(e_0 e)d_0^2 d e_0^2 e g([d_0, e_0])}{[\mathbf{d}, \mathbf{e}]\phi(d_0)\phi(e_0)}.
\end{align*}
The inner sum is multiplicative over $p\mid rsr_0s_0$, where the $p$th factor is given by
$$
\sigma_p(\mathbf r, \mathbf s, r_0, s_0) = 
\begin{cases}
    p-1                   & p\mid r_i,\; p\mid s_i,\; i\geq 1 \\
    \frac{p^2}{p-1} - 1   & p\mid r_0,\; p\mid s_0\;  \\
    - 1                   & p\mid r_i,\; p\mid s_j,\;  i\neq j, \;i,j\geq 0\\
    0                     & p \text{ divides exactly one of }rr_0 \text{ and } ss_0. \\
\end{cases}
$$

The product $\prod_{p\mid rsr_0s_0}\sigma_p(\mathbf r, \mathbf s, r_0, s_0)$ is $0$ unless $rr_0 = ss_0$.
Then using the bound $y_{\r,r_0} y_{\s,s_0} \leq y_{\r,r_0}^2 + y_{\s,s_0}^2$ we see (by symmetry) that \eqref{eq_selbsum1} is
\begin{align}
&\le 
\frac{x}{4q_3^2W^2\Tilde{\lambda_{1}}^2}
\sum_{\r , r_0}
\frac{y_{\r,r_0}^2}{\phi_{\omega^*}^2(rr_0)}
\sum_{\substack{\s, s_0 \\ ss_0 = rr_0}}
\prod_{p \mid rr_0}|\sigma_p(\r,\s,r_0,s_0)| \nonumber
\\
&=\frac{x}{4q_3^2W^2\Tilde{\lambda_{1}}^2}
\sum_{\r , r_0}
\frac{y_{\r,r_0}^2}{\phi_{\omega^*}^2(rr_0)}
\prod_{p\mid r}
\left(K + p - 1\right)
\prod_{p\mid r_0}
\left(K + \frac{p^2}{p-1}- 1\right) \nonumber
\\ 
&\le
\frac{x}{4q_3^2W^2\Tilde{\lambda_{1}}^2}
\sum_{\r , r_0}
y_{\r,r_0}^2
\prod_{p\mid rr_0}
\left(
\frac{K + \frac{p^2}{p-1}- 1}{(p - K -1)^2}
\right) \nonumber
\\
&\ll
\frac{x}{4q_3^2W^2\Tilde{\lambda_{1}}^2}
\sum_{\r , r_0}
\frac{y_{\r,r_0}^2}{\prod_{p\mid rr_0}\left(p + O(K) \right)}.\label{eq-s6b}
\end{align}

In order to estimate this sum, we wish to express $y_{\r,r_0}$ in terms of $y_{\r}$ and $\Tilde{y_{r_0}}$.
This is very similar to the computation done in \cite[Proposition 9.4]{MR3530450-Maynard-dense-clusters}.
Writing $y_{\r,r_0}$ as in \eqref{eq:y-r-r-0-in-terms-of-lambdas} and using the definition of $\lambda_{\mathbf d}$ in \eqref{eq:defn-of-wn-L:defn-of-lambda}
and $\Tilde{\lambda_{d_0}}$ in \eqref{eq-tilde-lambda}, we get
\begin{align*}
y_{\r,r_0} &= 
\mu(r_0r)\phi_{\omega^*}(r_0r) \sum_{\substack{\r|\mathbf d \\ r_0|d_0\\ (d_0,d) = 1}} \frac{\l_{\mathbf d}\Tilde{\l_{d_0}}\phi(d_0)}{dd_0^2}
\\
&= \mu(r_0r)\phi_{\omega^*}(r_0r)
\sum_{r_0\mid d_0}\mu(d_0)
\sum_{d_0 \mid f_0}
\frac{\Tilde{y_{f_0}}}{\phi(f_0)}
\sum_{\substack{\r \mid \mathbf d \\ (d, d_0) = 1}}
\mu(d)\sum_{\mathbf d \mid \mathbf f}\frac{y_{\mathbf f}}{\phi(f)}
\\
&=
\mu(r_0r)\phi_{\omega^*}(r_0r)
\sum_{\substack{f_0, \mathbf f \\ r_0\mid f_0 , \r \mid \mathbf f}}
\frac{y_{\mathbf f} \Tilde{y_{f_0}}}{\phi(f_0)\ph(f)}
\sum_{\substack{d_0, \mathbf d \\ r_0\mid d_0, \r \mid \mathbf d \\ d_0 \mid f_0 , \mathbf d \mid \mathbf f \\ (d,d_0) = 1}}\mu(d)\mu(d_0).
\end{align*}

The inner sum is 0 unless every prime dividing one of $f$ and $f_0$ but not the other is a divisor of $rr_0$; in that case, the inner sum is $\pm 1$.
Thus, using the fact that $y_{\mathbf{r}} \geq y_{\mathbf{f}}$ (since $F$ is decreasing), as well as the fact that $\Tilde{y_{r_0}} \geq \Tilde{y_{f_0}}$, we get that 
\begin{equation*}
y_{\r,r_0} \leq
\phi_{\omega^*}(r_0r)
y_{\mathbf{r}} \Tilde{y_{r_0}}
\sum_{\substack{f_0 \\ r_0 \mid f_0 \\ (f_0, q_3W) = 1}}
\sum_{\substack{ \mathbf f \in \mathcal{D}_K\\ \r \mid \mathbf f \\ f f_0 / (f,f_0)^2 \mid rr_0}}\frac{\mu^2(f_0)}{\phi(f_0)\phi(f)}.
\end{equation*}

Let $f_0 = r_0 f_0' g_0$ and $f_i = r_i f_i' g_i$ for $1\leq i \leq K$ where $f_i' = f_i / (f_i, r r_0)$ is the largest divisor of $f_i$ that is relatively prime to $rr_0$. In particular, $g_0\mid r$ and $g_i\mid r_0$ for $1\leq i \leq K$. Since $f f_0 / (f,f_0)^2 \mid rr_0$, we must have $f_0' = \prod_{i=1}^{K} f_i'$. Thus $y_{\r,r_0}$ is bounded by
\begin{align*}
y_{\r,r_0} &\leq 
\phi_{\omega^*}(r_0r)
y_{\mathbf{r}} \Tilde{y_{r_0}}
\frac{1}{\phi(r_0 r)}
\sum_{\mathbf f'\in \mathcal D_K} \frac{1}{\phi^2(f')}
\sum_{\substack{\mathbf g \in \mathcal D_K \\ g_i\mid r_0 \forall 1\leq i \leq k}}\frac{1}{\phi(g)}
\sum_{\substack{g_0 \mid r}}{\frac{1}{\phi(g_0)}} \\
&\leq 
\phi_{\omega^*}(r_0r)
y_{\mathbf{r}} \Tilde{y_{r_0}}
\prod_{\substack{p > D_0 \\ p\equiv 3 \bmod{4}}}\left( 1 + \frac{K}{(p-1)^2}\right)
\prod_{\substack{p\mid r_0  }}\left( 1 + \frac{K}{(p-1)}\right)
\prod_{\substack{p\mid r }}\left( 1 + \frac{1}{(p-1)}\right).
\end{align*}

The first product is $\ll O_K(1)$. 
By the definition of $\phi_{\omega^*}$, we then have
\begin{equation*}
y_{\r,r_0} \ll
y_{\mathbf{r}} \Tilde{y_{r_0}}
\prod_{p\mid r_0}
\left(1 + \frac{K}{p-1}\right)\left( 1 - \frac{K+1}{p}\right)
\prod_{p\mid r}
\left(1 + \frac{1}{p-1}\right)\left( 1 - \frac{K+1}{p}\right),
\end{equation*}
which in turn implies that $y_{\mathbf r,r_0} \ll y_{\mathbf r} \Tilde{y_{r_0}}$ because both products are $\le 1$.

Plugging this into \eqref{eq-s6b} we get that
\begin{equation}\label{eq-s6bound-2s}
S_6^{(b)}(\nu_1)
\ll
\frac{x}{q_3^2W^2\Tilde{\lambda_{1}}^2}
\Biggr(\sum_{\substack{r_0 \leq x^{\xi} \\ (r_0,W) = 1}}
\frac{\Tilde{y_{r_0}}^2}{\prod_{p\mid r_0}(p+O(K))}\Biggr)
\left(\sum_{\textbf{r}\in\mathcal{D}_K}
\frac{y_{\textbf{r}}^2}{\prod_{p\mid r}(p+O(K))}\right).
\end{equation}

Recalling that $\Tilde{y_{r_0}} = 1$ for $r_0 \le x^{\xi}$ and $(r_0,q_3W) = 1$, we have
$$
\sum_{\substack{r_0 \leq x^{\xi} \\ (r_0,q_3W) = 1}}
\frac{\Tilde{y_{r_0}}^2}{\prod_{p\mid r_0}(p+O(K))}
\ll 
\left(\frac{\xi\log x}{\log D_0}\right)^{1/2}.
$$
We can bound the sum over $\mathbf r$ using \Cref{lem:evaluating-sieve-sums-lemma-8.4-maynard}. 
From the definition of $L_K(F)$ we then get
$$
\sum_{\textbf{r}\in\mathcal{D}_K}
\frac{y_{\textbf{r}}^2}{\prod_{p\mid r}(p+O(K))}
\ll
\left(\frac{e^{-\gamma /2}}{\Gamma(1/2)}\right)^K
\left(\frac{\log R}{\log D_0}\right)^{K/2}
L_K(F).
$$

Using these estimates as well as the bound \eqref{eq-lambda1-bound} on $\lambda_{\mathbf r}$, equation \eqref{eq-s6bound-2s} becomes
\begin{align*}
S_6^{(b)}(\nu_1)
&\ll_K
\frac{x}{q_3^2W^2} \frac{\log D_0}{\xi\log x} \left(\frac{\xi \log x}{\log D_0}\right)^{1/2} \left(\frac{\log R}{\log D_0}\right)^{K/2}L_K(F) \\
&\ll 
\frac{x}{q_3^2W^2}\xi^{-1/2}\theta_2^{-1/2} \left(\frac{\log R}{\log D_0}\right)^{\frac{K-1}{2}} L_K(F),
\end{align*}
as desired.

\end{proof}

\section{Singular series estimates}\label{sec:singular-series-estimates}

We now prove several computational lemmas providing bounds on sums over $\mathbf E$-admissible tuples $\mathcal L = \mathcal L(\mathbf b)$ for $\mathbf b \in \mathcal B$. We begin with an average that appears in the sums over the terms $S_1(\nu_0)$ through $S_5^{(m)}(\nu_0)$ in the proof of Theorem \ref{thm:main-theorem-on-density}, before turning to bounding $S_6^{(b)}(\nu_1)$ on average over different values of $b$.

\subsection{Averaging over \texorpdfstring{$\mathcal B$}{B} and \texorpdfstring{$\nu_0$}{v0}}\label{subsec-Bv_average}

\begin{lemma}\label{lem-averaging_admissible_tuples}
In the notation of Section \ref{sec:statement-of-sieve-and-pf-of-main},
\[
\sum_{\substack{\mathbf b \in \mathcal B \\ \mathcal L = \mathcal L(\mathbf b) \text{ adm.}}} 
\sum_{\substack{\nu_0 \mod W \\ (\ell(\nu_0),W) = 1 \forall \ell \in \mathcal L}} 1
\gg_{K}
 \left(\frac{\eta}{q}\right)^{K-1} \left(\log x\right)^{\frac{K-1}{2}} \left(\frac{\phi(W)}{W}\right)^K W.
\]
\end{lemma}
\begin{proof}
We first consider the number of $\mathbf b \in \mathcal B$ that will produce an admissible tuple $\mathcal L (\mathbf b)$.
If $\mathcal L (\mathbf b)$ is not admissible, then there is some prime $p\leq K$, $p\equiv 3 \bmod 4$, $p\nmid q$ such that $\prod_{\ell \in \mathcal L  (\mathbf b)}\ell(n)$ is always divisible by $p$.
In order to prevent this situation, we can consider only those $\mathbf b$ for which each $b_i$, $i\geq2$ satisfies $q + a_i + qb_i \not \equiv 0,1 \bmod{p}$ for all $p\leq K$, $p\equiv 3 \bmod 4$, $p\nmid 2q$.
Having excluded two congruence classes for each prime $p$, together with the linear form corresponding to $b_1$, the tuple $\mathcal{L}(\mathbf b)$ cannot cover all of the congruence classes mod $p$.

Thus, for each $2\leq i \leq K$, we can choose $b_i$ from a set of size
$$
\frac{\eta}{8q}\sqrt{\log x}
\prod_{\substack{p\mid W \\ 2<p\leq K } }
\left(1 - \frac{2}{p}\right)
\gg_K \frac{\eta}{q}\sqrt{\log x}
$$
while ensuring that the resulting $\mathbf b$ is admissible.
It follows that there are at least $\gg_K (\frac{\eta}{q})^{K-1}(\log x)^{(K-1)/2}$ choices of $\mathbf b \in \mathcal B$ with $\mathcal L(\mathbf b)$ admissible.

For each $\mathbf b$ with $\mathcal L(\mathbf b)$, we now consider the sum over $\nu_0$.
For fixed $\mathbf b$, this is bounded by
$$
\sum_{\substack{\nu_0 \mod W \\ (\ell(\nu_0),W) = 1 \forall \ell \in \mathcal L}} 1
\gg
W \prod_{\substack{p\mid W \\2<p\leq K }}\left(\frac{1}{p}\right)
\prod_{\substack{p\mid W \\ p>K }}\left(1 - \frac{K}{p}\right)
\gg_K W \left(\frac{\phi(W)}{W}\right)^K,
$$
which, along with the number of choices of $\mathbf b$ yielding admissible tuples, completes the proof.
\end{proof}

\subsection{Averaging over \texorpdfstring{$S_6^{(b)}(\nu_1)$}{Averaging S6}}\label{subsec-s6_average}

In this section, we will analyze the sum over $S_6^{(b)}(\nu_1)$ terms appearing in the proof of Theorem \ref{thm:main-theorem-on-density}, and in particular provide the proof of Lemma \ref{lem-average_s6}. 
To begin with, the $S_6^{(b)}(\nu_1)$ sum can be bounded via \Cref{thm:S-i-sums-estimate} by
\begin{multline}\label{eq:S-6-average-bound-before-sing-series}
    \sum_{\substack{\mathbf b \in \mathcal B \\ \mathcal L = \mathcal L(\mathbf b) \text{ adm.}}} \sum_{\substack{\nu_1 \mod W^2 \\ (\ell(\nu_1),W) = 1 \forall \ell \in \mathcal L}} \sum_{\substack{b \le \eta \sqrt{\log x} \\ qn + b \not\in \mathcal L \\ (q\nu_1 + b, W^2) = \square}} S_6^{(b)}(\nu_1) \\
    \ll_K 
    \frac x{q_3^2W^2} \xi^{-1/2}\theta_2^{-1/2} \left(\frac{\log R}{\log D_0}\right)^{\tfrac{K-1}{2}} L_K(F)
    \sum_{\substack{\mathbf b \in \mathcal B \\ \mathcal L = \mathcal L(\mathbf b) \text{ adm.}}}\sum_{\substack{b \le \eta \sqrt{\log x}  \\ qn + b \not\in \mathcal L}} \sum_{\substack{\nu_1 \mod W^2 \\ (\ell(\nu_1),W) = 1 \forall \ell \in \mathcal L\\ (q\nu_1 + b, W^2) = \square}} 1.
\end{multline}
Our next task is estimating the sums over $\mathbf b, b$, and $\nu_1$. The constraints on $\nu_1 \mod W^2$ are multiplicative, so we can understand them separately for each $p\mid W$. For a fixed $p\mid W$, let $\Tilde{N_{p^2}}(\mathcal L,b)$ denote the number of congruence classes $\nu \mod p^2$ such that $p\mid\ell(\nu)$ for some $\ell \in \mathcal L$ or such that $p\mid\ell^b(\nu)$ but $p^2\nmid \ell^b(\nu)$. Then we have, for fixed $\mathcal L(\mathbf b)$ and fixed $b$, that
\begin{equation*}
    \sum_{\substack{\nu_1 \mod W^2 \\ (\ell(\nu_1),W) = 1 \forall \ell \in \mathcal L\\ (q\nu_1 + b, W^2) = \square}} 1 = \prod_{\substack{p \mid W}} (p^2-\Tilde{N_{p^2}}(\mathcal L,b)) =  W^2\left(\frac{\phi(W)}{W}\right)^{K+1} \prod_{p \mid W} \frac{1-\Tilde{N_{p^2}}(\mathcal L,b)/p^2}{(1-1/p)^{K+1}}.
\end{equation*}

The remaining sum over $\mathbf b$ and $b$ is bounded in the following proposition.
\begin{proposition}\label{prop:sing-series-estimate-S-6}
We have
\begin{equation*}
    \sum_{\substack{\mathbf b \in \mathcal B \\ \mathcal L(\mathbf b) \text{adm.}}} \sum_{\substack{b \le \eta \sqrt{\log x} \\ qn + b \not\in \mathcal L}} \prod_{p|W} \frac{1-\Tilde{N_{p^2}}(\mathcal L,b)/p^2}{(1-1/p)^{K+1}} \ll_{K} \frac{(\eta\sqrt{\log x})^{K}}{(8q)^{K-1}},
\end{equation*}
where the implied constant depends only on $K$. 
\end{proposition}

Plugging this estimate into \eqref{eq:S-6-average-bound-before-sing-series}, we have
\begin{align*}
    \sum_{\substack{\mathbf b \in \mathcal B \\ \mathcal L = \mathcal L(\mathbf b) \text{ adm.}}} &\sum_{\substack{\nu_1 \mod W^2 \\ (\ell(\nu_1),W) = 1 \forall \ell \in \mathcal L}} \sum_{\substack{b \le \eta \sqrt{\log x} \\ qn + b \not\in \mathcal L \\ (q\nu_1 + b, W^2) = \square}} S_6^{(b)} \\
    &\ll_K 
    \frac x{q_3^2W^2}  L_K(F)  \xi^{-1/2}\theta_2^{-1/2} \left(\frac{\log R}{\log D_0}\right)^{\tfrac{K-1}{2}}
    W^2 \frac{\phi(W)^{K+1}}{W^{K+1}} \frac{(\eta\sqrt{\log x})^K}{(8q)^{K-1}}.
\end{align*}
Using the fact that 
$$
\frac{\phi(W)}{W} \asymp \frac{q_3}{\phi(q_3)}(\log D_0)^{-1}
$$
and that $\log R = \theta_2 \log x$ we get
\begin{align*}
    \sum_{\substack{\mathbf b \in \mathcal B \\ \mathcal L = \mathcal L(\mathbf b) \text{ adm.}}} &\sum_{\substack{\nu_1 \mod W^2 \\ (\ell(\nu_1),W) = 1 \forall \ell \in \mathcal L}} \sum_{\substack{b \le \eta \sqrt{\log x} \\ qn + b \not\in \mathcal L \\ (q\nu_1 + b, W^2) = \square}} S_6^{(b)} \\
    & \ll_K 
    \xi^{-1/2}
    \theta_2^{\frac{K}{2} - 1}
    L_K(F)
    \left(\frac{q_3}{\phi(q_3)}\right)^{K+1}
    \frac{\eta^K}{q_3^2q^{K-1}}
    \frac{\left(\log x\right)^{K - \frac{1}{2}}}
    {\left(\log D_0\right)^K}
    x.
\end{align*}
which completes the proof of Lemma \ref{lem-average_s6}.

It remains to prove Proposition \ref{prop:sing-series-estimate-S-6}. To do so, we will make use of the following lemma.
\begin{lemma}\label{lem:bounded-by-singular-series}
Let $\Tilde{N_{p^2}}(\mathcal L,b)$ denote the number of congruence classes $\nu \mod p^2$ such that $p \mid \ell(\nu)$ for some $\ell \in \mathcal L$ or such that $p \mid \ell^b(\nu)$ but $p^2\nmid \ell^b(\nu)$. Let $N_p(\mathcal L,b)$ denote the number of congruence classes $\nu \mod p$ such that $p \mid \ell(\nu)$ for some $\ell \in \mathcal L$ or such that $p \mid \ell^b(\nu)$. Then
\begin{equation*}
    \prod_{\substack{p \mid W}} \frac{1-\Tilde{N_{p^2}}(\mathcal L,b)/p^2}{(1-1/p)^{K+1}} \ll_K \prod_{\substack{p \mid W \\ p > K+1}} \frac{1-N_p(\mathcal L,b)/p}{(1-1/p)^{K+1}},
\end{equation*}
where the implied constant depends only on $K$.
\end{lemma}
\begin{proof}
By definition, $\Tilde{N_{p^2}}(\mathcal L,b)$ almost consists of all elements of a certain set of congruence classes modulo $p$ when lifted to $\mathbb Z/p^2\mathbb Z$, with the possible exception of one congruence class $\nu$ modulo $p^2$ such that $\ell^b(\nu) \equiv 0 \mod p^2$.
In particular, this implies that
\begin{equation*}
    \Tilde{N_{p^2}}(\mathcal L,b) = pN_p(\mathcal L,b) - E,
\end{equation*}
where $E$ is either $0$ or $1$. Thus 
\begin{equation*}
    \prod_{\substack{p \mid W}} \frac{1-\Tilde{N_{p^2}}(\mathcal L,b)/p^2}{(1-1/p)^{K+1}} \le \prod_{\substack{p \mid W}} \frac{1-N_p(\mathcal L,b)/p + 1/p^2}{(1-1/p)^{K+1}},
\end{equation*}
since the numerator of each factor in the product is either unchanged or has increased. 

We can then rewrite the right-hand side as
\begin{align*}
&\prod_{\substack{p|W}} \frac{1-N_p(\mathcal L,b)/p + 1/p^2}{(1-1/p)^{K+1}} = \prod_{\substack{p|W \\ p \le K+1}} \frac{1-N_p(\mathcal L,b)/p + 1/p^2}{(1-1/p)^{K+1}} \prod_{\substack{p|W \\ p > K+1}} \frac{1-N_p(\mathcal L,b)/p + 1/p^2}{(1-1/p)^{K+1}} \\
&\qquad \qquad \qquad
\le \prod_{\substack{p \le K+1}} \frac{1-1/p}{(1-1/p)^{K+1}}\prod_{\substack{p|W \\ p > K+1}} \frac{(1-N_p(\mathcal L,b)/p)}{(1-1/p)^{K+1}} \left(1 + \frac 1{p(p-N_p(\mathcal L,b))} \right)\\
&\qquad \qquad \qquad
\ll_K \prod_{\substack{p|W \\ p > K+1}} \frac{(1-N_p(\mathcal L,b)/p)}{(1-1/p)^{K+1}} \prod_{\substack{p|W \\ p > K+1}} \left(1+\frac 1{p(p-K-1)}\right).
\end{align*}
The second Euler product converges to a constant dependent only on $K$ when extended over all primes $p > K+1$, which completes the proof.
\end{proof}

We are now ready to prove Proposition \ref{prop:sing-series-estimate-S-6}. This estimate is an analog of Gallagher's result \cite{MR0409385-gallagher} that the average value of the singular series constants appearing in the Hardy--Littlewood $k$-tuples conjecture is $1$. Our proof will closely follow Gallagher's argument.

\begin{proof}[Proof of Proposition \ref{prop:sing-series-estimate-S-6}]
We begin by applying Lemma \ref{lem:bounded-by-singular-series} to bound the left-hand side by
\begin{equation*}
   \ll_K \sum_{\substack{\mathbf b \in \mathcal B \\ \mathcal L(\mathbf b) \text{adm.}}} \sum_{\substack{b \le \eta \sqrt{\log x} \\ qn + b \not\in \mathcal L}} \prod_{\substack{p|W \\ p > K +1}} \frac{1-N_p(\mathcal L,b)/p}{(1-1/p)^{K+1}},
\end{equation*}
where $N_p(\mathcal L,b)$ is the number of congruence classes $\nu \mod p$ such that $p|\ell(\nu)$ for some $\ell \in \mathcal L$ or such that $p|\ell^b(\nu)$. 

Let $\Delta(\mathcal L, b)$ denote the product
\begin{equation*}
    \Delta(\mathcal L, b) := \prod_{1 \le i_1 < i_2 \le K}(q(b_{i_2}-b_{i_1}) + a_{i_2}-a_{i_1}) \prod_{1 \le i \le K} (b-qb_i - a_i).
\end{equation*}
Thus $1 \le N_p(\mathcal L,b) \le K+1$, with equality on the right if and only if $p \nmid \Delta(\mathcal L,b)$. Define $a(p,N_p)$ via
\begin{equation*}
\frac{1-N_p/p}{(1-1/p)^{K+1}} = 1 + a(p,N_p),
\end{equation*}
and for squarefree $r$ define $a_{\mathcal L,b}(r)$ multiplicatively via
$
    a_{\mathcal L,b}(r) = \prod_{p|r} a(p,N_p(\mathcal L,b)),
$
so that 
\begin{equation*}
    \prod_{p|W}\frac{1-N_p(\mathcal L,b)/p}{(1-1/p)^{K+1}} = \sum_{r|W} a_{\mathcal L,b}(r).
\end{equation*}
By the same reasoning as in Gallagher's proof of equation (3) in \cite{MR0409385-gallagher}, for a constant $x$ to be fixed later and for all $\ep > 0$, we have
\begin{multline}\label{eq:gallagher-eq-9}
    \sum_{\substack{\mathbf b \in \mathcal B \\ \mathcal L(\mathbf b) \text{ adm.}}} \sum_{\substack{b \le \eta \sqrt{\log x} \\ qn + b \not\in \mathcal L}} \prod_{\substack{p|W \\ p > K +1}} \frac{1-N_p(\mathcal L,b)/p}{(1-1/p)^{K+1}} \\
    = \sum_{r \le x} \sum_{\substack{\mathbf b \in \mathcal B \\ \mathcal L(\mathbf b) \text{ adm.}}} \sum_{\substack{b \le \eta \sqrt{\log x} \\ qn + b \not\in \mathcal L}} a_{\mathcal L,b}(r) + O_{K,\ep}(\eta^K(\log x)^{K/2}(x\eta\log x)^{\ep}/x).
\end{multline}
The inner sums over $\mathbf b \in \mathcal B$ and $b$ are equal to
\begin{equation*}
    \sum_{(N_p)_p} \prod_{\substack{p|r \\ p>K+1}} a(p,N_p) \left\{ \sideset{}{'}\sum 1 + O((\eta\sqrt{\log x})^{K-1})\right\},
\end{equation*}
where each $(N_p)_p$ is a vector with positive integer entries for each prime $p|r$ with $p>K+1$, and where $\sideset{}{'}\sum$ denotes the number of ways to choose values $b_2, \dots, b_K$ and $b$, not necessarily distinct, such that each $b_i \equiv 3 \mod 4$, such that $1 < qb_i \le \frac{\eta}{2}\sqrt{\log x}$ for $2 \le i \le jM$, such that $\frac{\eta}{2}< qb_i \le \eta\sqrt{\log x}$ for $jM+1 \le i \le K$, such that $b \le \eta\sqrt{\log x}$, and most crucially, such that $b_1, \dots, b_K, b$ occupy precisely $N_p$ congruence classes modulo $p$ for each $p|r$. Recall that $b_1 = 3$ is fixed for all $\mathbf b \in \mathcal B$.

By the Chinese remainder theorem, for $r \le \eta\sqrt{\log x}$,
\begin{equation*}
    \sideset{}{'}\sum = \left\{\left(\frac{\eta\sqrt{\log x}}{8qr}\right)^{K-1}\left(\frac{\eta\sqrt{\log x}}r\right) + O\left(\frac{\eta\sqrt{\log x}}{r}\right)^{K-1}\right\} \times \prod_{p|r} \binom{p-1}{N_p-1}\sigma(K,N_p),
\end{equation*}
where $\sigma(K,N_p)$ denotes the number of surjective maps from $\{1,\dots, K\}$ onto $\{1,\dots,N_p\}$. 
Thus the inner sum is
\begin{equation*}
\left(\frac{\eta\sqrt{\log x}}{8qr}\right)^{K-1}\left(\frac{\eta\sqrt{\log x}}{r}\right)A(r) + O\left(\left(\frac{\eta\sqrt{\log x}}{r}\right)^{K-1}B(r)\right) + O\left((\eta\sqrt{\log x})^{K-1}C(r)\right),
\end{equation*}
where
\begin{align*}
    A(r) &= \sum_{(N_p)_{p|r}} \prod_{p|r} a(p,N_p)\binom{p}{N_p} \sigma(K,N_p), \\
    B(r) &= \sum_{(N_p)_{p|r}} \prod_{p|r} |a(p,N_p)|\binom{p}{N_p} \sigma(K,N_p), \text{ and} \\
    C(r) &= \sum_{(N_p)_{p|r}} \prod_{p|r} |a(p,N_p)|.
\end{align*}

One can show via a combinatorial argument (identical to the one performed in \cite{MR0409385-gallagher}) that $A(r) = 0$ whenever $r> 1$. Also by the same arguments as in \cite{MR0409385-gallagher}, we have $B(r) \le C^{\omega(r)}\frac{r^K}{\phi(r)}$ and $C(r) \le C^{\omega(r)}\frac{r}{\phi(r)}$ for a suitable constant $C$. 

Altogether we get that
\begin{align*}
    \sum_{\substack{\mathbf b \in \mathcal B \\ \mathcal L(\mathbf b) \text{ adm.}}} &\sum_{\substack{b \le \eta \sqrt{\log x} \\ qn + b \not\in \mathcal L}} \prod_{\substack{p|W \\ p > K +1}} \frac{1-N_p(\mathcal L,b)/p}{(1-1/p)^{K+1}} \\
    = &\left(\frac{\eta\sqrt{\log x}}{8q}\right)^{K-1}(\eta\sqrt{\log x}) + O\left((\eta\sqrt{\log x})^{K-1} \sum_{r \le x} \frac{C^{\omega(r)}r}{\phi(r)}\right) \\
    &+ O(\eta^K(\log x)^{K/2}(x\eta\log x)^{\ep}/x) \\
    = &\frac{(\eta\sqrt{\log x})^{K}}{(8q)^{K-1}} +O_{K,\ep,q}((\eta\sqrt{\log x})^{K-1/2 + \ep}),
\end{align*}
choosing $x = (\eta\sqrt{\log x)})^{1/2}.$ This completes the proof.
\end{proof}

\textbf{Competing Interests:} The authors have no competing interests to declare.
\bibliographystyle{amsplain}
\bibliography{bibliography}

\printindex

\end{document}